\documentclass[11pt, a4paper]{article}

\usepackage{microtype}

\usepackage{thmtools, thm-restate}

\usepackage[utf8]{inputenc} 
\usepackage[T1]{fontenc} 
\usepackage{geometry} 
\usepackage[%
    colorlinks      = true,
  linktocpage     = true,
  linkcolor       = blue,
  citecolor       = blue,
  urlcolor        = blue,
  breaklinks =true,
  hypertexnames=false,
]{hyperref}

\usepackage{setspace}
\usepackage[affil-it]{authblk}

\usepackage{amsmath,  amssymb,slashed,url,bm}
\usepackage{textgreek,upgreek}

\usepackage{graphicx}
\usepackage{epstopdf}
 
\def\be{\begin{equation}}
\def\ee{\end{equation}}

\def\hat{\widehat}

\def\tilde{\widetilde}

\def\P{{\mathcal P}}

\def\tilde{\widetilde}

\def\max{{\mathrm{max}}}

\usepackage{amsmath, amsthm, amsfonts, amssymb, amsbsy, bigstrut, graphicx, enumerate,  upref, longtable, comment, booktabs, array, caption, subcaption}

\usepackage[numbers, sort&compress]{natbib}

\newcommand{\ep}{\epsilon}

\newcommand{\tf}{\tilde{f}}

\usepackage{aliascnt}

\newtheorem{thm}{Theorem}[subsection]

\newaliascnt{lmm}{thm}
\newtheorem{lmm}[lmm]{Lemma}
\aliascntresetthe{lmm}

\newtheorem{cor}[thm]{Corollary}
\newtheorem{prop}[thm]{Proposition}

\theoremstyle{definition}

\newcommand{\md}{\mathcal{D}}
\newcommand{\mf}{\mathcal{F}}

\newcommand{\cp}{\mathcal{P}}

\numberwithin{equation}{subsection}

\renewcommand{\Re}{\operatorname{Re}}

\newcommand{\slc}{\mathrm{SL}(2,\C)}

\usepackage[utf8]{inputenc}
\usepackage{amsmath}
\usepackage{amsfonts}
\usepackage{bbm}
\usepackage{mathtools}
\usepackage{tikz}
\usetikzlibrary{decorations.pathreplacing}
\usepackage{amsthm}
\usepackage[english]{babel}
\usepackage{fancyhdr}
\usepackage{amssymb}
\usepackage{enumitem}
\usepackage{qtree}
\usepackage{mathrsfs}

\newcommand{\Z}{\mathbb{Z}}

\newcommand{\R}{\mathbb{R}}

\newcommand{\C}{\mathbb{C}}

\newcommand{\E}{\mathbb{E}}

\renewcommand{\P}{\mathbb{P}}

\renewcommand{\tilde}{\widetilde}
\renewcommand{\hat}{\widehat}

\renewcommand{\Im}{\operatorname{Im}}

\newcommand{\normord}[1]{:\mathrel{#1}:}

\newcommand{\vertiii}[1]{{\vert\kern-0.25ex\vert\kern-0.25ex\vert #1 
    \vert\kern-0.25ex\vert\kern-0.25ex\vert}}

\newcommand{\tmu}{\tilde{\mu}}
\newcommand{\talpha}{\tilde{\alpha}}

\begin{document}

\title{\bf Rigorous results for timelike Liouville field theory}
\author{Sourav Chatterjee\thanks{Department of Statistics, Stanford University, 390 Jane Stanford Way, Stanford, CA 94305, USA. Email: \href{mailto:souravc@stanford.edu}{\tt souravc@stanford.edu}.}}
\affil{Stanford University}


\maketitle


\begin{abstract}
Liouville field theory has long been a cornerstone of two-dimensional quantum field theory and quantum gravity, which has attracted much recent attention in the mathematics literature. Timelike Liouville field theory is a version of Liouville field theory where the kinetic term in the action appears with a negative sign, which makes it closer to a theory of quantum gravity than ordinary (spacelike) Liouville field theory. Making sense of this `wrong sign' requires a theory of Gaussian random variables with negative variance. Such a theory is developed in this paper, and is used to prove the timelike DOZZ formula for the $3$-point correlation function when the parameters satisfy the so-called `charge neutrality condition'. Expressions are derived also for the $k$-point correlation functions for all $k\ge 3$, and it is shown that these functions approach the correct semiclassical limits as the coupling constant is sent to zero.
\newline
\newline
\noindent {\scriptsize {\it Key words and phrases.} Liouville field theory, quantum gravity, conformal field theory, DOZZ formula, semiclassical limit.}
\newline
\noindent {\scriptsize {\it 2020 Mathematics Subject Classification.} 81T40, 17B69, 81T20.}
\end{abstract}



\tableofcontents


\section{Introduction}
Liouville field theory was introduced by \citet{polyakov81} in 1981 in the context of bosonic string theory and 2D quantum gravity. In a nutshell, it is a 2D conformal field theory whose classical equation of motion is a generalization of Liouville's equation, which is a PDE describing the evolution of a Riemannian metric on $\R^2$. Liouville field theory has found applications in various areas of theoretical physics, including string theory \cite{polyakov81}, three-dimensional general relativity \cite{coussaertetal95}, string theory in anti-de Sitter space \cite{ribaultteschner05}, and supersymmetric gauge theory \cite{aldayetal10}. Recent years have seen an explosion of activity in the mathematical literature on proving the many tantalizing conjectures made by physicists in the early years of Liouville field theory. This includes the connection with Gaussian multiplicative chaos and the proof of the KPZ formula by \citet{duplantiersheffield11}, the proof of the DOZZ formula by \citet*{kupiainenetal20}, the connection with the Brownian map by \citet{millersheffield20}, existence and uniqueness of the Liouville metric by~\citet{dingetal20} and \citet{gwynnemiller21}, and many other pathbreaking works~\cite{gwynneetal19, gwynneetal20, aruetal21, davidetal16, huangetal18, bourgadefalconet22, gwynnepfeffer22, anggwynne21, milleretal22, borgaetal22}. We refer to \cite{chatterjeewitten25, berestyckipowell24} for surveys.


Liouville field theory has a parameter $b>0$ known as the `Liouville coupling constant'. When this parameter is replaced by $ib$, where $i=\sqrt{-1}$, we obtain `timelike' or `imaginary' Liouville field theory (in contrast with the usual Liouville field theory, which is sometimes called `spacelike' Liouville field theory). Timelike Liouville theory has applications in quantum cosmology~\cite{freivogeletal06, sekinosusskind09, harlowsusskind10}, tachyon condensation~\cite{stromingertakayanagi03}, and other areas of theoretical physics. Timelike Liouville theory also has deep and unexpected connections in probability theory and statistical mechanics. \citet{delfinoviti10} conjectured a formula for the $3$-point connectivity probabilities in 2D critical percolation in terms of the $3$-point correlation function of timelike Liouville theory. \citet*{ikhlefetal16} conjectured a similar formula for the nesting loops statistics of conformal loop ensembles. Both conjectures were recently proved by \citet*{angetal21}.

Replacing $b$ by $ib$ and replacing the Liouville field $\phi$ by $i\phi$ have the cumulative effect of reversing the sign in front of the kinetic term in the Liouville action. The wrong sign in front of the kinetic term is a signature of models of quantum gravity~\cite{hawking78}. For this reason, timelike Liouville theory is closer to a theory of 2D quantum gravity than ordinary (spacelike) Liouville theory~\cite{bautistaetal19}. From a mathematical perspective, the wrong sign presents an unusual challenge. While spacelike Liouville theory has been made rigorous using tools from probability theory, converting those proofs to the timelike case (or indeed, any `true' model of quantum gravity) would require a theory of Gaussian random variables with negative variance. One of the contributions of this paper is the development of such a theory. 

The key object that allows explicit computations in spacelike Liouville field theory is the $3$-point correlation function, given by the celebrated DOZZ formula (named in honor of those who first computed it, \citet{dornotto94} and \citet{zamolodchikovzamolodchikov96}). The DOZZ formula was proved rigorously by \citet*{kupiainenetal20} in 2016, nearly twenty years after it was discovered in the physics literature. Surprisingly, although timelike Liouville theory is formally obtained by replacing $b$ by $ib$, it turns out that a straightforward analytic continuation to replace $b$ by $ib$ in the DOZZ formula does not yield the correct $3$-point function for timelike Liouville theory~\cite{zamolodchikov05}. An explicit expression, called the timelike DOZZ formula, was proposed by \citet{schomerus03}, \citet{zamolodchikov05}, and \citet{kostovpetkova06, kostovpetkova07, kostovpetkova07a}, with various heuristic justifications. \citet{schomerus03}, for instance, obtained the formula assuming that certain recursion relations discovered by \citet{teschner95} for spacelike Liouville theory (and rigorously proved by \citet{kupiainenetal20}) continue to hold for the timelike theory. Later, \citet*{harlowetal11} argued that the formula may be obtained by changing the cycle of integration for the path integral, and \citet{giribet12} argued that it can obtained from a Coulomb gas representation. 

The quest for a rigorous construction of timelike Liouville theory was initiated by \citet*{guillarmouetal23} for a compactified version of this theory, where the Liouville field at a point is allowed to take values in a compact subset of the real line instead of the entire real line. In the present paper, we construct the original non-compact version of the theory in a subset of the parameter space (satisfying the so-called `charge neutrality condition') using our theory of wrong sign Gaussian distributions, and prove the validity of the timelike DOZZ formula in this region. In addition to this, we also give expressions for $k$-point correlation functions for all $k\ge 3$. Interestingly, for $k=3$, the formulas so obtained are very similar to the ones obtained in \cite{guillarmouetal23} for the compact theory under similar conditions on the parameters. The two models are, however, different in other aspects; for example, the compact theory has a discrete spectrum (as noted in \cite{guillarmouetal23}) and the non-compact theory is likely to have continuous spectrum~\cite{ribaultsantachiara15}.  Lastly, we show that the correlation functions approach the correct semiclassical limits for timelike Liouville theory as $b\to 0$.

In the remainder of this section, we give the heuristic definitions of timelike Liouville theory and its correlation functions, as they appear in physics, together with a summary of the main results.

\subsection{The action for timelike Liouville theory}
Let $g(z)|dz|^2$ be the round metric on $\C$, given by 
\[
g(z) := \frac{4}{(1+|z|^2)^2}. 
\]
For a field $\phi:\C \to\R$, the action for timelike Liouville theory with background metric $g$ is heuristically defined as
\begin{align}\label{idef}
I(\phi) = \frac{1}{4\pi}\int_{\C}  (\phi(z)\Delta_g\phi(z) + 2Q \phi(z) + 4\pi \mu \normord{e^{2b\phi(z)}})g(z) d^2z,
\end{align}
where the various terms are as follows:
\begin{itemize}
\item $b$ is a positive constant known as the Liouville coupling constant.
\item $\mu$ is another positive constant known as the cosmological constant.
\item $Q := b - \frac{1}{b}$. 
\item $\Delta_g $ is the Laplacian operator for the metric $g$, given by 
\[
\Delta_g \phi (z) = g(z)^{-1}\Delta\phi(z),
\]
where $\Delta \phi$ is the usual Laplacian of $\phi$, viewing $\phi$ as a function of two real variables.
\item $\int_{\C} \ldots d^2z$ denotes integration with respect to Lebesgue measure on $\C$.
\item $\normord{e^{2b\phi(z)}}$ denotes the normal ordered version of the function $e^{2b\phi(z)}$, defined heuristically~as 
\begin{align*}
\normord{e^{2b\phi(z)}} &= e^{2b\phi(z)+2b^2G_g(z,z)}, 
\end{align*}
where $G_g$ is the Green's function for $g$, defined as the inverse of $-\frac{1}{2\pi}\Delta_g$ on functions that integrate to zero with respect to $g(z)d^2z$. The Green's function $G_g$ is explicitly given by
\[
G_g(z,z') = \sum_{n=1}^\infty \frac{2\pi }{\lambda_n} f_n(z)f_n(z'),
\]
where $\{f_n\}_{n\ge 0}$ are eigenfunctions of $-\Delta_g$ with eigenvalues $0=\lambda_0<\lambda_1<\cdots$. Note that the above definition of $\normord{e^{2\phi(z)}}$ is not mathematically well-defined because $G_g(z,z)= \infty$ for all $z$. 
\end{itemize}
Readers familiar with spacelike Liouville field theory may recognize that the above action is obtained by replacing $b$ with $ib$ and $\phi$ with $i\phi$ in the action for the spacelike theory. 

Heuristically, timelike Liouville field theory defines a `measure' on the space of fields that has density $e^{-I(\phi)}$ with respect to `Lebesgue measure' on the space of fields. A major problem with making sense of the above `measure' is that the action $I$ is unbounded below. To see this, observe that by integration by parts, the first term in the action (which is called the `kinetic term') is given by 
\[
\int_{\C} \phi(z) \Delta_g \phi(z) g(z) d^2z = - \int_{\C} |\nabla \phi(z)|^2 d^2z,
\]
where $\nabla \phi$ is the gradient of $\phi$ (viewing $\phi$ as a function of two real variables) and $|\nabla \phi(z)|$ is the Euclidean norm of $\nabla \phi(z)$. This means that we can make $\phi$ more and more wiggly to make this term diverge to $-\infty$, while preserving the remaining terms in the action within finite bounds. This problem does not occur in spacelike Liouville theory, because the kinetic term $\int |\nabla \phi|^2$ appears with a plus sign. The appearance of the kinetic term with the `wrong' (i.e., negative) sign in the action is a common feature of models in quantum gravity. Its most consequential appearance is in the Einstein--Hilbert action for Einstein gravity~\cite{gibbonshawking77}, which is one of the roadblocks to quantizing Einstein gravity.

\subsection{Correlation functions of timelike Liouville theory}
Once we have some kind of sense of a measure with density $e^{-I(\phi)}$ on the space of fields, we would then like to understand the behavior of a `random' field $\phi$ `drawn' from this measure, in the sense of drawing a random field from a probability distribution. The main role of such a random field $\phi$ is that it defines a random metric $\normord{e^{2b\phi(z)}} g(z) |dz|^2$ on $\C$. Any theory of quantum gravity is a theory of a random metric that fluctuates around the critical points of the action, where the critical points give the classical equations of motion. For example, the critical points of the Einstein--Hilbert action are the metrics on $\R^4$ that satisfy Einstein's equation of general relativity. 

One way to understand the behavior of this random field is to take expectations of observables  like $e^{2\alpha \phi(z)}$ for $\alpha, z \in \C$. More generally, we can take expectations of products of such observables. To get finite results, we must normalize $e^{2\alpha \phi(z)}$ appropriately. This normalization yields the so-called `vertex operators'. Take any $k\ge 1$, and let $z_1,\ldots,z_k$ be distinct points in $\C$. Let $\alpha_1,\ldots,\alpha_k$ be arbitrary complex numbers. The $k$-point correlation function of timelike Liouville theory is heuristically is defined~as
\begin{align}\label{favg}
C(\alpha_1,\ldots,\alpha_k; z_1,\ldots,z_k; b; \mu) :=\int \biggl(\prod_{j=1}^k V_\phi(\alpha_j, z_j) \biggr)e^{-I(\phi)}\mathcal{D}\phi,
\end{align}
where $\int \ldots \mathcal{D}\phi$ denotes integration with respect to a hypothetical `Lebesgue measure' on the space of functions from $\C$ into $\R$, $I$ is the action defined in equation \eqref{idef}, and $V_\phi(\alpha,z)$ is the `vertex operator' 
\begin{align*}
V_\phi(\alpha, z) &:= e^{\chi \alpha(b-\alpha)} g(z)^{-\Delta_\alpha}\normord{e^{2\alpha \phi(z)}},
\end{align*}
where $\Delta_\alpha := \alpha(Q-\alpha)$ and $\chi := \ln 4 - 1$. The number $\Delta_\alpha$ is called the `conformal weight' of the vertex operator, for reasons related to conformal field theory.

Let us now express the correlation function differently, as a path integral over functions on the unit sphere $S^2$. Let $a$ denote the area measure on $S^2$. Let $e_3 := (0,0,1)$ denote the `north pole' of $S^2$, and let $\sigma:S^2 \setminus\{e_3\} \to \C$ denote the stereographic projection
\[
\sigma (x,y,z) := \frac{x + i y}{1-z}. 
\]
Integrals on $\C$ with respect to the measure $g(z)d^2z$ can be expressed as integrals over $S^2$ with respect to the area measure via the stereographic projection: For any $f:\C \to \R$ such that both sides below are absolutely integrable, we have
\[
\int_{\C} f(z) g(z) d^2z = \int_{S^2} f(\sigma(x)) da(x). 
\]
Thus, the action defined in equation \eqref{idef} can be expressed as 
\[
I(\phi) = \tilde{I}(\tilde{\phi}) := \frac{1}{4\pi}\int_{S^2}  (\tilde{\phi}(x)\Delta_{S^2}\tilde{\phi}(x) + 2Q \tilde{\phi}(x) + 4\pi \mu \normord{e^{2b\tilde{\phi}(x)}})da(x),
\]
where $\tilde{\phi} := \phi \circ \sigma$, $\Delta_{S^2}$ is the spherical Laplacian, and 
\[
\normord{e^{2b\tilde{\phi}(x)}} = e^{2b\tilde{\phi}(x)+2b^2G(x,x)}, 
\]
with 
\begin{align}\label{greendef}
G(x,y) := G_g(\sigma(x), \sigma(y)).
\end{align}
We note, for later use, that $G$ is the Green's function for the spherical Laplacian, meaning that $G$ is the inverse of $-\frac{1}{2\pi}\Delta_{S^2}$ on functions that integrate to zero with respect to the area measure. Let us also define, for $\alpha \in \C$ and $x\in S^2$,  the vertex operator
\[
\tilde{V}_{\tilde{\phi}}(\alpha, x) := e^{\chi\alpha(b-\alpha)} g(\sigma(x))^{-\Delta_\alpha}\normord{e^{2\alpha \tilde{\phi}(x)}}.
\]
Then the correlation function defined in equation \eqref{favg} can be rewritten as
\begin{align}\label{favg2}
\tilde{C}(\alpha_1,\ldots,\alpha_k; x_1,\ldots,x_k; b; \mu) :=\int \biggl(\prod_{j=1}^k \tilde{V}_{\tilde{\phi}}(\alpha_j, x_j) \biggr)e^{-\tilde{I}(\tilde{\phi})}\mathcal{D}\tilde{\phi},
\end{align}
where $\mathcal{D}\tilde{\phi}$ denotes `Lebesgue measure' on the space of real-valued function on $S^2$.

The correlation functions $C$ and $\tilde{C}$ are equivalent. We will use both of them in this paper, depending on which one is more convenient to use in a given situation. The timelike DOZZ formula gives an explicit formula for the $3$-point correlation function. One of our main results is a rigorous proof of this formula in a certain regime of parameters. In the next subsection, we present brief summaries of this and the other main results of the paper.

\subsection{Main results}\label{mainresultssec}
Our first main result, proved in Subsection \ref{planecorsec}, is a formula for the $k$-point correlation function. This is similar to the Coulomb gas expression for the $3$-point function of spacelike Liouville theory derived by \citet{goulianli91} and the analogous discussion for timelike Liouville theory in \citet{kostovpetkova06}. We prove this under the conditions that $(Q-\sum_{j=1}^k \alpha_j)/b$ is a positive integer and $\Re(\alpha_j)>-1/2b$ for each $j$. The first condition is sometimes called the `charge neutrality condition'. In the case $k=3$, \citet[Theorem 2.1]{guillarmouetal23} obtain a similar expression for the $3$-point correlations in their compactified model under the similar conditions on the parameters.

\begin{restatable}[Formula for $k$-point correlations]{thm}{intthmsecond}\label{intthmv2}
Suppose that $k\ge 3$, $\Re(\alpha_j) > -1/2b$ for each $j$, and the parameter $w := (Q-\sum_{j=1}^k \alpha_j)/b$ is a positive integer. Let $x_1,\ldots,x_k$ be distinct points on $S^2$, none of which are the north pole $e_3 = (0,0,1)$. Let $z_1,\ldots,z_k$ be their stereographic projections on $\C$. Then 
\begin{align*}
&C(\alpha_1,\ldots,\alpha_k; z_1,\ldots,z_k;b;\mu) = \tilde{C}(\alpha_1,\ldots,\alpha_k; x_1,\ldots,x_k;b;\mu)\\
&= \frac{e^{-i\pi w}\mu^{w}}{w!} (4/e)^{1-1/b^2}\prod_{1\le j<j'\le k} |z_j - z_{j'}|^{4\alpha_j \alpha_{j'}}\\
&\qquad \qquad \cdot  \int_{\C^w} \biggl(\prod_{j=1}^k \prod_{l=1}^{w} |z_j-t_l|^{4b\alpha_j}\biggr)\biggl(\prod_{1\le l<l'\le w}|t_l-t_{l'}|^{4b^2} \biggr) d^2t_1\cdots d^2t_w.
\end{align*}
\end{restatable}
Note that the above result is for $k\ge 3$. In particular, it cannot deal with $0$-point, $1$-point, or $2$-point functions. It is not clear how to calculate, say, the $2$-point function from the path integral. The $2$-point function is conjectured to be a distribution rather than a function~\cite[Equation (3.7)]{collieretal24} for certain special values of $\alpha_1,\alpha_2$. The techniques used for proving Theorem \ref{intthmv2} do not seem to yield such a result for the $2$-point function.

Our next main result is a rigorous statement of the timelike DOZZ formula, proved using the formula from the above theorem and a series of calculations using the complex Selberg integral formula of \citet{dotsenkofateev85} and \citet{aomoto87}, following ideas from \citet{giribet12}. To state this result, we need some preparation. The following special function was introduced by \citet{dornotto94}:
\begin{align*}
\Upsilon_b(z) := \exp\biggl(\int_0^\infty \frac{1}{\tau}\biggl(\biggl(\frac{b}{2} + \frac{1}{2b} - z\biggr)^2 e^{-\tau} - \frac{\sinh^2((\frac{b}{2}+\frac{1}{2b} - z)\frac{\tau}{2})}{\sinh(\frac{b\tau}{2}) \sinh(\frac{\tau}{2b})}\biggr)d\tau \biggr)
\end{align*}
on the strip $\{z\in  \C: 0< \Re(z)<b + \frac{1}{b}\}$ and continued analytically to the whole plane. Let $\gamma(z) := \Gamma(z)/\Gamma(1-z)$, where $\Gamma$ is the classical Gamma function. The following theorem gives a formula for the $3$-point correlation which is the same as the one displayed in \citet{harlowetal11} (after the notational changes $\hat{Q} \to -Q$, $\hat{\alpha}_j \to -\alpha_j$ and $\hat{b} \to b$), as well the ones appearing in the original proposals of \citet{schomerus03}, \citet{zamolodchikov05}, and \citet{kostovpetkova06, kostovpetkova07, kostovpetkova07a}. The only difference is that there is an additional factor depending only on $b$; but that is not a problem since the correlation function is supposed to be unique only up to a $b$-dependent factor. This theorem arises from a combination of Theorem \ref{threepoint} in Subsection \ref{dozzsec}, Corollary \ref{sl2ccor} in Subsection \ref{structuresec}, and Theorem \ref{analyticext} in Subsection \ref{analyticsec}. It gives a rigorous proof of the timelike DOZZ formula in a subset of the parameter space. The formula is supposed to be valid everywhere. As of the time of writing this, it is unclear how to extend the arguments of this paper to the full parameter space.
\begin{thm}[Timelike DOZZ formula]\label{tdozzthm}
Let $\alpha_1,\alpha_2,\alpha_3$ be complex numbers such that $w = (Q-\sum_{j=1}^3\alpha_j)/b$ is a positive integer less than $1+(2b^2)^{-1}$, and  $\Upsilon_b(2\alpha_j+1/b)\ne 0$ for $j=1,2,3$. Take any distinct $z_1,z_2,z_3\in \C$. For $1\le j<k\le 3$, define 
\[
z_{jk} := |z_j - z_k|, \ \ \ \Delta_{jk}:= 2\Delta_{\alpha_j} + 2\Delta_{\alpha_k} -\sum_{l=1}^3\Delta_{\alpha_l}.
\]
Then
\begin{align*}
C(\alpha_1,\alpha_2,\alpha_3; z_1,z_2,z_3;b;\mu) &= e^{-i\pi w} (-\pi \mu\gamma(-b^2))^w (4/e)^{1-1/b^2}b^{2b^2w + 2w}\\
&\qquad \cdot \frac{\Upsilon_b(bw+b)}{\Upsilon_b(b)} \prod_{j=1}^3 \frac{\Upsilon_b(2\alpha_j+bw + 1/b)}{\Upsilon_b(2\alpha_j+1/b)} \prod_{1\le j<k\le 3} |z_{jk}|^{2\Delta_{jk}}.
\end{align*}
The formula also holds if $w$ is any positive integer and $\alpha_1,\alpha_2,\alpha_3$ have real parts greater than $-1/2b$.
\end{thm}
One of the great utilities of the DOZZ formula for the spacelike theory and the timelike DOZZ formula for the timelike theory is that they identify the poles of the $3$-point function. The poles of the correlation functions carry great physical significance for conformal field theories; see \cite[Chapter 6]{difrancescoetal12}. The region covered by Theorem \ref{tdozzthm} does not contain all the poles of the $3$-point function, but it contains a nontrivial subset of those. For a discussion, see Subsection \ref{polesec}. 

The third and the fourth main results of the paper are about the semiclassical limit of timelike Liouville field theory after insertion of heavy vertex operators. What this means is that we look at the limit $b\to 0$, while simultaneously scaling the $\alpha_j$'s and $\mu$ as $\alpha_j = \talpha_j/b$ and $\mu = \tmu/b^2$, where the $\talpha_j$'s and $\tmu$ are fixed real numbers as $b\to 0$. The limit of the $k$-point correlation function under this kind of limit is identified by the theorem below, which is proved in Subsection \ref{semilimitsec}. 

Semiclassical limits are important for the following reason. Suppose one is able to construct a quantum theory of gravity in 4D. A valid theory should yield the equations of general relativity in the semiclassical limit. A toy version of this should hold for models of 2D gravity. Thus, semiclassical limits provide an important `test of consistency' for any theory of quantum gravity. The semiclassical limit of spacelike Liouville field theory has been investigated by \citet*{lacoinetal22}. Here, we investigate the semiclassical limit of timelike Liouville field theory.

For  a function $f:S^2\to\R$, let $Gf$ denote the function
\[
Gf(x) := \int_{S^2}G(x,y) f(y) da(y),
\]
where $G$ is the Green's function for the spherical Laplacian, defined in equation \eqref{greendef}. 
Let $\cp$ be the set of probability density functions (with respect to the area measure) on $S^2$. Define the following three functionals on $\cp$:
\begin{align*}
H(\rho) &:= \int_{S^2} \rho(x)\ln \rho(x) da(x),\\
R(\rho) &:= \int_{(S^2)^2} \rho(x)\rho(y) G(x,y) da(x) da(y),\\
L(\rho) &:= \sum_{j=1}^k 4\talpha_j \int_{S^2} G(x_j, x) \rho(x)da(x).
\end{align*}
Let $\cp'$ be the subset of $\cp$ consisting of all $\rho$ such that $H(\rho)$ is finite. We will see later that for $\rho\in \cp'$, the functionals $R(\rho)$ and $L(\rho)$ are also finite.  In the following theorem, we take the logarithm of the $k$-point correlation. While taking the logarithm, we interpret the logarithm of the $e^{-i\pi w}$ term appearing the formula from Theorem \ref{intthmv2} as $-i\pi w$. Since the remaining terms are real and positive, there is no ambiguity about their logarithms.

\begin{restatable}[Semiclassical limit with heavy operators]{thm}{semilimitthm}\label{semilimit}
Let $k\ge 3$, and $x_1,\ldots,x_k$ be distinct points on $S^2$. Let $\tmu$ be a positive real number and $\talpha_1,\ldots,\talpha_k$ be real numbers such that $\talpha_j > -1/2$ for each $j$, and $\beta := -1-\sum_{j=1}^k \talpha_j > 0$. For each positive integer $n$, let 
\[
b_n := \sqrt{\frac{\beta}{n-1}},
\]
so that $b_n >0$ and $b_n\to0$ as $n\to \infty$. Let $\cp'$, $H$, $R$ and $L$ be as above. Define the functional
\[
S(\rho) := L(\rho) + 2\beta R(\rho) + H(\rho). 
\]
Then
\begin{align*}
&\lim_{n\to \infty} \frac{1}{n} \log \tilde{C}(\talpha_1/b_n,\ldots,\talpha_k/b_n; x_1,\ldots,x_k;b_n; \tmu/b_n^2)\\
&= 1+ \ln \tmu - \ln \beta - i\pi +(1-\ln 4)\sum_{j=1}^k \frac{\talpha_j^2}{\beta} + \sum_{j=1}^k \frac{\talpha_j (1+\talpha_j)}{\beta} \ln g(\sigma(x_j))\\
&\qquad \qquad -\frac{4}{\beta}\sum_{1\le j<j'\le k}\talpha_j\talpha_{j'}G(x_j, x_{j'}) - \inf_{\rho\in \cp'} S(\rho).
\end{align*}
Moreover, the infimum on the right is attained at a unique (up to almost everywhere equivalence) $\hat{\rho}\in \cp'$. 
\end{restatable}
This formula for the semiclassical limit seems not to have appeared earlier in the literature, either in physics or mathematics. The closest result in physics is from a recent paper of \citet*{anninosetal21}, who give heuristic calculations for a semiclassical expansion of the timelike Liouville partition function (in the absence of operators) via Feynman diagrams. Two-loop expansions were investigated in \cite{anninosetal21}, and higher expansions were calculated in the follow-up work of \citet{muhlmann22}.

The optimizer $\hat{\rho}$ in Theorem \ref{semilimit} carries useful physical information.  In Subsection \ref{validsec}, we will show that  as $b\to 0$, 
\[
\tilde{C}(\talpha_1/b,\ldots,\talpha_k/b; x_1,\ldots,x_k;b; \tmu/b^2) = \int e^{J(\psi)/b^2 + O(1)}\md\psi,
\]
where
\begin{align*}
J(\psi) &:= -\chi \sum_{j=1}^k \talpha_j^2 + \sum_{j=1}^k \talpha_j (1+\talpha_j)\ln g(\sigma(x_j)) + \sum_{j=1}^k (2\talpha_j\psi(x_j) +2\talpha_j^2 G(x_j,x_j))\notag \\
&\qquad \qquad + \frac{1}{2\pi}\int_{S^2}\psi(x) da(x) -  \frac{1}{4\pi}\int_{S^2}( \psi(x)\Delta_{S^2} \psi(x) + 4\pi \tmu e^{2\psi(x)})da(x).
\end{align*}
Note that the definition of $J$ is not rigorous, since $G(x_j,x_j)=\infty$. But let us ignore this for the time being and keep going. From the above, we may expect that as $b\to 0$, $\tilde{C}(\talpha_1/b,\ldots,\talpha_k/b;x_1,\ldots,x_k;b;\tmu/b^2 )$ should behave like $e^{J(\hat{\psi})/b^2}$ for some critical point $\hat{\psi}$ of $J$. We will show via formal computations in Subsection \ref{validsec} that a critical point $\hat{\psi}$ must satisfy the (generalized) functional equation
\begin{align}\label{releq}
2\sum_{j=1}^k \talpha_j \delta_{x_j}(x) + \frac{1}{2\pi} - \frac{1}{2\pi}\Delta_{S^2} \hat{\psi}(x) - 2\tmu e^{2\hat{\psi}(x)} = 0.
\end{align}
Note that although the definition of $J$ is not rigorous, the above equation is a rigorously meaningful differential equation when the solutions are allowed to be distributions. 
Let $\hat{g}(x) := e^{2\hat{\psi}(x)}g(x)$ be the metric on $S^2$ induced by a critical point $\hat{\psi}$.  A simple computation shows that the Ricci scalar curvature of $\hat{g}$ is given by
\[
R_{\hat{g}}(x) = 2 e^{-2\hat{\psi}(x)}(1-\Delta_{S^2} \hat{\psi}(x)).
\]
Plugging this into equation \eqref{releq}, we get 
\[
R_{\hat{g}}(x) = 8\pi \tmu+8\pi \sum_{j=1}^k \talpha_j \hat{g}(x)^{-1}\delta_{x_j}(x).
\]
This is the equation of motion in JT gravity~\cite{saadetal19} upon insertion of charges at $x_1,\ldots,x_k$. Note that for spacelike Liouville field theory, the constant term on the right would be negative. In the absence of charges, this would give a surface of constant negative curvature. The timelike theory yields positive curvature, as required in JT gravity, for instance; this is another reason why it is closer to a theory of quantum gravity than the spacelike theory.


The following lemma, proved in Subsection~\ref{validsec}, shows that equation \eqref{releq} has no solutions among real-valued fields when the condition required for our Theorem~\ref{semilimit} is satisfied. 
\begin{restatable}[Nonexistence of real critical points]{lmm}{nocriticallmm}\label{nocritical}
Suppose that $\beta := - 1-\sum_{j=1}^k \talpha_j$ is strictly positive. Then there is no map $\psi:S^2 \to \R$ that is a critical point of $J$ in the sense of equation \eqref{releq}.
\end{restatable}

Even though $J$ has no critical points among real-valued functions, it turns out that it does have critical points among complex-valued functions, and the semiclassical limit obtained in Theorem \ref{semilimit} can indeed be expressed using one such critical point. This is the content of the following theorem, proved Subsection \ref{validsec}.
\begin{restatable}[Validity of the semiclassical limit]{thm}{semicritthm}\label{semicrit}
The limit obtained in Theorem~\ref{semilimit} (under the conditions of that theorem) can be formally expressed as $J(\hat{\psi})/\beta $ for some function $\hat{\psi}:S^2\to\C$ that is a critical point of $J$, in the sense that it  satisfies equation~\eqref{releq}. Moreover, this critical point is given by
\begin{align}\label{psihatdef}
\hat{\psi}(x) = -2\beta G\hat{\rho}(x) -\frac{\lambda}{2} +\frac{1}{2}\ln \beta + \frac{i\pi}{2} - \frac{1}{2}\ln \tmu -2\sum_{j=1}^k \talpha_j G(x,x_j),
\end{align}
where $\hat{\rho}$ is the unique minimizer of the function $S$ from Theorem \ref{semilimit}, and 
\begin{align}\label{lambdaval}
\lambda = \ln \int_{S^2} \exp\biggl(-4\beta G\hat{\rho}(x) - 4\sum_{j=1}^k \talpha_j G(x_j, x)\biggr) da(x).
\end{align}
(Here, we say that $J(\hat{\psi})/\beta $ is `formally' equal to the limit in Theorem~\ref{semilimit} because $J$ is not well-defined as a function due to the presence of the $G(x_j,x_j)$ term. However, when we plug in $\hat{\psi}$ as the argument of $J$, there are infinities coming from the term $\int_{S^2}\hat{\psi}(x)\Delta_{S^2} \hat{\psi}(x) da(x)$ that formally cancel out the infinities coming from $G(x_j,x_j)$, yielding a finite result that equals said limit.)
\end{restatable}
Again, this results does not seem to have appeared in the literature. Note that $\hat{\psi}$ has a constant imaginary component of $\frac{1}{2}i \pi$. This has to be the case, because $J$ has no critical points among real-valued functions, as already observed by \citet{harlowetal11}. But since the metric induced by $\hat{\psi}$ is $e^{2\hat{\psi}(x)}g(x)$, it is real-valued even though $\hat{\psi}$ is not. Note that it is not quite a metric, but a pseudometric or a symmetric $2$-tensor. It is common in physics to work with pseudometrics, the most famous example being the Minkowski pseudometric on $\R^4$.

An important observation about Theorem \ref{semicrit} is that it chooses a specific critical point of the action. It is known that the timelike Liouville action with operator insertions may have an infinite number of critical points~\cite{harlowetal11, anninosetal21}. Our analysis reveals that the semiclassical limit concentrates around a specific critical point, and identifies that critical point by relating it to a certain probability density on the sphere.

This concludes the statements of the main results of this paper. In the next subsection, we will introduce a regularized version of timelike Liouville field theory, which will allow us to treat the theory using a framework of `wrong sign' Gaussian random variables that will be developed subsequently.

\subsection{Correlation functions as wrong sign expectations}\label{wrongsec}
Let $L^2_{\R}(S^2)$ denote the space of all real-valued functions from $S^2$ that are square-integrable with respect to the area measure. Recall that a complete orthonormal basis of eigenfunctions of $\Delta_{S^2}$ in $L^2_{\R}(S^2)$  is given by the {\it real spherical harmonics}. The real spherical harmonics $Y_{lm}$ are real-valued functions on $S^2$, where $l$ runs over all nonnegative integers and for each $l$, $m$ ranges over integers from $-l$ to $l$. The functions are orthonormal with respect to the natural inner product on $L^2_{\R}(S^2)$; that is,
\begin{align}\label{s2orth}
\int_{S^2}Y_{lm}(x)Y_{l'm'}(x)da(x) = \delta_{ll'}\delta_{mm'},
\end{align}
where $\delta_{xy} = 1$ if $x=y$ and $0$ if $x\ne y$. Lastly, $Y_{00}$ is the constant function $(4\pi)^{-1/2}$. For each $l$ and $m$, $Y_{lm}$ is an eigenfunction of $\Delta_{S^2}$ with eigenvalue $-l(l+1)$; that is,
\begin{align}\label{eigeq}
\Delta_{S^2} Y_{lm} = -l(l+1)Y_{lm}.
\end{align}
Since the real spherical harmonics form a complete orthonormal basis of $L^2_{\R}(S^2)$, any smooth function $\phi:S^2 \to \R$ can be expanded~as
\begin{align}\label{phifourier0}
\phi = \sum_{l=0}^\infty\sum_{m=-l}^l \hat{\phi}_{lm} Y_{lm}
\end{align}
for unique coefficients $\hat{\phi}_{lm}\in \R$. Using equations~\eqref{s2orth} and~\eqref{eigeq}, we get
\begin{align}\label{nabla2}
\int_{S^2} \phi(x) \Delta_{S^2} \phi(x) da(x) = -\sum_{l=1}^\infty \sum_{m=-l}^l l(l+1)\hat{\phi}_{lm}^2,
\end{align}
where the sum over $l$ starts from $l=1$ because $l(l+1)=0$ for $l=0$. Note also that the zero mode of $\phi$ is given by 
\[
c(\phi) = \frac{1}{4\pi}\int_{S^2} \phi(x) da(x) = \frac{1}{\sqrt{4\pi}} \hat{\phi}_{00}.
\]
Combining this with the observation that the mapping $\phi \to \hat{\phi}$ is a linear bijection, we see that the integral $\int \ldots \mathcal{D}\phi$ with respect to `Lebesgue measure' on the set of all $\phi$ can be replaced by the integral $\int \ldots \mathcal{D}\hat{\phi}$ with respect to `product Lebesgue measure' over the space of all possible coefficients $\hat{\phi} = (\hat{\phi}_{lm})_{l\ge0,\, -l\le m\le l}$. Thus, we may heuristically rewrite the correlation function displayed in equation \eqref{favg2} as  
\begin{align}\label{favgdef0}
&\tilde{C}(\alpha_1,\ldots,\alpha_k; x_1,\ldots,x_k; b; \mu) \notag \\
&= \int \prod_{j=1}^k \tilde{V}_{\phi}(\alpha_j, x_j) \exp\biggl(- \frac{2Q \hat{\phi}_{00}}{\sqrt{4\pi}} + \frac{1}{4\pi}  \sum_{l=1}^\infty \sum_{m=-l}^l l(l+1)\hat{\phi}_{lm}^2\notag \\
&\qquad \qquad  - \mu \int_{S^2} e^{2b\phi(x) + 2b^2 G(x,x)}da(x) \biggr) \mathcal{D}\hat{\phi}, 
\end{align}
where $\phi$ is related to $\hat{\phi}$ through equation \eqref{phifourier0} and $G(x,y) := G_g(\sigma(x), \sigma(y))$ is the Green's function for the round metric on $S^2$.

To give a rigorous meaning to the integral displayed in equation \eqref{favgdef0}, we start by making the following modifications:
\begin{itemize}
\item Add a term $\epsilon \hat{\phi}_{00}^2$ inside the brackets, where $\epsilon>0$ is a number that will eventually be taken to zero. This is a regularization term.
\item Multiply the term $l(l+1)\hat{\phi}_{lm}^2$ by $\lambda^{-l}$, where $\lambda\in (0,1)$ is a number that will eventually be taken to one. This is another regularization term.
\item Replace the sum over $l$ from $1$ to $\infty$ with a sum from $1$ to $L$. This is an ultraviolet cutoff.
\end{itemize}
The modified integral takes the form 
\begin{align}
&\int \prod_{j=1}^k \tilde{V}_{\phi_L}(\alpha_j, x_j) \exp\biggl(\epsilon \hat{\phi}_{00}^2- \frac{2Q \hat{\phi}_{00}}{\sqrt{4\pi}} \notag \\
& + \frac{1}{4\pi}  \sum_{l=1}^L \sum_{m=-l}^l \lambda^{-l} l(l+1)\hat{\phi}_{lm}^2 - \mu \int_{S^2} e^{2b\phi_L(x)+2b^2G_{\lambda, L}(x,x)}da(x) \biggr) \prod_{l=0}^L\prod_{m=-l}^l d\hat{\phi}_{lm},\label{wrong20}
\end{align}
where
\begin{align*}
&\phi_L(x) := \sum_{l=0}^L\sum_{m=-l}^l \hat{\phi}_{lm} Y_{lm}(x),\\
&\tilde{V}_{\phi_L}(\alpha, x) := e^{\chi\alpha(b-\alpha)} g(\sigma(x))^{-\Delta_\alpha}\normord{e^{2\alpha \phi_L(x)}}\,= e^{\chi\alpha(b-\alpha)} g(\sigma(x))^{-\Delta_\alpha}e^{2\alpha \phi_L(x) + 2\alpha^2 G_{\lambda, L}(x,x)},
\end{align*}
and $G_{\lambda, L}$ is the regularized version of $G$, given by 
\begin{align}\label{gldef}
G_{\lambda, L}(x,y) := \sum_{l=1}^L \sum_{m=-l}^l \lambda^l \frac{2\pi}{l(l+1)} Y_{lm}(x)Y_{lm}(y).
\end{align}
Now, suppose we have some notion of a `wrong sign' standard Gaussian distribution on the real line, that has probability density proportional to $e^{\frac{1}{2}x^2}$ instead of $e^{-\frac{1}{2}x^2}$. Let us denote this distribution by $N(0,-1)$, that is, normal with mean $0$ and variance $-1$. 
Let $(X_{lm})_{0\le l\le L, -l\le m\le l}$ be i.i.d.~$N(0,-1)$ random variables. For each $x\in S^2$, define 
\begin{align}\label{xldef}
X_{\lambda, L}(x) := \sum_{l=1}^L \sum_{m=-l}^l \lambda^{l/2} \sqrt{\frac{2\pi}{l(l+1)}} X_{lm} Y_{lm}(x),
\end{align}
and let
\[
D := \frac{X_{00}}{\sqrt{8\pi \epsilon}}.
\]
Then it is evident that up to a normalizing constant that we will ignore, the integral displayed in equation~\eqref{wrong20} can be expressed as the  wrong sign expectation 
\begin{align}
&\tilde{C}_{\epsilon, \lambda, L}(\alpha_1,\ldots,\alpha_k;x_1,\ldots,x_k;b;\mu )\notag \\
&:= \E\biggl[\prod_{j=1}^k \tilde{V}_{X_{\lambda,L}}(\alpha_j, x_j)\exp\biggl(- 2bwD - \mu e^{2bD} \int_{S^2} e^{2bX_{\lambda,L}(x)+2b^2G_{\lambda,L}(x,x)}da(x) \biggr)\biggr],\label{regcor}
\end{align}
where $w = (Q-\sum_{j=1}^k \alpha_j)/b$. Our first goal will be to give a rigorous meaning to the above `wrong sign' expectation and calculate its value. To do this, we develop  a theory of `wrong sign' Gaussian random variables in the next section. After calculating this expectation, we will remove the cutoffs by sending $\epsilon \to 0$, $L\to \infty$, and $\lambda\uparrow 1$, in this order.

\section{A theory of wrong sign Gaussian random  variables}
In this section we develop a theory of Gaussian random variables with negative variance that arise in the definition of timelike Liouville field theory. The theory may be of independent interest, and may have applications in other models of quantum gravity.
\subsection{Problem formulation}\label{problemsec}
Recall that a Gaussian random variable $X$ with mean $a\in \R$ and variance $b> 0$ has probability density function 
\begin{align}\label{density}
\frac{1}{\sqrt{2\pi b}}\exp\biggl(-\frac{(x-a)^2}{2b}\biggr).
\end{align}
We write $X\sim N(a,b)$ to denote that $X$ has the above distribution. We will now define a notion of a `wrong sign' Gaussian distribution, where the variance is allowed to be negative. More generally, we will define the notion of an $(m+n)$-dimensional random vector $Z = (X_1,\ldots,X_m, Y_1,\ldots,Y_n)$ where the coordinates are independent,  $X_1,\ldots,X_m$ are~$N(0,1)$ random variables, and $Y_1,\ldots,Y_n$ are~$N(0,-1)$ random variables. 

Suppose we are able to define this distribution, in the sense that we are able to define $\E(f(Z))$ for $f$ belonging to some class of complex-valued functions $\mf_{m,n}$ on $\R^m \times \R^n$. From physical and mathematical considerations, we would like this definition  to at least have the following properties.
\begin{enumerate}
\item\label{cond:rich} The function class $\mf_{m,n}$ should be rich enough to include elementary functions such as polynomials and exponentials. Moreover, for such functions, $\E(f(Z))$ should be the same as what we would obtain if we first compute the expectation assuming that $Y_1,\ldots,Y_n$ are i.i.d.~$N(0,v)$ for some $v>0$, and naively substitute $v=-1$ in the formula obtained from this computation. For example, if $X\sim N(0,-1)$, then we should have $\E(e^{aX}) = e^{-\frac{1}{2}a^2}$, because $\E(e^{aY}) = e^{\frac{1}{2}v a^2}$ when $Y\sim N(0,v)$ for $v>0$. 
\item\label{cond:lin} Since expectation must be linear, the class $\mf_{m,n}$ should be a vector space over $\C$, and we should have $\E(af(Z)+bg(Z)) = a \E(f(Z))+b \E(g(Z))$ for all $f,g\in \mf_{m,n}$ and $a,b\in \C$. Moreover, if $f$ is identically equal to a constant $c$, then $\E(f(Z))$ should be equal to $c$.
\item\label{cond:real} If $f$ is real-valued, then $\E(f(Z))$ should be real. This comes from physical considerations, because the expected value of a real-valued observable should not have a nonzero imaginary component.
\end{enumerate}

\subsection{Ideas that do not work}\label{donotworksec}
Suppose we want to define $\E(f(X))$, where $X\sim N(0,-1)$, such that the definition obeys the natural conditions stated in the previous subsection. There are two `obvious' approaches that look promising but do not work.  The simplest idea is to just work with $e^{\frac{1}{2}x^2}$ as a density and define $\E(f(X))$ only for those $f$ that are integrable with respect to this density. But this clearly violates condition~\ref{cond:rich}, since this class does not even contain the polynomials. The other idea is that perhaps integrating with respect to some other suitable measure, or maybe a signed or complex measure, may give the desired result. We will now demonstrate that none of these ideas can work. 

Suppose that $\E(f(X))$ is given by $\int f(x) d\gamma(x)$ for some $\gamma$ which is either a nonnegative measure, or a signed measure, or a complex measure on the real line. Note that by condition~\ref{cond:rich}, we should have $\E(X^2)=-1$. Thus, $\gamma$ cannot be a  nonnegative measure. Suppose that $\gamma$ is a signed measure. A signed measure is allowed to take the values $+\infty$ or $-\infty$, but not both. We will now show that this condition would be violated for our $\gamma$. If $Y\sim N(0,v)$, then $\E(\cos aY) = e^{-\frac{1}{2}va^2}$ for any $a\in \R$. Thus, by condition~\ref{cond:rich}, we should have 
\[
\E(\cos aX) = e^{\frac{1}{2}a^2}.
\]
In other words, for any $a>0$,
\begin{align}\label{cosid}
\int \cos ax \, d\gamma(x) =e^{\frac{1}{2}a^2}.
\end{align}
Since the right side tends to $\infty$ as $a\to \infty$ and $|\cos ax|\le 1$ for all $a$ and $x$, it is easy to show from this that there exists a Borel set $A$ such that $\gamma(A)=\infty$. Next, by the same calculation, we have
\[
\int (1- \cos ax) \, d\gamma(x) = \E(1-\cos aX)  = 1- e^{\frac{1}{2}a^2}.
\]
The right side tends to $-\infty$ as $a \to \infty$. But $0\le 1-\cos ax \le 2$ for all $a$ and $x$. Again, it follows from this that there is some Borel set $B$ such that $\gamma(B)=-\infty$. This shows that $\gamma$ cannot be a signed measure. 

Lastly, suppose that $\gamma$ is a complex measure. Recall that for any complex measure $\gamma$, there is a nonnegative measure $|\gamma|$ with finite total mass, called the `variation' of $\gamma$, such that for any $f$, $|\int fd\gamma|\le \int |f|d|\gamma|$. But  then, the identity \eqref{cosid} shows that
\[
e^{\frac{1}{2}a^2}= \biggl|\int_{\R} \cos ax\, d\gamma(x)\biggr|\le \int_{\R} |\cos ax|\, d|\gamma|(x) \le |\gamma|(\R).
\]
Since this holds for every $a$, we have $|\gamma|(\R)=\infty$, which violates the condition that $|\gamma|$ has finite total mass.

\subsection{The wrong way to do analytic continuation}\label{anasec}
Physicists define wrong sign Gaussian distributions via analytic continuation. One approach goes as follows. Suppose we want to evaluate $\E(f(X))$ for some function $f$, where $X\sim N(0,-1)$. We define $h(s) := \E(f(sY))$, where $Y\sim N(0,1)$ and $s>0$; then, we analytically continue $h$ to the imaginary axis; finally, we define $\E(f(X)) := h(i)$, with the idea that $iY$ mimics a $N(0,-1)$ random variable. This works well in many situations, for example when $f$ is polynomial or exponential. However, there is no mathematical theory around this, and therefore we do not know precise conditions under which this approach does not lead to contradictions or violations of the conditions listed in Subsection~\ref{problemsec}.

Indeed, problems do arise in practice. This is particularly relevant for  timelike Liouville field theory. The approach outlined above corresponds to the following method of going from spacelike to timelike Liouville theory: Compute the correlations functions for spacelike Liouville theory, and then analytically continue in the parameter $b$ to replace it by $ib$. It was shown by \citet{zamolodchikov05} that this fails `rather dramatically', to quote from the discussion in \citet[Section 7]{harlowetal11}. The following simple example is a toy version of the path integral in timelike Liouville theory, which illustrates the kind of problem that leads to this failure.

Consider the function $f(x) := \exp(-e^x- e^{-x})$ on the real line. To define $\E(f(X))$ for $X\sim N(0,-1)$, let us define $h(s) := \E(f(sY))$ for $s>0$, where $Y\sim N(0,1)$. By the change of variable $u = e^{sy}$ below, we obtain
\begin{align*}
h(s) &= \frac{1}{\sqrt{2\pi}} \int_{-\infty}^\infty \exp\biggl(-e^{sy}-e^{-sy}-\frac{1}{2}y^2\biggr) dy\\
&= \frac{1}{\sqrt{2\pi} s} \int_0^\infty \frac{1}{u}\exp\biggl(-u-\frac{1}{u} -\frac{1}{2s^2}(\ln u)^2\biggr) du.
\end{align*}
Let us now extend the domain of $h$ by defining $h:\C\setminus\{0\}\to \C$ as
\begin{align*}
h(z) 
&:= \frac{1}{\sqrt{2\pi} z} \int_0^\infty \frac{1}{u} \exp\biggl(-u -\frac{1}{u}-\frac{1}{2z^2}(\ln u)^2\biggr) du.
\end{align*}
It is easy to see that the integral on the right is absolutely convergent for any $z\in \C \setminus\{0\}$. Moreover, a simple argument via the dominated convergence theorem shows that $h$ is holomorphic on this domain. Thus, the analytic continuation approach as outlined above, dictates that we should define
\begin{align}\label{efx}
\E(f(X)) := h(i) = \frac{1}{\sqrt{2\pi} i} \int_0^\infty\frac{1}{u} \exp\biggl(-u -\frac{1}{u}+\frac{1}{2}(\ln u)^2\biggr) du.
\end{align}
But note that this is not a real number, thus violating our condition \ref{cond:real} that the expected value of a real-valued function should be real.


\subsection{Analytic continuation done right}
We will now present a rigorous theory of wrong sign Gaussian random variables that avoids contradictions and gives the `right' way to implement analytic continuation. The idea is as follows. Suppose we want to calculate $\E(f(X))$ where $X\sim N(0,-1)$. Instead of first calculating $\E(f(sY))$ for $Y\sim N(0,1)$ and $s>0$ and then analytically continuing to $s=i$, the correct thing to do is to first analytically continue the function $f$, and define $\E(f(X))$ to be $\E(f(iY))$. This small adjustment guarantees that expected values of real-valued functions are real, as we will see below. Fundamentally, it is a consequence of the Schwarz reflection principle. We will see in subsequent sections that it yields  the correct formula for the correlation functions. 

For the general definition, take any $m\ge 0$ and $n>0$.  We define $\mf_{m,n}$ to be the class of functions $f:\R^{m+n}\to \C$ such that $f$ has an analytic continuation in the last $n$ coordinates  to a function $\tilde{f} : \R^m \times \Omega \to \C$, where $\Omega$ is an open subset of $\C^n$ that contains $(\R \cup i\R)^n$, such that 
\[
\E|\tilde{f}(W_1,\ldots,W_m, iW_{m+1},\ldots,iW_{m+n})|<\infty,
\]
where $W_1,\ldots,W_{m+n}$ are i.i.d.~$N(0,1)$ random variables. If such an $\tf$ exists, we define
\begin{align}\label{wrongdef}
\E(f(Z)) := \E(\tilde{f}(W_1,\ldots,W_m, iW_{m+1},\ldots,iW_{m+n})),
\end{align}
where the vector $Z:= (X_1,\ldots,X_m, Y_1,\ldots,Y_n)$ has independent coordinates, $X_1,\ldots,X_m$ are  $N(0,1)$ random variables, and $Y_1,\ldots,Y_n$ are $N(0,-1)$ random variables. We will henceforth denote this by $Z\sim N_{m,n}$. 

An immediate problem with the above definition is that it is not clear that $\tf$, if it exists, is unique. 
The following lemma shows that $\tf$ is unique on the relevant part of the domain.
\begin{lmm}\label{uniquelmm}
In the above setting, if $\tf_1$ and $\tf_2$ are two analytic continuations of $f$ that satisfy the necessary criteria, then $\tf_1=\tf_2$ on $\R^m \times(\R\cup i\R)^n$.
\end{lmm}
\begin{proof}
We will  show by induction on $k$ that for each $1\le k\le n$, $\tf_1$ and $\tf_2$ coincide on $\R^m \times (\R \cup i\R)^k \times \R^{n-k}$. First, take $k=1$. Let the domains of $\tf_1$ and $\tf_2$ be $\R^m \times \Omega_1$ and $\R^m \times \Omega_2$. Fix $x_1,\ldots,x_m, y_2,\ldots,y_n\in \R$. Let 
\[
\Omega:=\{z\in \C: (x_1,\ldots,x_m, z, y_2,\ldots,y_n)\in \Omega_1\cap \Omega_2\}.
\]
Since $\Omega_1\cap \Omega_2$ contains $(\R\cup i\R)^n$, it follows that $\Omega$ contains $\R \cup i\R$. Since $\R \cup i\R$ is a connected set, $\R \cup i\R$ must be a subset of one of the connected components $\Omega_0$ of $\Omega$. Since $\Omega_1\cap \Omega_2$ is open, it follows that $\Omega$ is open.  Thus, every connected component of $\Omega$ is open. In particular, $\Omega_0$ is open. 

Define $h_1, h_2:\Omega_0\to \C$ as 
\[
h_1(z) := \tf_1(x_1,\ldots,x_m, z, y_2,\ldots,y_n), \ \ \  h_2(z) := \tf_2(x_1,\ldots,x_m, z, y_2,\ldots,y_n).
\]
Then $h_1$ and $h_2$ are holomorphic on the open connected set $\Omega_0$, and coincide on the real line due to the hypothesis that $\tf_1=\tf_2=f$ on $\R^m \times (\R\cup i\R)^n$. Thus, $h_1(z)=h_2(z)$ for all $z\in \Omega_0$. In particular, $h_1(iy)=h_2(iy)$ for all $y\in \R$. This proves the claim for~$k=1$. 

Next, suppose that we have proved the claim up to $k-1$, for some $k\ge 2$. Fix some $x_1,\ldots,x_m, y_1,\ldots, y_{k-1}, y_{k+1},\ldots,y_n\in \R$, and define
\[
\Omega:=\{z\in \C: (x_1,\ldots,x_m, y_1,\ldots, y_{k-1}, z, y_{k+1}\ldots,y_n)\in \Omega_1\cap \Omega_2\}.
\]
Let $\Omega_0$ be the connected component of $\Omega$ that contains $\R\cup i\R$. Define $h_1, h_2:\Omega_0\to \C$ as 
\begin{align*}
h_1(z) &:= \tf_1(x_1,\ldots,x_m, y_1,\ldots, y_{k-1},z, y_{k+1},\ldots,y_n), \\
 h_2(z) &:= \tf_2(x_1,\ldots,x_m, y_1,\ldots, y_{k-1},z, y_{k+1},\ldots,y_n).
\end{align*}
As before, $\Omega_0$ is an open connected set, $h_1$ and $h_2$ are holomorphic on $\Omega_0$, and coincide on $\R$. Thus, $h_1=h_2$ on  $i\R$. This completes the induction step. The case $k=n$ proves the lemma.
\end{proof}

It follows from equation \eqref{wrongdef} and Lemma \ref{uniquelmm} that our definition of wrong sign expectation satisfies the  conditions \ref{cond:rich} and \ref{cond:lin}. (To see that it satisfies condition~\ref{cond:lin}, note that equation \eqref{wrongdef} implies linearity of wrong sign expectation, and that if $f$ is identically equal to a constant $c$, then $\E(f(Z)) = c$. To see that it satisfies condition \ref{cond:rich}, note that polynomials and exponentials admit straightforward analytic continuations, and applying the definition \eqref{wrongdef} of wrong sign expectation to such functions yield the required answers.) The following result shows that condition \ref{cond:real} is also satisfied.
\begin{thm}\label{wrongthm}
Suppose that $f\in \mf_{m,n}$ is real-valued and $Z\sim N_{m,n}$. Then $\E(f(Z))$, as defined in equation \eqref{wrongdef}, is real.
\end{thm}
\begin{proof}
Let $\tf$ be the analytic continuation of $f$ on the domain $\R^m \times \Omega$, satisfying the required criteria. We will prove by induction on $k$ that for any $1\le k\le n$ and any $x_1,\ldots,x_m, y_1,\ldots,y_n\in \R$,
\begin{align}\label{schwarz}
\sum_{s_1,\ldots,s_k\in \{-1,1\}} \tf(x_1,\ldots,x_m, i s_1 y_1,\ldots, i s_k y_k, y_{k+1},\ldots,y_n)\in \R.
\end{align}
First, take $k=1$. Fix $x_1,\ldots,x_m, y_2,\ldots,y_n\in \R$ and let
\[
\Omega_*:=\{z\in \C: (x_1,\ldots,x_m, z, y_2,\ldots,y_n)\in \Omega\}.
\]
Define $h:\Omega_* \to \C$ as 
\[
h(z) :=\tf(x_1,\ldots,x_m, z, y_2,\ldots,y_n).
\]
As in the proof of Lemma \ref{uniquelmm}, $\Omega_*$ has an open connected component $\Omega_0$ containing $\R \cup i\R$. Let $\Omega_0^+ := \{z\in \Omega_0: \Im(z) >0\}$ be the intersection of $\Omega_0$ with the open upper half-plane. Clearly, $\Omega_0^+$ is open. We claim that $\Omega_0^+$ is connected. To see this, suppose not. Since a subset of Euclidean space is connected if and only if it is path connected, this implies that $\Omega_0^+$ has multiple nonempty path connected components. It is easy to see that each component is open. Thus there exist disjoint nonempty open sets $A,B$ whose union is $\Omega_0^+$, such that no point in $A$ is connected to a point in $B$ by a continuous path that lies entirely in $\Omega_0^+$. 

Since $\Omega_0$ is open and contains $\R$, we deduce that for each $x\in \R$, there is an open disk $B_x$ with center $x$ such that $B_x\subseteq \Omega_0$. Let $B_x^+ := B_x\cap \Omega_0^+$, so that $B_x^+$ is also the intersection of $B_x$ and the open upper half-plane. Thus, $B_x^+$ is open and connected, but is also the union of the disjoint open sets $B_x^+\cap A$ and $B_x^+\cap B$. Therefore, one of these two sets must be empty. In other words, either $B_x^+ \subseteq A$ or $B_x^+\subseteq B$. Let
\[
U := \bigcup_{x: B_x^+\subseteq A} (B_x\cap \R), \ \ \  V := \bigcup_{x: B_x^+\subseteq B} (B_x\cap \R).
\]
Since $U$ and $V$ are unions of open intervals, they are open subsets of $\R$. By the above argument, they are disjoint, and their union is $\R$. Since $\R$ is connected, this implies that one of $U$ and $V$ must be empty. Suppose that $V$ is empty, so that no $B_x^+$ intersects $B$. 

Take any $x\in B$ and $y\in A$. Since $\Omega_0$ is connected and hence path connected, there is a continuous curve $\gamma:[0,1]\to \Omega_0$ such that $\gamma(0)=x$ and $\gamma(1)=y$. By the construction of $A$ and $B$, we know that this path cannot lie entirely in $\Omega_0^+$. Let $t:= \inf\{s: \gamma(s)\notin \Omega_0^+\}$. By the continuity of $\gamma$, we deduce that $\Im(\gamma(t)) = 0$, and hence, $w:=\gamma(t)\in \R$. Let $B_w$ be as above, so that $B_w^+\subseteq A$. Then, again by the continuity of $\gamma$, there exists $s\in [0,t)$ that is so close to $t$ that $\gamma(s)\in B_w$. Then the restriction of $\gamma$ to the interval $[0,s]$ is a continuous curve, fully contained in $\Omega_0^+$, that connects a point in $A$ to a point in $B$. This is a contradiction that proves our claim that $\Omega_0^+$ is connected.

Let $\tilde{\Omega}$ be the union of $\Omega_0^+$, its reflection across the real line, and $\R$. It is easy to see that $\tilde{\Omega}$ is open, by checking that each point in $\tilde{\Omega}$ belongs to an open disk that is contained in $\tilde{\Omega}$. Note that  $h$ is holomorphic in $\Omega_0^+$, $h$ real-valued on $\R$, and each $x\in \R$ is contained in an open disk $B_x$ such that the intersection of $B_x$ and the open upper half-plane is contained in $\Omega_0^+$. Thus, by the generalized form of the Schwarz reflection principle~\cite[Theorem 11.14]{rudin87}, $h$ has an extension to a holomorphic function  $\tilde{h}$ on $\tilde{\Omega}$, defined as
\[
\tilde{h}(z) :=
\begin{cases}
h(z) &\text{ if } z\in \Omega_0^+\cup \R,\\
\overline{h(\overline{z})} &\text{ if } \overline{z}\in \Omega_0^+.
\end{cases}
\]  
Note that $\tilde{\Omega}$ and $\Omega_0$ are open sets that contain $i\R$. Thus, for each $ix\in i \R$, there is an open disk $B_{ix}$ centered at $ix$ such that $B_{ix}\subseteq \tilde{\Omega}\cap \Omega_0$. Let $W$ be the union of these disks. Since $i \R$ is path connected, it follows that $W$ is path connected, and hence, connected. Clearly, $W$ is open. Now, note that both $h$ and $\tilde{h}$ are holomorphic on $W$, and coincide on the upper half of the imaginary axis. Therefore, $h=\tilde{h}$ everywhere on $W$. In particular, for any $ix$ on the upper half of the imaginary axis,
\[
h(-ix) = \tilde{h}(-ix) = \overline{h(ix)}. 
\]
This is the same as saying that any $y_1\in \R$,
\[
\tf(x_1,\ldots,x_m,-iy_1,y_2,\ldots,y_n) = \overline{\tf(x_1,\ldots,x_m,i y_1, y_2,\ldots,y_n)}. 
\]
This proves the claim \eqref{schwarz} for $k=1$. Next, suppose that it holds for $k-1$. Fix $x_1,\ldots,x_m,y_1,\ldots,y_{k-1}, y_{k+1},\ldots,y_n\in \R$, and define
\[
\Omega_*:=\{z\in \C: (x_1,\ldots,x_m, y_1,\ldots, y_{k-1}, z, y_{k+1},\ldots,y_n)\in \Omega\}.
\]
Define $h:\Omega_* \to \C$ as 
\[
h(z) := \sum_{s_1,\ldots,s_{k-1}\in \{-1,1\}} \tf(x_1,\ldots,x_m, i s_1y_1,\ldots,is_{k-1}y_{k-1}, z, y_{k+1},\ldots,y_n). 
\]
Note that $\Omega_*$ is open, $\Omega_*$ contains $\R\cup i\R$, $h$ is holomorphic on $\Omega_*$, and by the induction hypothesis, $h$ real-valued on $\R$. Proceeding exactly as before, we deduce that $h(-ix) = \overline{h(ix)}$ for any $x\in \R$. Thus, 
\begin{align*}
\sum_{s_1,\ldots,s_{k-1}\in \{-1,1\}} &\tf(x_1,\ldots,x_m, i s_1y_1,\ldots,is_{k-1}y_{k-1}, -iy_k, y_{k+1},\ldots,y_n) \\
&= \overline{\sum_{s_1,\ldots,s_{k-1}\in \{-1,1\}} \tf(x_1,\ldots,x_m, i s_1y_1,\ldots,is_{k-1}y_{k-1}, iy_k, y_{k+1},\ldots,y_n)}
\end{align*}
for any $x_1,\ldots,x_m, y_1,\ldots,y_n\in \R$. This completes the induction step. 

Let $W_1,\ldots,W_{m+n}$ be i.i.d.~$N(0,1)$ random variables. Since the standard Gaussian distribution is symmetric around zero, we have that for any $s_1,\ldots,s_n\in \{-1,1\}$, 
\begin{align*}
\E(\tf(W_1,\ldots,W_m, i W_{m+1},\ldots, iW_{m+n})) =\E(\tf(W_1,\ldots,W_m, is_1W_{m+1},\ldots,i s_n W_{m+n})).
\end{align*}
This shows that 
\begin{align*} 
&\E(\tf(W_1,\ldots,W_m, i W_{m+1},\ldots, iW_{m+n}))\\
 &= \frac{1}{2^n}\sum_{s_1,\ldots,s_n\in \{-1,1\}}\E(\tf(W_1,\ldots,W_m, is_1W_{m+1},\ldots,i s_n W_{m+n}))\\
&= \frac{1}{2^n}\E\biggl(\sum_{s_1,\ldots,s_n\in \{-1,1\}}\tf(W_1,\ldots,W_m, is_1W_{m+1},\ldots,i s_n W_{m+n})\biggr).
\end{align*}
By equation \eqref{schwarz} for $k=n$, the right side is real. This completes the proof.
\end{proof}

\subsection{Examples}
It is clear that we get the expected results with simple functions like polynomials and exponentials. For a nontrivial example, let us consider the function 
\[
f(x)=\exp(-e^x-e^{-x})
\] 
from Subsection \ref{anasec}, which caused problems with the `naive' analytic continuation approach. This function can be analytically continued as $\tf(z) = \exp(-e^z - e^{-z})$ to the entire complex plane. Also, if $Z\sim N(0,1)$, then $\E|\tf(iZ)|< \infty$ since $\tf$ is bounded on the imaginary axis. Thus, $f\in \mf_{0,1}$, and for $X\sim N(0,-1)$, we have
\begin{align*}
\E(f(X)) &= \E(\tf(iZ)) = \E(\exp(-e^{iZ} - e^{-iZ})). 
\end{align*}
A simple application of the dominated convergence theorem shows that the expectation on the right can be calculated by expanding in power series and moving the expectation within the sum, to give
\begin{align*}
\E(f(X)) &= \sum_{k=0}^\infty \frac{(-1)^k}{k!}\E((e^{iZ} + e^{-iZ})^k)\\
&=   \sum_{k=0}^\infty \frac{(-1)^k}{k!}\biggl\{\sum_{j=0}^k {k\choose j} \E(e^{i(k-2j)Z})\biggr\}\\
&= \sum_{k=0}^\infty \sum_{j=0}^k \frac{(-1)^k e^{-\frac{1}{2}(k-2j)^2}}{j!(k-j)!}.
\end{align*}
This is a finite real value, not suffering from the problem with the `wrong' imaginary value in equation \eqref{efx} that we previously calculated using the naive analytic continuation method.

The next example shows that it is important to carefully verify the conditions of Theorem \ref{wrongthm} before defining a wrong sign expectation. Consider the function 
\[
f(x)=\sqrt{1+x^2}
\]
on $\R$. This a smooth function on $\R$ which can be analytically continued to a domain whose {\it closure} contains $\R\cup i\R$, as follows. Let $\sqrt{\cdot}$ denote the analytic branch of the square-root in $\C\setminus\{ix: -\infty<x\le 0\}$. Explicitly, if $z= re^{i\theta}$ for $r> 0$ and $\theta \in (-\pi/2, 3\pi/2)$, then $\sqrt{z} = \sqrt{r}e^{i\theta/2}$. In particular, for a negative real number $x$, $\sqrt{x} = i\sqrt{|x|}$, since $x=|x|e^{i\pi}$. Then the function
\begin{align}\label{tfdef}
\tilde{f}(z) := \sqrt{1+z^2}
\end{align}
is analytic in the domain 
\begin{align}\label{omegadef}
\Omega := \{x+iy : 1+x^2 - y^2 \ne 0 \text{ or } xy > 0\},
\end{align}
and is equal to $f$ on $\R$. Note that $\Omega$ contains $\R$ and $i\R\setminus\{-i,i\}$. The omission of $\pm i$ from the domain can be remedied by  extending  $\tilde{f}$ continuously to the full imaginary axis by defining $\tilde{f}(\pm i) = 0$. Now suppose we define $\E(f(X))$, for $X\sim N(0,-1)$, to be the number $\E(\tilde{f}(iZ))$, where $Z\sim N(0,1)$. Then it turns out the $\E(f(X))$ has a nonzero imaginary component even though $f$ is real-valued, thus violating our condition \ref{cond:real}. To see this, simply note that 
\[
\tilde{f}(iZ) =
\begin{cases}
\sqrt{1-Z^2} &\text{ if } |Z|\le 1,\\
i\sqrt{Z^2 - 1} &\text{ if } |Z|>1,
\end{cases}
\]
which gives 
\[
\E(\tilde{f}(iZ)) = \int_{|x|\le 1}\frac{1}{\sqrt{2\pi}} \sqrt{1-x^2} e^{-\frac{1}{2}x^2} dx + i\int_{|x|>1} \frac{1}{\sqrt{2\pi}} \sqrt{x^2-1} e^{-\frac{1}{2}x^2} dx.
\]
Clearly, the imaginary part is nonzero. In a nutshell, we cannot define $\E\sqrt{1+X^2}$ as $\E\sqrt{1+(iZ)^2}$ even though the latter expectation is mathematically well-defined and `almost' satisfies the conditions of Theorem \ref{wrongthm}.

Can we simulate wrong sign Gaussian random variables on the computer? Unfortunately, the answer seems to be no, for the following reason. Let $A$ be a Borel subset of $\R$ and $f$ be the indicator of $A$, that is, the function that is $1$ on $A$ and $0$ outside. Then, unless $A=\R$ or $A = \emptyset$, $f$ is not continuous on $\R$ and therefore cannot be analytically continued to any open subset of $\C$ that contains $\R$. Consequently, unless $A=\R$ or $A=\emptyset$, we cannot define $\P(X\in A)$ for $X\sim N(0,-1)$. On the other hand, $\P(X\in \R)=1$ and $\P(X\in \emptyset) = 0$. This indicates that it is impossible to simulate $X$  on a computer.

\subsection{Integration by parts}
Recall that if $Y\sim N(0,1)$, then for any differentiable function $f$ such that $\E|Yf(Y)|$ and $\E|f'(Y)|$ are finite, we have the identity 
\[
\E(Yf(Y)) = \E(f'(Y))
\]
obtained using integration by parts. This  simple identity is the basis of a large class of identities in physics that often go by the name of `Ward identities'. More generally, if $Y\sim N(0,v)$, then
\[
\E(Yf(Y)) = v \E(f'(Y)). 
\]
Thus, if $X\sim N(0,-1)$, we should have
\[
\E(Xf(X)) = - \E(f'(X))
\]
for any $f$ such that both sides are defined. 
It is not clear if the naive analytic continuation from Subsection \ref{anasec} satisfies this condition. The following theorem shows that our definition of wrong sign expectation does.
\begin{thm}
Let $Z = (X_1,\ldots,X_m, Y_1,\ldots,Y_n)\sim N_{m,n}$. Take any $1\le j\le n$ and let $f:\R^{m+n} \to \C$ be a function that is differentiable in coordinate $m+j$, with the derivative denoted by $\partial_{m+j} f$. If the functions $g(x_1,\ldots,x_{m+n}) := x_{m+j}f(x_1,\ldots,x_{m+n})$ and $\partial_{m+j} f$ are in $\mf_{m,n}$, then $\E(Y_j f(Z)) = -\E(\partial_{m+j} f(Z))$.
\end{thm}
\begin{proof}
Since $\partial_{m+j}f \in \mf_{m,n}$, it has an analytic continuation $h$ to $\R^m \times \Omega$, where $\Omega$ is an open subset of $\C^n$ that contains $(\R\cup i\R)^n$. Define $\tf: \R^m \times \Omega\to \C$ as
\begin{align*}
\tf(x_1,\ldots,x_m, z_1,\ldots,z_n) &:= f(x_1,\ldots,x_m, z_1,\ldots,z_{j-1},0,z_{j+1},\ldots,z_n) \\
&\qquad + \int_0^{1} z_j h(x_1,\ldots, x_m, z_1,\ldots,z_{j-1}, tz_j, z_{j+1},\ldots,z_n) dt.
\end{align*}
Then $\tf$ is an analytic function on $\R^m \times \Omega$, and is equal to $f$ on $\R^{m+n}$. Thus, the function 
\[
\tilde{g}(x_1,\ldots,x_m, z_1,\ldots,z_n) := z_j \tf(x_1,\ldots,x_m, z_1,\ldots,z_n)
\]
is the analytic continuation of $g$ to $\R^m \times \Omega$.  So by Lemma \ref{uniquelmm} and the fact that $g\in \mf_{m,n}$, we conclude that $\E|\tilde{g}(W)|<\infty$ and $\E(g(Z)) = \E(\tilde{g}(W))$, where
\[
W := (W_1,\ldots,W_m, iW_{m+1},\ldots,iW_{m+n}),
\]
and $W_1,\ldots,W_{m+n}$ are  i.i.d.~$N(0,1)$ random variables.
Thus, 
\[
\E(Y_j f(Z)) = \E(g(Z)) = \E(\tilde{g}(W)) = \E(iW_{m+j} \tf(W)). 
\]
On the other hand, $\E(\partial_{m+j} f(Z)) =\E(h(W))$ and $\E|h(W)|<\infty$. By the definition of $\tf$, we have that $\partial_{m+j} \tf = h$. Thus, by the usual integration by parts formula for Gaussian random variables, we get
\begin{align*}
\E(iW_{m+j}\tf(W))= i^2\E(\partial_{m+j} \tf(W)) = - \E(h(W)) = -\E(\partial_{m+j} f(Z)).
\end{align*}
Combining the last two displays completes the proof.
\end{proof}

\subsection{Application to the backward heat equation}
The notion of wrong sign Gaussian random variables that we defined can be used to produce solutions to the backward heat equation, just as Gaussian random variables can be used to produce solutions to the heat equation. Recall that the heat equation on $[0,\infty) \times \R^n$ is the partial differential equation  
\[
\partial_t f = \Delta f, \ \  f(0,\cdot) = h(\cdot)
\]
where $f:[0,\infty) \times \R^n \to \R$ is a continuous function which is smooth on $(0,\infty)\times \R^n$, $\partial_t f$ is the partial derivative of $f$ in the first  coordinate (that is, the time coordinate), and $\Delta f$ is the Laplacian of $f$ in the last $n$ coordinates (that is, the space coordinates), and $h:\R^n \to \R$ is an initial condition. It is well known that under mild conditions on $h$, the solution can be expressed as
\[
f(t,x) = \E(h(x + \sqrt{2t} Z)),
\]
where $Z=(Z_1,\ldots,Z_n)$ is a vector of i.i.d.~$N(0,1)$ random variables.

The backward heat equation is 
\[
\partial_t f = -\Delta f, \ \ f(0,\cdot) = h(\cdot).
\]
Unlike the forward heat equation, the backward equation requires far more stringent conditions on $h$ for a solution to exist. The reason is that a solution to the backward equation is just a time-reversed version of the forward equation; and the forward equation causes a function to become instantly analytic. Thus, unless $h$ is real analytic, there cannot exist a solution to the backward equation for any length of time. Moreover, it is known that the solution does not depend   continuously on the initial condition $h$, making it hard to simulate solutions.

The backward heat equation has a sizable body of literature. It is part of the general area of inverse problems, and is sometimes called the `final value problem' for the forward heat equation. Following early results by \citet{yosida59} and \citet{miranker61}, the problem was investigated in depth using the theory of quasi-reversibility developed  by \citet{latteslions69}. It has since been the subject of many investigations. For recent advances and a survey of the literature, see \cite{dongzhang20}.

The following theorem shows that the solution to the backward heat equation may be obtained using wrong sign Gaussian random variables, analogous to the probabilistic solution to the forward equation described above.
\begin{thm}
Let $h:\C^n \to\C$ be a holomorphic function and $T$ be a positive real number such that $|h(z)| = O(e^{\frac{1}{4T}|z|^2})$ as $|z|\to \infty$. 
Let $X\sim N_{0,n}$. Then, for any $x\in \R^n$ and $t\in [0, T)$,  $f(t,x) := \E(h(x + \sqrt{2t}X))$ is well-defined, and the function $f$ solves the backward heat equation in the time interval $[0,T)$ with initial condition $h|_{\R^n}$.
\end{thm}
\begin{proof}
Take any $t\in [0,T)$ and $x\in \R^n$. Define $g:\C^n \to \C$ as 
\[
g(z) := h(x+\sqrt{2t} z). 
\]
If $t=0$, then $g$ is just the constant $h(x)$. Thus, 
\[
f(0,x) = h(x).
\]
Next, suppose that $t>0$. By the Cauchy--Schwarz inequality, we have that for any $z\in \C^n$ and any $\alpha> 0$,
\begin{align*}
|x+\sqrt{2t} z|^2 &\le (|x|+\sqrt{2t} |z|)^2\\
&= (\alpha^{-1}\alpha |x|+ \sqrt{2t}|z|)^2\\
&\le (\alpha^{-2}+1) (\alpha^2 |x|^2 + 2 t|z|^2). 
\end{align*}
Let us now choose $\alpha$ so large that $\gamma:= (\alpha^{-2}+1) t < T$ (which is possible since $t< T$), and let $\beta := 1+\alpha^2$. Then  the above inequality and the hypothesis of the theorem shows that there is a finite constant $C$ such that for all $z\in \C^n$, 
\begin{align*}
|g(z)| &= |h(x+\sqrt{2t} z)|\\
&\le C \exp\biggl(\frac{1}{4T}|x+\sqrt{2t}z|^2\biggr)\\
&\le C\exp\biggl(\frac{\beta}{4T}|x|^2 + \frac{2\gamma}{4T}|z|^2\biggr).
\end{align*}
Since $\gamma<T$, this proves that $\E|g(iZ)|< \infty$, where $Z\sim N_{n,0}$. Thus, $f(t,x) = \E(g(X))$ is well-defined.

Next, fix some $(t,x)\in [0,T)\times \R^n$. The above inequality shows that there is a small enough ball $B$ around $(t,x)$, and some $\beta >0$ and $\gamma \in (0,T)$ such that for all $(s,y)\in B$ with $s\ge 0$, and all $z\in \C^n$,
\begin{align}\label{hsy}
|h(y + \sqrt{2s}z)|\le C\exp\biggl(\frac{\beta}{4T}|x|^2 + \frac{\gamma}{2T}|z|^2\biggr).
\end{align}
By the dominated convergence theorem, this implies that for any sequence $(s_n,y_n)\to (t,x)$, we have
\begin{align*}
f(s_n,y_n) &= \E(h(y_n + i\sqrt{2s_n} Z)) \to \E(h(x+ i \sqrt{2t} Z)) = f(t,x). 
\end{align*}
Thus, $f$ is continuous on $[0,T)\times \R^n$. Since we have already observed that $f(0,x)=h(x)$, it only remains to show that $f$ satisfies the backward heat equation in $(0,T)\times \R^n$. 

For this, take any $(t,x)\in (0,T)\times \R^n$ and let $B$ be a ball as above, with the radius $r$ of the ball so small that the time coordinate is bigger than $\frac{1}{2}t$ everywhere in the ball. Then note that for any $(s,y)\in B$ and $z\in \C^n$, 
\begin{align}\label{delhsy}
\partial_s h(y + \sqrt{2s}z)= \frac{1}{\sqrt{2s}} \sum_{j=1}^n z_j\partial_j h(y+\sqrt{2s} z),
\end{align}
where $\partial_s$ denotes partial derivative with respect to the parameter $s$, and $\partial_j h$ is the partial derivative of $h$ in coordinate $j$. Let $B'$ denote the ball of radius $\frac{1}{2}r$ with center $(t,x)$.   Since $h$ is holomorphic, Cauchy's integral formula for partial derivatives of a holomorphic function in several complex variables gives that for any $w_1,\ldots,w_n \in \C$ and $r_1,\ldots,r_n >0$,
\begin{align*}
&\partial_j h(w_1,\ldots,w_n) \\
&= \frac{1}{(2\pi i)^n} \oint_{C(w_1, r_1)}\cdots\oint_{C(w_n, r_n)}\frac{h(u_1,\ldots,u_n)}{(u_j-w_j)\prod_{k=1}^n(u_k-w_k)} du_n \cdots du_1,
\end{align*}
where $C(w_j, r_j)$ is the circular contour of radius $r_j$ centered at $w_j$, traversed in the counterclockwise direction. From this formula, and the inequality \eqref{hsy}, it follows that there is a constant $C_1$ such that for any $(s,y)\in B'$ and $1\le j\le n$,
\begin{align}\label{delhsy2}
|\partial_j h(y+\sqrt{2s} z)| &\le C_1\exp\biggl(\frac{\beta}{4T}|x|^2 + \frac{\gamma}{2T}|z|^2\biggr).
\end{align}
By the identity \eqref{delhsy}, this allows us to apply the dominated convergence theorem and conclude that 
\begin{align*}
\partial_t f(t,x) &= \E(\partial_t h(x + i\sqrt{2t} Z))= \frac{1}{\sqrt{2t}}\sum_{j=1}^n \E(iZ_j \partial_j h(x+i \sqrt{2t} Z)). 
\end{align*}
Again by the inequality \eqref{delhsy2}, we can apply usual Gaussian integration by parts to get
\[
\E(iZ_j \partial_j h(x+i \sqrt{2t} Z)) = - \E(\partial_j^2 h(x+i\sqrt{2t}Z)). 
\]
Combining the last two displays completes the proof.
\end{proof}

\subsection{Semiclassical approximation}
Suppose we want to `tilt' the $N(0,-1)$ density by a  multiplicative factor of $e^{-F(x)}$ for some function $F$. That is, we want to make sense of `probability densities' proportional to $e^{H(x)}$, where 
\begin{align}\label{tilted}
H(x) := \frac{1}{2}x^2 - F(x).
\end{align}
Let $\langle g \rangle$ denote expected value of a function $g:\R\to \R$ under such a hypothetical probability density. Given our development until now, a natural definition would be
\[
\langle g \rangle := \frac{\E(g(X)e^{-F(X)})}{\E(e^{-F(X)})},
\]
where $X\sim N(0,-1)$. 
For this definition to work, we need that the functions $e^{-F}$ and $ge^{-F}$ are in $\mf_{0,1}$, and further, that the denominator is nonzero. Suppose that these conditions are satisfied, so that the above definition makes sense. A test for whether the definition is `physically meaningful' is whether it converges to the `correct semiclassical limit'. Roughly speaking, this means the following. 

Suppose that instead of $e^{H(x)}$, we have a probability density proportional to $e^{H(x)/b^2}$, 
where $b$ is a positive real number that we would eventually like to send to zero. We want to understand the behavior of the expectation of $g$ with respect to this density. In analogy with the above discussion, this would require making sense of the ratio
\[
\frac{\int g(x)e^{H(x)/b^2} dx }{\int e^{H(x)/b^2} dx},
\]
which, by a `change of variable', is equal to
\[
\frac{\int g(bx)\exp(\frac{1}{2}x^2 - \frac{F(bx)}{b^2}) dx }{\int \exp(\frac{1}{2}x^2 - \frac{F(bx)}{b^2})dx}. 
\]
Thus, we should define the expectation of $g$ with respect to the density proportional to $e^{H(x)/b^2}$ as 
\begin{align}\label{gbdef}
\langle g\rangle_b := \frac{\E(g(bX) e^{-F(bX)/b^2})}{\E(e^{-F(bX)/b^2})},
\end{align}
where $X\sim N(0,-1)$. 
Now, if $e^{H(x)/b^2}$ were a true probability density, then it would concentrate near the global maxima of $H(x)$ as $b\to 0$. In the case of wrong sign distributions, one would similarly expect the `density' to concentrate near the critical points of $H$ as $b\to 0$. The most likely scenario is that $\langle g\rangle_b \to g(x_*)$ as $b\to 0$, where $x_*$ is some distinguished critical point of $H$. This is known as the semiclassical limit. 

An alternative approach is to look at the behavior of the `partition function' of the model, given by $\E(e^{-F(bX)/b^2})$.  In the limit $b\to 0$, this should behave like $e^{H(x_*)/b^2}$ (to leading order) for some critical point $x_*$ of $H$.

A simple example where this works is the following. Let $F(x)=ax$, where $a$ is some given real number. Then $H(x)=\frac{1}{2}x^2 - ax$ has a single critical point, at $x = a$. The following result shows that the model approaches the correct semiclassical limit in both of the senses outlined above. 
\begin{prop}\label{semiex1}
Let $F(x)=ax$ for some $a\in \R$, so that $H(x) = \frac{1}{2}x^2 - ax$ and $a$ is the unique critical point of $H$. Take any $g:\R \to \R$ that has an analytic continuation to $\C$, which satisfies the condition that for some $c>0$, $|g(z)|=O(e^{c|z|^2})$ as $|z|\to\infty$.     Let $\langle g\rangle_b$ be defined as in equation~\eqref{gbdef}. Then $\lim_{b\to 0} \langle g\rangle_b = g(a)$, and $\E(e^{-F(bX)/b^2}) = e^{H(a)/b^2}$ for all $b< (2c)^{-1/2}$.
\end{prop}
\begin{proof}
From the given condition on $g$, it is easy to see that if $b< (2c)^{-1/2}$, then the functions $g(bx)$ and $g(bx)e^{-F(bx)/b^2}$ are in $\mf_{0,1}$. Let us henceforth assume that $b< (2c)^{-1/2}$. Let $X\sim N(0,-1)$ and $Y\sim N(0,1)$. Then  note that
\begin{align*}
\E(e^{-F(bX)/b^2}) &= \E(e^{-F(ibY)/b^2}) = \E(e^{-iabY/b^2}) = e^{-a^2/2b^2} = e^{H(a)/b^2},
\end{align*}
which proves the second claim of the proposition. Next, note that
\begin{align}
\E(g(bX) e^{-F(bX)/b^2}) &= \E(g(ibY) e^{-iabY/b^2})\notag \\
&= \frac{1}{\sqrt{2\pi}}\int_{-\infty}^\infty g(iby)e^{-iay/b} e^{-y^2/2} dy.\label{bcontour}
\end{align}
To evaluate the above integral, we will change the contour of integration. Since $|g(z)|=O(e^{c|z|^2})$ as $|z|\to\infty$ and $b< (2c)^{-1/2}$, we can change the contour of integration in \eqref{bcontour} from $\R$ to $\R + iL$ for any $L\in \R$ without affecting the value of the integral. Then, taking $L = -a/b$, we get
\begin{align*}
\E(g(bX) e^{-F(bX)/b^2}) &= \frac{1}{\sqrt{2\pi}}\int_{-\infty}^\infty g(ib(y-ia/b))e^{-ia(y-ia/b)/b} e^{-(y-ia/b)^2/2} dy\\
&= \frac{e^{-a^2/2b^2} }{\sqrt{2\pi}}\int_{-\infty}^\infty g(iby + a) e^{-y^2/2} dy.
\end{align*}
Thus,
\[
\langle g\rangle_b = \frac{1 }{\sqrt{2\pi}}\int_{-\infty}^\infty g(iby + a) e^{-y^2/2} dy,
\]
which tends to $g(a)$ as $b\to0$, by a simple application of the dominated convergence theorem and the growth rate of $g$ obtained above.
\end{proof}

Our second example has two aims. First, it illustrates how semiclassical approximation can work even if $H$ has multiple critical points. Second, it prepares the ground for semiclassical approximation in timelike Liouville theory in a later section. For this example, consider the function $H:\R\to \R$ defined as 
\begin{align}\label{hxdef}
H(x) := \frac{1}{2}x^2 - e^{\alpha x},
\end{align}
where $\alpha\in (0,1/\sqrt{e})$. Suppose that we want to make sense of the wrong sign probability distribution with density proportional to $e^{H(x)/b^2}$, and investigate its limit as $b\to 0$. The critical points of $H$ are  solutions of the equation
\begin{align*}
x - \alpha e^{\alpha x} = 0.
\end{align*}
Since $h(x) := \alpha e^{\alpha x}$ is a convex function of $x$, there are three possibilities: either this functions remains strictly above the diagonal line $y=x$ for all $x$, in which case $H$ has no critical points; or it touches the diagonal line at exactly one point, in which case $H$ has one critical point; or it goes below the diagonal line at some point $x_0$ and goes above it at some later point $x_1$, and stays above the diagonal for all $x>x_1$. In the last case, $H$ has exactly two critical points, $x_0$ and $x_1$. We claim that if $\alpha\in (0,1/\sqrt{e})$, the last scenario holds, and $0< x_0 < 1/\alpha$. To see this, note that $h(1/\alpha) = \alpha e< 1/\alpha$. This shows that $h$ goes below the diagonal line at some point before $1/\alpha$. Moreover, $h(0) = \alpha > 0$. Thus, $h$ intersects the diagonal line at two points $x_0$ and $x_1$, and the smaller point $x_0$ is strictly less than $1/\alpha$ and strictly bigger than $0$. 

We expect the above wrong sign distribution to concentrate around these critical points. Actually, we will show that it concentrates around the smaller critical point $x_0$. For that, we first have to make sense of the wrong sign distribution with density proportional to $e^{H(x)/b^2}$. 
Fix some $\alpha \in (0,1/\sqrt{e})$ and let $F(x):=e^{\alpha x}$. Choose some function $f:\R\to \R$ that has an analytic continuation to $\C$, satisfying the condition that for some $c>0$, $|f(z)|=O(e^{c|z|^2})$ as $|z|\to\infty$. Take $b< (2c)^{-1/2}$, so that the functions $f(bx)$ and  $f(bx)e^{-F(bx)/b^2}$ are in $\mf_{0,1}$. Then we can define expectation with respect to the wrong sign probability density proportional to $e^{H(x)/b^2}$ as 
\[
\langle f\rangle_b := \frac{\E(f(bX) e^{-F(bX)/b^2})}{\E(e^{-F(bX)/b^2})},
\]
where $X\sim N(0,-1)$. Let $X_b$ be a wrong sign random variable with this distribution, meaning that for any $f$ as above, we define
\[
\E(f(X_b)) :=\langle f\rangle_b.
\]
The following result shows that when $b$ is small, $X_b$ behaves like a random variable that is close to the smaller critical point of $H$. It also establishes that the partition function of the model has the correct semiclassical limit (up to leading order).
\begin{prop}\label{xbthm}
Let $X_b$ be the wrong sign random variable defined above. Let $x_0$ be the smaller critical point of $H$. Take any $f$ as above.  Then 
\[
\lim_{b\to 0} \E(f(X_b)) = f(x_0). 
\]
Moreover, the partition function of the model satisfies
\[
\lim_{b\to 0} \frac{\E(e^{-F(bX)/b^2})}{e^{H(x_0)/b^2}} = \frac{1}{\sqrt{1-\alpha x_0}},
\]
where $X\sim N(0,-1)$.
\end{prop}

\begin{proof}
Let $Z\sim N(0,1)$. Then note that 
\begin{align}
\E(e^{-F(bX)/b^2}) &= \E(e^{-e^{i\alpha bZ}/b^2})\notag \\
&= \int_{-\infty}^\infty \frac{1}{\sqrt{2\pi}} \exp\biggl(-\frac{1}{b^2} e^{i\alpha b x} - \frac{1}{2}x^2\biggr)dx.\label{lambda1}
\end{align}
Take any $\beta\in \R$. From the form of the integrand, it is not hard to justify that the contour of integration can be shifted to the line parallel to the real line that passes through $-i\beta /b$. Thus,
\begin{align*}
\E(e^{-F(bX)/b^2}) &= \int_{-\infty}^\infty \frac{1}{\sqrt{2\pi}} \exp\biggl(-\frac{1}{b^2} e^{i\alpha b (x-i\beta/b)} - \frac{1}{2}(x-i\beta/b)^2\biggr)dx\\
&= \int_{-\infty}^\infty \frac{1}{\sqrt{2\pi}} \exp\biggl(-\frac{e^{\alpha \beta} }{b^2} e^{i\alpha bx} - \frac{1}{2}x^2 + \frac{i\beta}{b} x + \frac{\beta^2}{2b^2}\biggr)dx.
\end{align*}
Let us now choose $\beta=x_0$. Then for any $x$, we have that as $b\to 0$,
\begin{align*}
&\frac{e^{\alpha \beta}}{b^2} -\frac{e^{\alpha \beta} }{b^2} e^{i\alpha bx} - \frac{1}{2}x^2 + \frac{i\beta}{b} x \\
&= \frac{e^{\alpha \beta}}{b^2} -\frac{e^{\alpha \beta} }{b^2} \biggl(1+ i\alpha bx -\frac{1}{2}\alpha^2 b^2 x^2 + o(b^2)\biggr) - \frac{1}{2}x^2 + \frac{i\beta}{b} x\\
&= -\frac{1}{2}(1-\alpha^2e^{\alpha \beta}) x^2 + o(1) = -\frac{1}{2}(1-\alpha \beta)x^2 + o(1).
\end{align*}
Moreover, for any $x$, 
\begin{align*}
\biggl|\exp\biggl(\frac{e^{\alpha \beta}}{b^2} -\frac{e^{\alpha \beta} }{b^2} e^{i\alpha bx} - \frac{1}{2}x^2 + \frac{i\beta}{b} x\biggr)\biggr| &= \exp\biggl(\Re\biggl(\frac{e^{\alpha \beta}}{b^2} -\frac{e^{\alpha \beta} }{b^2} e^{i\alpha bx} - \frac{1}{2}x^2 + \frac{i\beta}{b} x\biggr)\biggr)\\
&= \exp\biggl(\frac{e^{\alpha \beta}}{b^2}(1-\cos(\alpha bx)) - \frac{1}{2}x^2\biggr)\\
&\le \exp\biggl(\frac{e^{\alpha \beta}}{2b^2}\alpha^2 b^2 x^2 - \frac{1}{2}x^2\biggr) \\
&= \exp\biggl(-\frac{1}{2}(1-\alpha \beta) x^2\biggr). 
\end{align*}
Since $\alpha\beta =\alpha x_0< 1$, this allows us to apply to dominated convergence theorem to conclude that 
\begin{align*}
&\lim_{b\to 0} \exp\biggl(\frac{e^{\alpha \beta}}{b^2} -\frac{\beta^2}{2b^2}\biggr) \E(e^{-F(bX)/b^2}) \\
&= \lim_{b\to 0}\int_{-\infty}^\infty \frac{1}{\sqrt{2\pi}} \exp\biggl(\frac{e^{\alpha \beta}}{b^2} -\frac{e^{\alpha \beta} }{b^2} e^{i\alpha bx} - \frac{1}{2}x^2 + \frac{i\beta}{b} x \biggr) dx\\
&= \int_{-\infty}^\infty \lim_{b\to 0} \frac{1}{\sqrt{2\pi}} \exp\biggl(\frac{e^{\alpha \beta}}{b^2} -\frac{e^{\alpha \beta} }{b^2} e^{i\alpha bx} - \frac{1}{2}x^2 + \frac{i\beta}{b} x \biggr) dx\\
&= \int_{-\infty}^\infty\frac{1}{\sqrt{2\pi}} \exp\biggl(-\frac{1}{2}(1-\alpha \beta) x^2\biggr)dx= \frac{1}{\sqrt{1-\alpha \beta}}.
\end{align*}
This proves the claim about the limiting behavior of the partition function. Next, let $g(x) := f(x) e^{-F(x)/b^2}$. Then by the given conditions on $f$ and $b$, we have
\begin{align*}
\E (g(bX)) &= \E(f(ibZ) e^{-e^{i\alpha bZ}/b^2})\\
&= \int_{-\infty}^\infty \frac{1}{\sqrt{2\pi}} f(ibx) \exp\biggl(-\frac{1}{b^2} e^{i\alpha b x} - \frac{1}{2}x^2\biggr)dx.
\end{align*}
Shifting the contour exactly as before (with $\beta=x_0$), we get
\begin{align*}
\E(g(bX)) &= \int_{-\infty}^\infty \frac{1}{\sqrt{2\pi}} f(ib(x-i\beta/b))\exp\biggl(-\frac{1}{b^2} e^{i\alpha b (x-i\beta/b)} - \frac{1}{2}(x-i\beta/b)^2\biggr)dx.
\end{align*}
Proceeding exactly as before,  this leads to
\begin{align*}
\lim_{b\to 0} \exp\biggl(\frac{e^{\alpha \beta}}{b^2} -\frac{\beta^2}{2b^2}\biggr) \E(g(bX)) &= \int_{-\infty}^\infty\frac{1}{\sqrt{2\pi}} f(\beta) \exp\biggl(-\frac{1}{2}(1-\alpha \beta) x^2\biggr)dx\\
&= \frac{f(\beta)}{\sqrt{1-\alpha \beta}}.
\end{align*}
Since $\E(f(X_b)) = \E(g(bX))/\E(e^{-F(bX)/b^2})$, this completes the proof.
\end{proof}

\section{Correlation functions in a special region}
We will now use the theory of wrong sign Gaussian random variables developed in the previous section to derive formulas for the correlation functions in timelike Liouville field theory in a subset of the parameter space.
\subsection{The regularized correlation function}
The following result gives the regularized correlation function of timelike Liouville field theory defined in equation~\eqref{regcor}. 
\begin{thm}\label{tlemma}
Take any $b,\mu, \epsilon > 0$, $L\ge 1$, $\lambda\in (0,1)$, $k\ge 0$,  $x_1,\ldots,x_k\in S^2$ and $\alpha_1,\ldots,\alpha_k\in \C$. Let $w := (Q-\sum_{j=1}^k\alpha_j)/b$. 
Then, in our framework of wrong sign Gaussian random variables, the regularized correlation function defined in equation \eqref{regcor} has the value
\begin{align*}
&\tilde{C}_{\epsilon, \lambda, L}(\alpha_1,\ldots,\alpha_k;x_1,\ldots,x_k;b;\mu )\\
&= \sum_{n=0}^\infty\frac{(-\mu)^n}{n!} \biggl(\prod_{j=1}^ke^{\chi\alpha_j(b-\alpha_j)} g(\sigma(x_j))^{-\Delta_{\alpha_j}}\biggr)\\
&\qquad \cdot \exp\biggl(- \frac{(n-w)^2b^2}{4\pi \epsilon} - \sum_{1\le j<j'\le k} 4\alpha_j\alpha_j' G_{\lambda, L}(x_j, x_{j'})\biggr) \\
&\qquad \cdot \int_{(S^2)^n}\exp\biggl( - \sum_{j=1}^k \sum_{l=1}^n 4\alpha_j bG_{\lambda,L}(x_j, y_l) - \sum_{1\le l<l'\le n} 4b^2 G_{\lambda, L}(y_l,y_{l'})\biggr) da(y_1)\cdots da(y_n),
\end{align*}
where the integral is taken to be $1$ for $n=0$.
\end{thm}
\begin{proof}
Let $Z = (Z_{lm})_{0\le l\le L,\, -l\le m\le l}$ be a collection of i.i.d.~$N(0,1)$ random variables, and define the Gaussian field
\begin{align}\label{zldef}
Z_{\lambda,L}(x) := \sum_{l=1}^L \sum_{m=-l}^l \lambda^{l/2}\sqrt{\frac{2\pi}{l(l+1)}} Z_{lm} Y_{lm}(x)
\end{align}
on $S^2$, and the Gaussian random variable
\[
E := \frac{Z_{00}}{\sqrt{8\pi \epsilon}}. 
\]
Then note that  $\E(Z_{\lambda, L}(x) Z_{\lambda,L}(y)) = G_{\lambda,L}(x,y)$, where $G_{\lambda, L}$ is the regularized Green's function defined in equation \eqref{gldef}. Noticing that 
\[
1+\sum_{l=1}^L (2l+1) = L^2 + 1,
\]
let us index the components of points in $\C^{L^2+1}$ as $z = (z_{lm})_{0\le l\le L, \, -l\le m\le l}$. Define the function $f:\C^{L^2+1}\times S^2 \to \C$ as 
\begin{align}\label{fxydef}
f(z, y) := \sum_{l=1}^L \sum_{m=-l}^l \lambda^{l/2}\sqrt{\frac{2\pi}{l(l+1)}} z_{lm} Y_{lm}(y),
\end{align}
and the function $c:\C^{L^2+1}\to \C$ as 
\[
c(z) := \frac{z_{00}}{\sqrt{8\pi\epsilon}}.
\]
Then define $h:\C^{L^2+1}\to \C$ as
\begin{align*}
h(z) &:=\biggl(\prod_{j=1}^ke^{\chi\alpha_j(b-\alpha_j)} g(\sigma(x_j))^{-\Delta_\alpha} e^{2\alpha_j f(z,x_j) + 2\alpha_j^2G_{\lambda, L}(x_j,x_j)}\biggr)\\
&\qquad \cdot \exp\biggl(-2wbc(z) -\mu e^{2bc(z)} \int_{S^2}e^{2bf(z,y)+2b^2 G_{\lambda, L}(y,y)}da(y) \biggr).
\end{align*}
Then it is clear that the regularized correlation function defined in equation \eqref{regcor} can be expressed as $\E(h(X))$, where $X$ is a wrong sign standard Gaussian random vector in $\R^{L^2+1}$. It is easy to check that $h$ is analytic on $\C^{L^2+1}$ and  $\E|h(i Z)| <\infty$. Thus, the restriction of $h$ to $\R^{L^2+1}$ is in $\mf_{0,L^2+1}$, and by our definition of wrong sign expectations, we have
\[
\E(h(X)) = \E(h(iZ)). 
\]
Now note that $h$ can be expanded as 
\begin{align*}
h(z) &:=\biggl(\prod_{j=1}^ke^{\chi\alpha_j(b-\alpha_j)} g(\sigma(x_j))^{-\Delta_\alpha} e^{2\alpha_j f(z,x_j) + 2\alpha_j^2G_{\lambda, L}(x_j,x_j)}\biggr)\\
&\qquad \cdot \sum_{n=0}^\infty\frac{(-\mu)^n e^{2(n-w)bc(z)}}{n!}\biggl( \int_{S^2}e^{2bf(z,y)+2b^2 G_{\lambda, L}(y,y)}da(y) \biggr)^n.
\end{align*}
This gives
\begin{align*}
h(iZ) &:=\biggl(\prod_{j=1}^ke^{\chi\alpha_j (b-\alpha_j)} g(\sigma(x_j))^{-\Delta_\alpha} e^{2i\alpha_j Z_{\lambda, L}(x_j) + 2\alpha_j^2G_{\lambda, L}(x_j,x_j)}\biggr)\\
&\qquad \cdot \sum_{n=0}^\infty\frac{(-\mu)^n e^{2i(n-w)b E}}{n!}\biggl( \int_{S^2}e^{2ibZ_{\lambda,L}(y)+2b^2 G_{\lambda, L}(y,y)}da(y) \biggr)^n.
\end{align*}
We will show later, in Lemma \ref{gbound}, that $G_{\lambda,L}$ is a bounded function. Using this, it is easy to justify that while evaluating $\E(h(iZ))$, we can take the expectation inside the above infinite sum. A simple calculation completes the proof.
\end{proof}

\subsection{Correlation functions on the sphere}
The following theorem gives a formula for the $k$-point correlation function of timelike Liouville theory when the parameter $w$ is a nonnegative integer. We obtain this formula considering the regularized correlation function $\tilde{C}_{\epsilon, \lambda,L}$ and taking $\epsilon \to 0$, $L\to\infty$ and $\lambda \uparrow 1$, in this order. Note that the limit is finite only because we are integrating the zero mode with respect to a probability measure, which gets more and more spread out as $\ep\to 0$. If, instead, we integrated the zero mode with respect to Lebesgue measure, we would get a delta function prefactor in front of the answer presented below.
\begin{thm}\label{intthm}
Suppose that $k\ge 3$, $\Re(\alpha_j) > -1/2b$ for each $j$, and the parameter  $w = (Q-\sum_{j=1}^k\alpha_j)/b$ is a positive integer. Let $\tilde{C}_{\epsilon, \lambda, L}$ be the regularized correlation function defined in equation \eqref{regcor}. 
Then the limit 
\begin{align*}
\tilde{C}(\alpha_1,\ldots,\alpha_k;x_1,\ldots,x_k;b;\mu) &:= \lim_{\lambda \uparrow 1} \lim_{L\to\infty} \lim_{\epsilon \to 0} \tilde{C}_{\epsilon, \lambda, L}(\alpha_1,\ldots,\alpha_k;x_1,\ldots,x_k;b;\mu) 
\end{align*}
exists, and is equal to 
\begin{align*}
&\frac{e^{-i\pi w} \mu^{w}}{w!}\biggl(\prod_{j=1}^ke^{\chi\alpha_j(b-\alpha_j)} g(\sigma(x_j))^{-\Delta_{\alpha_j}}\biggr)\biggl(\prod_{1\le j < j'\le k} e^{-4\alpha_j\alpha_{j'}G(x_j, x_{j'})} \biggr) \\
&\qquad \qquad \cdot \int_{(S^2)^{w}}\biggl(\prod_{j=1}^k \prod_{l=1}^{w} e^{-4b\alpha_jG(x_j, y_l)}\biggr)\biggl(\prod_{1\le l<l'\le w}e^{-4b^2G(y_l, y_{l'})} \biggr) da(y_1)\cdots da(y_{w}),
\end{align*}
where
\[
G(x,y) = -\ln \|x-y\| - \frac{1}{2}+\ln 2,
\]
with $\|x-y\|$ denoting the Euclidean norm of $x-y$.
\end{thm}
The condition $k\ge 3$ is needed in the above theorem only to ensure that it is not vacuous; indeed, note that if $\Re(\alpha_j) > -1/2b$ for each $j$, then
\[
w < 1-\frac{1}{b^2} +\frac{k}{2b^2}. 
\]
Thus, to simultaneously have $w\ge 1$, we need $k> 2$. 

For simplicity, let us henceforth denote the correlation function by $\tilde{C}_{\epsilon, \lambda, L}$ instead of $\tilde{C}_{\epsilon, \lambda, L}(\alpha_1,\ldots,\alpha_k;x_1,\ldots,x_k;b)$. The first step towards the proof of Theorem \ref{intthm} is the following lemma.
\begin{lmm}\label{step1lmm}
In the setting of Theorem \ref{intthm}, we have
\begin{align*}
&\tilde{C}_{\lambda, L} := \lim_{\epsilon \to 0} \tilde{C}_{\epsilon, \lambda, L} \\
&= \frac{(-\mu)^{w}}{w!}\biggl(\prod_{j=1}^ke^{\chi\alpha_j(b-\alpha_j)} g(\sigma(x_j))^{-\Delta_{\alpha_j}}\biggr)\biggl(\prod_{1\le j < j'\le k} e^{-4\alpha_j\alpha_{j'}G_{\lambda,L}(x_j, x_{j'})} \biggr) \\
& \cdot \int_{(S^2)^{w}}\biggl(\prod_{j=1}^k \prod_{l=1}^{w} e^{-4b\alpha_jG_{\lambda, L}(x_j, y_l)}\biggr)\biggl(\prod_{1\le l<l'\le w}e^{-4b^2G_{\lambda, L}(y_l, y_{l'})} \biggr) da(y_1)\cdots da(y_{w}).
\end{align*}
\end{lmm}
\begin{proof}
If we take the formula for $C_{\epsilon, \lambda, L}$ from Theorem \ref{tlemma} and send $\epsilon$ to zero, only the term corresponding to $n = w$ survives; all other terms tend to zero. The absolute convergence of the series allows us to apply the dominated convergence theorem to obtain the claimed result. 
\end{proof}

It remains to take $L\to \infty$ and $\lambda \to 1$. For this, we need several lemmas about the function $G_{\lambda, L}$, and in particular, its behavior as $L\to\infty$ and $\lambda \to 1$. 
\begin{lmm}\label{gbound}
For any $x,y\in S^2$, $\lambda \in (0,1)$ and $L\ge 1$,
\[
G_{\lambda, L}(x,y) = \sum_{l=1}^L \lambda^l \frac{2l+1}{2l(l+1)} P_l(x\cdot y), 
\]
where $P_l$ denotes the $l^{\text{th}}$ Legendre polynomial. As a consequence, 
\[
|G_{\lambda, L}(x,y)| \le \frac{\lambda}{1-\lambda},
\]
and
\begin{align*}
\lim_{L\to \infty} G_{\lambda, L}(x,y) &= G_\lambda(x,y) := \sum_{l=1}^\infty \lambda^l \frac{2l+1}{2l(l+1)} P_l(x\cdot y).
\end{align*}
\end{lmm}
\begin{proof}
It is well known that the real spherical harmonics satisfy the {\it addition theorem}~\cite[Theorem 2]{muller66} 
\[
P_l(x\cdot y) = \frac{4\pi}{2l + 1}\sum_{m=-l}^l Y_{lm}(x)Y_{lm}(y),
\]
This gives 
\begin{align*}
G_{\lambda, L}(x,y) &= \sum_{l=1}^L \sum_{m=-l}^l \lambda^l \frac{2\pi}{l(l+1)} Y_{lm}(x)Y_{lm}(y) \\
&= \sum_{l=1}^L \lambda^l \frac{2l+1}{2l(l+1)} P_l(x\cdot y),
\end{align*}
proving the first claim. Next, recall the well-known fact~\cite[Equation (7.21.1)]{szego39} that 
\begin{align}\label{plbound0}
|P_l(x)|\le 1 \text{ for all $l\ge 0$ and $x\in [-1,1]$}.
\end{align}
Using this bound in the above expression proves the second and third claims of the lemma.
\end{proof}

We are now ready to send $L$ to infinity.
\begin{lmm}\label{step2lmm}
In the setting of Theorem \ref{intthm}, we have
\begin{align*}
&\tilde{C}_\lambda := \lim_{L\to\infty}  \tilde{C}_{\lambda, L} \\
&= \frac{(-\mu)^{w}}{w!}\biggl(\prod_{j=1}^ke^{\chi\alpha_j(b-\alpha_j)} g(\sigma(x_j))^{-\Delta_{\alpha_j}}\biggr)\biggl(\prod_{1\le j < j'\le k} e^{-4\alpha_j\alpha_{j'}G_\lambda(x_j, x_{j'})} \biggr) \\
&\qquad \cdot \int_{(S^2)^{w}}\biggl(\prod_{j=1}^k \prod_{l=1}^{w} e^{-4b\alpha_jG_\lambda(x_j, y_l)}\biggr)\biggl(\prod_{1\le l<l'\le w}e^{-4b^2G_\lambda(y_l, y_{l'})} \biggr) da(y_1)\cdots da(y_{w}).
\end{align*}
\end{lmm}
\begin{proof}
This is a simple consequence of Lemma \ref{step1lmm}, Lemma \ref{gbound}, and the dominated convergence theorem.
\end{proof}

Our final step is to take $\lambda$ to $1$. Unlike $G_{\lambda, L}$, the absolute value of the  function $G_\lambda$ is not uniformly bounded over all $\lambda \in (0,1)$. Luckily, it is uniformly bounded below, as shown by the following lemma. This will suffice for us.

\begin{lmm}\label{glambdalow}
We have that 
\[
\inf_{\lambda \in (0,1), \, x,y\in S^2} G_\lambda(x,y) \ge -\frac{9}{8}.
\]
\end{lmm}
\begin{proof}
The generating function of Legendre polynomials is given by~\cite[Equation (4.7.23) with $\lambda = 1/2$]{szego39}:
\begin{align}\label{plgen}
\sum_{l=0}^\infty t^l P_l(x) = \frac{1}{\sqrt{1-2xt + t^2}},
\end{align}
where the series converges absolutely when $|x|\le 1$ and $|t|<1$. Using the bound \eqref{plbound0}, we can integrate both sides with respect to $t$ from $0$ to $\lambda$, and move the integral inside the sum to get
\begin{align*}
\sum_{l=0}^\infty \frac{\lambda^{l+1}P_l(x)}{l+1} = \int_0^\lambda \frac{1}{\sqrt{1-2xt + t^2}}dt\ge 0.
\end{align*}
Consequently,
\begin{align}\label{pl1}
\sum_{l=1}^\infty \frac{\lambda^{l}P_l(x)}{l+1} &\ge -1.
\end{align}
Next, note that by equation \eqref{plgen}, we have that for any $t\in (0,1)$,
\begin{align}\label{tl1}
\sum_{l=1}^\infty t^{l-1} P_l(x) &= \frac{1}{t}\biggl(\frac{1}{\sqrt{1-2xt + t^2}} - 1\biggr). 
\end{align}
Now, for any $y\in \R$, 
\begin{align*}
(1-y)\biggl(1+\frac{y}{2}\biggr)^2 &= (1-y)\biggl(1+y+\frac{ y^2}{4}\biggr) \\
&= 1-\frac{3y^2}{4} - \frac{y^3}{4} = 1 -\frac{3y^2 + y^3}{4}.
\end{align*}
This shows, in particular, that for all $y\ge -3$, 
\[
(1-y)\biggl(1+\frac{y}{2}\biggr)^2 \le 1.
\]
Thus, for $y> -3$, we can take square root of both sides to get
\[
\frac{1}{\sqrt{1-y}} \ge 1+\frac{y}{2}. 
\]
Take any $x,t$ such that $|x|\le 1$ and $|t|< 1$. Then applying the above inequality with $y = 2xt - t^2> -3$, we get
\[
\frac{1}{\sqrt{1-2xt+t^2}} \ge 1+xt - \frac{t^2}{2}.
\]
Using this in equation \eqref{tl1},  and integrating $t$ from $0$ to $\lambda$, we get
\begin{align}\label{pl2}
\sum_{l=1}^\infty \frac{\lambda^l P_l(x)}{l} &\ge \int_0^\lambda \biggl(x - \frac{t}{2}\biggr) dt =x\lambda - \frac{\lambda^2}{4} \ge -\frac{5}{4}. 
\end{align}
Combining equations \eqref{pl1} and \eqref{pl2}, we get
\begin{align*}
\sum_{l=1}^\infty \lambda^l \frac{2l+1}{2l(l+1)} P_l(x) &= \sum_{l=1}^\infty \frac{\lambda^l P_l(x)}{2(l+1)} + \sum_{l=1}^\infty \frac{\lambda^l P_l(x)}{2l}\\
&\ge -\frac{1}{2} - \frac{5}{8} = -\frac{9}{8}. 
\end{align*}
By the formula for $G_\lambda$ from Lemma \ref{gbound}, this completes the proof.
\end{proof}
The next lemma gives the explicit formula for the limit of $G_\lambda(x,y)$ as $\lambda\uparrow 1$.
\begin{lmm}\label{glimit}
For distinct $x,y\in S^2$,
\[
\lim_{\lambda \uparrow 1} G_\lambda(x,y) = G(x,y) = -\ln \|x-y\| - \frac{1}{2}+\ln 2,
\]
and $G$ is the Green's function for the spherical Laplacian defined in equation \eqref{greendef}.
\end{lmm}
\begin{proof}
By \cite[Theorem 7.3.3]{szego39}, 
\begin{align*}
|P_l(\cos \theta)| \le \sqrt{\frac{2}{\pi l \sin \theta}}
\end{align*}
for all $\theta\in (0,\pi)$. This allows us to show, via the dominated convergence theorem, that for any distinct $x,y\in S^2$, 
\begin{align}\label{gform}
G(x,y) := \lim_{\lambda \uparrow 1} G_\lambda(x,y) = \sum_{l=1}^\infty \frac{2l+1}{2l(l+1)}P_l(x\cdot y),
\end{align}
where the series on the right converges absolutely. By the addition theorem for spherical harmonics, it follows that this is indeed the Green's function for the spherical Laplacian. It remains to prove the explicit formula for $G$. Recall that the Legendre polynomials form an orthonormal basis of $L^2_{\R}([-1,1])$, and any $f\in L^2_{\R}([-1,1])$ admits an expansion
\[
f(x) = \sum_{l=0}^\infty a_l P_l(x),
\] 
where
\[
a_l = \frac{2l+1}{2}\int_{-1}^1 f(x) P_l(x) dx,
\]
and the sum converges in $L^2$~\cite{newmanrudin52}. Let $f(x) = \ln(1-x)$ and $a_l$ be as above. Recall that for $l\ge 1$, $P_l$ satisfies Legendre's differential equation
\[
(1-x^2)P_l''(x) -2xP_l'(x) + l(l+1)P_l(x)=0.
\]
Thus,
\begin{align*}
a_l &= \frac{2l+1}{2}\int_{-1}^1\ln(1-x) P_l(x) dx\\
&= -\frac{2l+1}{2l(l+1)}\int_{-1}^1((1-x^2) \ln(1-x)P_l''(x)-2x\ln(1-x)P_l'(x)) dx.
\end{align*}
Now, integration by parts (and the fact that $P_l$ is a polynomial) shows that 
\begin{align*}
\int_{-1}^1(1-x^2) \ln(1-x)P_l''(x)dx &= -\int_{-1}^1 P_l'(x) \frac{d}{dx}((1-x^2)\ln (1-x)) dx\\
&= -\int_{-1}^1P_l'(x) (-2x\ln (1-x) -(1+x)) dx. 
\end{align*}
Plugging this into the previous display, we get
\begin{align*}
a_l &= -\frac{2l+1}{2l(l+1)}\int_{-1}^1(1+x)P_l'(x)dx.
\end{align*}
Applying integration by parts once again, we have 
\begin{align*}
\int_{-1}^1(1+x)P_l'(x)dx &= 2P_l(1) - \int_{-1}^1 P_l(x) dx = 2. 
\end{align*}
Thus, for any $l\ge 1$,
\[
a_l = -\frac{2l+1}{l(l+1)}.
\]
By direct calculation,
\[
a_0 = \frac{1}{2}\int_{-1}^1 \ln(1-x) dx = \ln 2 -1.
\]
Thus, by the formula \eqref{gform}, we get
\begin{align*}
G(x,y) &= \frac{1}{2}\sum_{l=1}^\infty a_l P_l(x\cdot y) = -\frac{1}{2}(\ln(1-x\cdot y) + 1 - \ln 2).
\end{align*}
But 
\[
\|x-y\|^2 = \|x\|^2+\|y\|^2-2x\cdot y = 2(1-x\cdot y). 
\]
Plugging this into the previous display shows that it is equal to 
\begin{align*}
 -\frac{1}{2}(\ln (\|x-y\|^2/2) + 1- \ln 2) = -\ln \|x-y\|-\frac{1}{2} + \ln 2.
\end{align*}
Since the series in equation~\eqref{gform} is absolutely convergent, this completes the proof.
\end{proof}

We need one more ingredient to complete the proof of Theorem \ref{intthm}.
\begin{lmm}\label{glambdaup}
For any distinct $x,y\in S^2$ and $\lambda \in (\frac{1}{2},1)$,
\[
G(x,y) - \lambda G_\lambda(x,y) \ge -\frac{13}{8}. 
\]
\end{lmm}
\begin{proof}
By the bound \eqref{plbound0}, it is easy to justify that $G_\lambda(x,y)$ can be differentiated with respect to $\lambda$ by moving the derivative inside the infinite sum, giving
\begin{align*}
\frac{\partial }{\partial \lambda}(\lambda G_\lambda(x,y) ) &= \sum_{l=1}^\infty \frac{\partial}{\partial \lambda} \biggl(\lambda^{l+1}\frac{2l+1}{2l(l+1)} P_l(x\cdot y)\biggr)\\
&= \sum_{l=1}^\infty \lambda^{l}\frac{2l+1}{2l} P_l(x\cdot y)\\
&= \sum_{l=1}^\infty \lambda^{l}P_l(x\cdot y) + \sum_{l=1}^\infty\frac{ \lambda^{l}}{2l} P_l(x\cdot y).
\end{align*}
By equation \eqref{plgen} and the inequality \eqref{pl2}, the above quantity is bounded below by $-1-\frac{5}{8}$. This proves the claim.
\end{proof}

We are now ready to complete the proof of Theorem \ref{intthm}.
\begin{proof}[Proof of Theorem \ref{intthm}]
By Lemma \ref{step2lmm} and Lemma \ref{glimit}, we only need to show that as we take $\lambda \uparrow 1$, the $G_\lambda$'s in the integral displayed in the statement of Lemma \ref{step2lmm} can be replaced by $G$. We already know that $G_\lambda$ converges to $G$ pointwise. So it remains to verify the condition needed to apply the dominated convergence theorem. Recall that the integrand is
\begin{align}\label{integrand}
\biggl(\prod_{j=1}^k \prod_{l=1}^{w} e^{-4b\alpha_jG_\lambda(x_j, y_l)}\biggr) \biggl(\prod_{1\le l<l'\le w}e^{-4b^2G_\lambda(y_l, y_{l'})} \biggr). 
\end{align}
Let $S$ be the set of $j$ such that $\Re(\alpha_j) < 0$. Take any $j\notin S$. Then by Lemma \ref{glambdalow}, 
\[
|e^{-4b\alpha_jG_\lambda(x_j, y_l)}|\le e^{\frac{9}{2}b\Re(\alpha_j)}. 
\]
Similarly, for all $l,l'$,
\[
e^{-4b^2G_\lambda(y_l, y_{l'})} \le e^{\frac{9}{2}b^2}.
\]
Next, take any $j\in S$. Then by Lemma \ref{glambdaup},
\begin{align*}
|e^{-4b\alpha_jG_\lambda(x_j, y_l)}| &\le e^{-4b\Re(\alpha_j)(\frac{13}{8\lambda} + \frac{1}{\lambda}G(x_j,y_l))}.
\end{align*}
Now recall that by assumption, $\Re(\alpha_j) > -1/2b$ for each $j$. This implies that for each $j\in S$, $-4b\Re(\alpha_j )< 2$. Consequently, we can find $\lambda_0\in (0,1)$ so close to $1$ that 
\[
\kappa := \max_{j\in S} \frac{-4b\Re(\alpha_j)}{\lambda_0} < 2. 
\]
Thus, if $\lambda \in (\lambda_0, 1)$, then for each $j\in S$,
\[
e^{-4b\Re(\alpha_j)G_\lambda(x_j, y_l)} \le e^{\kappa(\frac{13}{8} + G(x_j, y_l))}. 
\]
Combining all of the above, we see that if $\lambda \in (\lambda_0,1)$, then the (nonnegative) quantity displayed in equation \eqref{integrand} is bounded above by
\begin{align}\label{integrandbd}
C \prod_{j\in S} \prod_{l=1}^w e^{\kappa G(x_j, y_l)},
\end{align}
where $C$ has no dependence on $y_1,\ldots,y_{w}$. Since $\kappa < 2$ and $x_1,\ldots,x_k$ are distinct, the above function has a finite integral over $(S^2)^{w}$. This allows us to apply the dominated convergence theorem to complete the proof of Theorem \ref{intthm}. Observe that the $(-1)^w$ term can be written as $e^{-i\pi w}$ since $w$ is an integer.
\end{proof}

\subsection{Correlation functions on the plane}\label{planecorsec}
The following theorem, quoted from Subsection \ref{mainresultssec}, is the same as Theorem \ref{intthm}, except that the integrals are expressed differently as integrals with respect to Lebesgue measure on $\C$.

\intthmsecond*

For the proof, we need the following lemma.
\begin{lmm}\label{stereoform}
If $x,y\in S^2\setminus\{e_3\}$ and $u,v$ are the stereographic projections of $x,y$ on $\C$, then 
\[
\|x-y\|^2 = \frac{4|u-v|^2}{(1+|u|^2)(1+|v|^2)}.
\]
\end{lmm}
\begin{proof}
Take any $x=(x_1,x_2,x_3)\in S^2\setminus\{e_3\}$, and let $u=u_1+i u_2\in\C$ be its stereographic projection on $\C$. Then the following relations hold:
\[
u =\frac{x_1+i x_2}{1-x_3}, \ \ x = \frac{1}{1+|u|^2}(2u_1, 2u_2, |u|^2-1).
\]
Thus, if $x,y\in S^2\setminus\{e_3\}$ and $u,v$ are the stereographic projections of $x,y$, then
\begin{align*}
|u-v|^2 &= \biggl(\frac{x_1}{1-x_3} - \frac{y_1}{1-y_3}\biggr)^2 + \biggl(\frac{x_2}{1-x_3} - \frac{y_2}{1-y_3}\biggr)^2\\
&= \frac{x_1^2(1-y_3)^2 + y_1^2(1-x_3)^2-2x_1y_1(1-x_3)(1-y_3)}{(1-x_3)^2(1-y_3)^2}\\
&\qquad + \frac{x_2^2(1-y_3)^2 + y_2^2(1-x_3)^2-2x_2y_2(1-x_3)(1-y_3)}{(1-x_3)^2(1-y_3)^2}.
\end{align*}
Since $\|x\|^2=\|y\|^2=1$, we have
\begin{align*}
&x_1^2(1-y_3)^2 + y_1^2(1-x_3)^2 + x_2^2(1-y_3)^2 + y_2^2(1-x_3)^2 \\
&= (1-x_3^2)(1-y_3)^2 + (1-y_3^2)(1-x_3)^2\\
&= (1-x_3)(1-y_3)((1+x_3)(1-y_3)+ (1+y_3)(1-x_3))\\
&= 2(1-x_3)(1-y_3)(1-x_3y_3). 
\end{align*}
Using this in the previous display, we get
\[
|u-v|^2 = \frac{2(1-x\cdot y)}{(1-x_3)(1-y_3)} = \frac{\|x-y\|^2}{(1-x_3)(1-y_3)}. 
\]
Now note that
\begin{align*}
1+|u|^2 &= 1+\frac{x_1^2+x_2^2}{(1-x_3)^2} \\
&=1+ \frac{1-x_3^2}{(1-x_3)^2}= \frac{2-2x_3}{(1-x_3)^2} = \frac{2}{1-x_3}.
\end{align*}
Similarly,
\[
1+|v|^2 = \frac{2}{1-y_3}. 
\]
Combining the last three displays completes the proof.
\end{proof}
We are now ready to complete the proof of Theorem \ref{intthmv2}.
\begin{proof}[Proof of Theorem \ref{intthmv2}]
Note that by Lemma \ref{stereoform}, for any $x,y\in S^2\setminus\{e_3\}$, 
\begin{align}
G(x,y) &= -\ln |\sigma(x)-\sigma(y)| + \frac{1}{2}\ln (1+|\sigma(x)|^2) \notag \\
&\qquad \qquad +\frac{1}{2}\ln (1+|\sigma(y)|^2) -\frac{1}{2}. \label{gpform}
\end{align}
Take any $y_1,\ldots, y_w\in S^2\setminus\{e_3\}$. Then the above identity shows that
\begin{align*}
-4b^2 \sum_{1\le l<l'\le w} G(y_l, y_{l'}) &= 4b^2 \sum_{1\le l<l'\le w} \ln |\sigma(y_l)-\sigma(y_{l'})|\\
&\qquad  - 2b^2(w-1)\sum_{l=1}^w \ln (1+|\sigma(y_l)|^2) + b^2w(w-1). 
\end{align*}
Similarly,
\begin{align*}
-4b\sum_{j=1}^k \sum_{l=1}^w\alpha_j G(x_j, y_l) &= 4b\sum_{j=1}^k \sum_{l=1}^w \alpha_j \ln |\sigma(x_j)-\sigma(y_l)|\\
&\qquad - 2bw\sum_{j=1}^k \alpha_j \ln(1+|\sigma(x_j)|^2) \\
&\qquad - 2b\biggl(\sum_{j=1}^k\alpha_j\biggr)\sum_{l=1}^w\ln (1+|\sigma(y_l)|^2) +2bw\sum_{j=1}^k\alpha_j.
\end{align*}
Now note that the coefficients of each $\ln(1+|\sigma(y_l)|^2)$ from the above two displays add up to 
\begin{align*}
-2b^2(w-1) - 2b\sum_{j=1}^k \alpha_j &= -2b\biggl(Q-\sum_{j=1}^k \alpha_j\biggr) +2b^2 -2b\sum_{j=1}^k \alpha_j\\
&= -2bQ + 2b^2 = 2.
\end{align*}
Lastly, note that the constant terms add up to
\begin{align*}
b^2w(w-1) + 2bw\sum_{j=1}^k\alpha_j &= b^2w^2 - b^2w + 2bw(Q-bw) \\
&=-b^2w^2 -b^2 w + 2bwQ\\
&= -b^2 w^2 + b^2 w -2w.
\end{align*}
Thus, combining the last four displays, we get
\begin{align*}
&-4b^2 \sum_{1\le l<l'\le w} G(y_l, y_{l'})  -4b\sum_{j=1}^k \sum_{l=1}^w\alpha_j G(x_j, y_l)\\
&= 4b^2 \sum_{1\le l<l'\le w} \ln |\sigma(y_l)-\sigma(y_{l'})| + 4b\sum_{j=1}^k \sum_{l=1}^w \alpha_j \ln |\sigma(x_j)-\sigma(y_l)|\\
&\qquad +2 \sum_{l=1}^w \ln (1+|\sigma(y_l)|^2) - 2bw\sum_{j=1}^k\alpha_j \ln(1+|\sigma(x_j)|^2)\\
&\qquad -b^2 w^2 + b^2 w -2w.
\end{align*}
It is well known that for any $f:S^2\to\C$ that is integrable with respect to the area measure on $S^2$, we have
\begin{align}\label{change}
\int_{S^2} f(y) da(y) = \int_{\C} f(\sigma^{-1}(z)) \frac{4}{(1+|z|^2)^2} d^2z.
\end{align}
Combining this with the previous display, and recalling the $u_j := P(x_j)$, we get
\begin{align*}
&\int_{(S^2)^{w}}\biggl(\prod_{j=1}^k \prod_{l=1}^{w} e^{-4b\alpha_jG(x_j, y_l)}\biggr) \biggl(\prod_{1\le l<l'\le w}e^{-4b^2G(y_l, y_{l'})} \biggr) da(y_1)\cdots da(y_{w})\\
&= 4^we^{-b^2 w^2 + b^2 w -2w}\biggl(\prod_{j=1}^k (1+|z_j|^2)^{-2bw\alpha_j}\biggr)\\
&\qquad \qquad \cdot  \int_{\C^w} \biggl(\prod_{j=1}^k \prod_{l=1}^{w} |z_j-t_l|^{4b\alpha_j}\biggr)\biggl(\prod_{1\le l<l'\le w}\|t_l-t_{l'}\|^{4b^2} \biggr) d^2t_1\cdots d^2t_w.
\end{align*}
Again, by equation \eqref{gpform}, 
\begin{align*}
&\prod_{1\le j < j'\le k} e^{-4\alpha_j\alpha_{j'}G(x_j, x_{j'})}  = \prod_{1\le j<j'\le k} \biggl(\frac{\sqrt{e}|z_j - z_{j'}|}{\sqrt{(1+|z_j|^2) (1+|z_{j'}|^2)}}\biggr)^{4\alpha_j \alpha_{j'}}\\
&= e^{2\sum_{1\le j<j'\le k} \alpha_j \alpha_{j'}} \prod_{j=1}^k (1+|z_j|^2)^{-2\alpha_j \sum_{j'\ne j} \alpha_{j'}}\prod_{1\le j<j'\le k} |z_j -z_{j'}|^{4\alpha_j \alpha_{j'}}. 
\end{align*}
Next, note that 
\begin{align*}
\prod_{j=1}^k g(\sigma(x_j))^{-\Delta_{\alpha_j}} &=  4^{-\sum_{j=1}^k \Delta_{\alpha_j}} \prod_{j=1}^k (1+|z_j|^2)^{2\Delta_{\alpha_j}}
\end{align*}
We have to now multiply the expressions from the last three displays. Upon multiplying them together, the exponent of $1+|z_j|^2$ becomes
\begin{align*}
-2bw\alpha_j -2\alpha_j\sum_{j'\ne j }\alpha_{j'}+2\Delta_{\alpha_j} &= -2\biggl(Q-\sum_{j'=1}^k \alpha_j\biggr) \alpha_j -2\alpha_j\sum_{j'\ne j }\alpha_{j'}+2\alpha_j(Q-\alpha_j)\\
&= 0. 
\end{align*}
The exponent of $4$ becomes
\begin{align*}
w  - \sum_{j=1}^k \Delta_{\alpha_j} &= 1 - \frac{1}{b^2} - \frac{1}{b}\sum_{j=1}^k\alpha_j - \sum_{j=1}^k \alpha_j \biggl(b-\frac{1}{b} - \alpha_j\biggr)\\
&= 1-\frac{1}{b^2} -b \sum_{j=1}^k \alpha_j +\sum_{j=1}^k \alpha_j^2.
\end{align*}
Lastly, the exponent of $e$ is 
\begin{align*}
&-b^2w^2 +b^2w - 2w + 2\sum_{1\le j<j'\le k} \alpha_j \alpha_j' \\
&= -\biggl(Q-\sum_{j=1}^k\alpha_j\biggr)^2 + \biggl(b-\frac{2}{b}\biggr)\biggl(Q-\sum_{j=1}^k\alpha_j\biggr)+ 2\sum_{1\le j<j'\le k} \alpha_j \alpha_j'\\
&= -Q^2 +\biggl(2Q- b + \frac{2}{b}\biggr)\sum_{j=1}^k\alpha_j + Q\biggl(b-\frac{2}{b}\biggr) - \sum_{j=1}^k\alpha_j^2\\
&= \frac{1}{b^2}-1 + b \sum_{j=1}^k \alpha_j - \sum_{j=1}^k \alpha_j^2.
\end{align*}
Note that this is exactly the negative of the exponent of $4$. Thus, the constant factor in the combined expression is equal to 
\begin{align*}
\exp\biggl(\biggl(1-\frac{1}{b^2} -b \sum_{j=1}^k \alpha_j +\sum_{j=1}^k \alpha_j^2\biggr)(\ln 4 - 1)\biggr) = e^{\chi(1-1/b^2)} e^{-\chi\sum_{j=1}^k \alpha_j (b-\alpha_j)}.
\end{align*}
Combining the above calculations yields the claimed result.
\end{proof}

\subsection{The timelike DOZZ formula}\label{dozzsec}
Let us now specialize the correlation function on the plane to the case $k=3$, $z_1 = 0$, $z_2= 1$, and $|z_3|\to \infty$. It turns out that we can evaluate this explicitly, after normalizing by a suitable factor depending on $z_3$ as $|z_3|\to \infty$. For the formula, we need the following special function $\Upsilon_b$ introduced by \citet{dornotto94}:
\begin{align*}
\Upsilon_b(x) := \exp\biggl(\int_0^\infty \frac{1}{\tau}\biggl(\biggl(\frac{b}{2} + \frac{1}{2b} - x\biggr)^2 e^{-\tau} - \frac{\sinh^2((\frac{b}{2}+\frac{1}{2b} - x)\frac{\tau}{2})}{\sinh(\frac{b\tau}{2}) \sinh(\frac{\tau}{2b})}\biggr)d\tau \biggr)
\end{align*}
on the strip $\{x\in  \C: 0< \Re(x)<b + \frac{1}{b}\}$ and continued analytically to the whole plane. 
Also, let $\gamma(x) := \Gamma(x)/\Gamma(1-x)$, where $\Gamma$ is the classical Gamma function. The following theorem gives a formula for the $3$-point correlation which is the same as the one displayed in \citet{harlowetal11} (after the notational changes $\hat{Q} \to -Q$, $\hat{\alpha}_j \to -\alpha_j$ and $\hat{b} \to b$), as well the ones appearing in the original proposals of \citet{schomerus03}, \citet{zamolodchikov05}, and \citet{kostovpetkova06, kostovpetkova07, kostovpetkova07a}. The only difference is that there is an additional factor depending only on $b$; but that is not a problem since the correlation function is supposed to be unique only up to a $b$-dependent factor.
\begin{thm}\label{threepoint}
Suppose that $\Re(\alpha_j)> -1/2b$ for $j=1,2,3$, and that $w = (Q-\sum_{j=1}^3 \alpha_j)/b$ is a positive integer. Then $\Upsilon_b(b- 2\alpha_j)\ne 0$ for $j=1,2,3$, and 
\begin{align*}
&C(\alpha_1,\alpha_2,\alpha_3; b;\mu) := \lim_{|z_3|\to \infty} |z_3|^{-4\Delta_{\alpha_3}}C(\alpha_1,\alpha_2,\alpha_3;0,1,z_3;b;\mu) \\
&= e^{-i\pi w} (-\pi \mu\gamma(-b^2))^w (4/e)^{1-1/b^2}b^{2b^2w + 2w} \frac{\Upsilon_b(Q - \alpha_1-\alpha_2-\alpha_3+b)}{\Upsilon_b(b)}\\
&\qquad \cdot \frac{\Upsilon_b(\alpha_1-\alpha_2-\alpha_3+b)\Upsilon_b(\alpha_2-\alpha_1-\alpha_3+b)\Upsilon_b(\alpha_3-\alpha_1-\alpha_2+b)}{\Upsilon_b(b - 2\alpha_1)\Upsilon_b(b-2\alpha_2) \Upsilon_b(b-2\alpha_3)}.
\end{align*}
\end{thm}
In the proof below, we  follow the calculations of \citet{giribet12} who derived the timelike DOZZ formula starting from a hypothetical Coulomb gas representation of the $3$-point function. Our starting point is the formula for the $k$-point correlation function given in Theorem \ref{intthmv2}, specialized to $k=3$, which yields the following lemma.
\begin{lmm}\label{threelmm}
Under the assumptions of Theorem \ref{threepoint},
\begin{align*}
&\lim_{|z_3|\to \infty} |z_3|^{-4\Delta_{\alpha_3}}C(\alpha_1,\alpha_2,\alpha_3;0,1,z_3;b;\mu) \\
&= \frac{e^{-i\pi w}\mu^w}{w!} (4/e)^{1-1/b^2} \int_{\C^w} \biggl(\prod_{l=1}^w |t_l|^{4b\alpha_1}|1-t_l|^{4b\alpha_2}\biggr)\biggl(\prod_{1\le l<l'\le w}|t_l-t_{l'}|^{4b^2} \biggr) d^2t_1\cdots d^2t_w.
\end{align*}
\end{lmm}
\begin{proof}
Since $\alpha_j> -1/2b$ for $j=1,2,3$ and $w$ is a positive integer, Theorem \ref{intthmv2} gives 
\begin{align*}
&C(\alpha_1,\alpha_2,\alpha_3;0,1,z_3;b;\mu) \\
&= \frac{e^{-i\pi w}\mu^{w}}{w!} (4/e)^{1-1/b^2} |z_3|^{4\alpha_1\alpha_3} |1-z_3|^{4\alpha_2 \alpha_3} \\
&\qquad \cdot \int_{\C^w} \biggl(\prod_{l=1}^w |t_l|^{4b\alpha_1}|1-t_l|^{4b\alpha_2}|z_3 - t_l|^{4b \alpha_3}\biggr)\biggl(\prod_{1\le l<l'\le w}|t_l-t_{l'}|^{4b^2} \biggr) d^2t_1\cdots d^2t_w.
\end{align*}
The powers of the terms involving $z_3$ add up to
\begin{align*}
4\alpha_1\alpha_3 + 4\alpha_2 \alpha_3 + 4bw\alpha_3 &= 4(\alpha_1+\alpha_2 + (Q-\alpha_1-\alpha_2-\alpha_3)) \alpha_3\\
&= 4\alpha_3(Q-\alpha_3) = 4\Delta_{\alpha_3}. 
\end{align*}
This shows that it suffices to prove the following:
\begin{align}\label{newlim}
&\lim_{|z_3|\to \infty} \int_{\C^w} \biggl(\prod_{l=1}^w |t_l|^{4b\alpha_1}|1-t_l|^{4b\alpha_2}\frac{|z_3 - t_l|^{4b \alpha_3}}{|z_3|^{4b\alpha_3}}\biggr)\biggl(\prod_{1\le l<l'\le w}|t_l-t_{l'}|^{4b^2} \biggr) d^2t_1\cdots d^2t_w\notag \\
&= \int_{\C^w} \biggl(\prod_{l=1}^w |t_l|^{4b\alpha_1}|1-t_l|^{4b\alpha_2}\biggr)\biggl(\prod_{1\le l<l'\le w}|t_l-t_{l'}|^{4b^2} \biggr) d^2t_1\cdots d^2t_w.
\end{align}
Let $\nu$ be the probability measure on $\C$ with density function $\frac{1}{4\pi}g(z)$ with respect to Lebesgue measure, where $g$ is the round metric. Then note that the integral in the first line of equation~\eqref{newlim} can be written as 
\[
\int_{\C^w} f_{z_3}(t_1,\ldots,t_w) d\nu(t_1)\cdots d\nu(t_w),
\]
where 
\begin{align*}
f_{z_3}(t_1,\ldots,t_w) := \biggl(\prod_{l=1}^w \frac{4\pi}{g(t_l)}\biggr) \biggl(\prod_{l=1}^w |t_l|^{4b\alpha_1}|1-t_l|^{4b\alpha_2}\frac{|z_3 - t_l|^{4b \alpha_3}}{|z_3|^{4b\alpha_3}}\biggr)\biggl(\prod_{1\le l<l'\le w}|t_l-t_{l'}|^{4b^2} \biggr). 
\end{align*}
The integral in the second line in equation \eqref{newlim} can be written as
\[
\int_{\C^w} f(t_1,\ldots,t_w) d\nu(t_1)\cdots d\nu(t_w),
\]
where 
\begin{align*}
f(t_1,\ldots,t_w) := \biggl(\prod_{l=1}^w \frac{4\pi }{g(t_l)}\biggr) \biggl(\prod_{l=1}^w |t_l|^{4b\alpha_1}|1-t_l|^{4b\alpha_2}\biggr)\biggl(\prod_{1\le l<l'\le w}|t_l-t_{l'}|^{4b^2} \biggr). 
\end{align*}
Clearly, $f_{z_3}\to f$ pointwise as $|z_3|\to \infty$. We need to show that the integral of $f_{z_3}$ with respect to $\nu^w$ converges to the integral of $f$ with respect to $\nu^w$. Since $\nu^w$ is a probability measure on $\C^w$, well-known results from probability theory imply that it is sufficient to prove that for some $\delta>0$, 
\begin{align}\label{uicond}
\limsup_{|z_3|\to \infty} \int_{\C^w} |f_{z_3}(t_1,\ldots,t_w)|^{1+\delta} d\nu(t_1)\cdots d\nu(t_w) <\infty.
\end{align} 
Let us fix some $\delta$, to be chosen later. By the change of variable formula from equation~\eqref{change}, we have
\begin{align*}
&\int_{\C^w} |f_{z_3}(t_1,\ldots,t_w)|^{1+\delta} d\nu(t_1)\cdots d\nu(t_w) \\
&= \frac{1}{(4\pi)^w}\int_{(S^2)^w} |f_{z_3}(\sigma(y_1),\ldots,\sigma(y_w))|^{1+\delta} da(y_1)\cdots da(y_w).
\end{align*}
Now, recall that by Lemma \ref{stereoform}, 
\[
\|x-y\| = g(\sigma(x))^{1/4} g(\sigma(y))^{1/4} |\sigma(x)-\sigma(y)|
\]
for all $x,y\in S^2\setminus\{e_3\}$. Let $u := \sigma^{-1}(z_3)$. Then note that by the above identity, and the fact that 
\[
-\frac{1}{4}(4b\alpha_1 + 4b\alpha_2 + 4b\alpha_3 + 4b^2(w-1))= 1,
\]
we get
\begin{align*}
&f_{z_3}(\sigma(y_1),\ldots,\sigma(y_w)) \\
&= \biggl(\prod_{l=1}^w \frac{4\pi}{g(\sigma(y_l))}\biggr) \biggl(\prod_{l=1}^w |\sigma(-e_3) - \sigma(y_l)|^{4b\alpha_1}|\sigma(e_1)-\sigma(y_l)|^{4b\alpha_2}\frac{|\sigma(u) - \sigma(y_l)|^{4b \alpha_3}}{|\sigma(u)-\sigma(-e_3)|^{4b\alpha_3}}\biggr)\\
&\qquad \cdot \biggl(\prod_{1\le l<l'\le w}|\sigma(y_l)-\sigma(y_{l'})|^{4b^2} \biggr)\\
&= (4\pi)^w g(\sigma(-e_3))^{-bw(\alpha_1-\alpha_3)} g(\sigma(e_1))^{-bw\alpha_2}\\
&\qquad \cdot\biggl(\prod_{l=1}^w \|e_3+ y_l\|^{4b\alpha_1}\|e_1-y_l\|^{4b\alpha_2}\frac{\|u - y_l\|^{4b \alpha_3}}{\|u+e_3\|^{4b\alpha_3}}\biggr) \biggl(\prod_{1\le l<l'\le w}\|y_l-y_{l'}\|^{4b^2} \biggr).
\end{align*}
From this, it is clear that if $u$ is in the upper hemisphere, then
\[
|f_{z_3}(\sigma(y_1),\ldots,\sigma(y_w))| \le K \prod_{l=1}^w \|e_3+ y_l\|^{4b\Re(\alpha_1)}\|e_1-y_l\|^{4b\Re(\alpha_2)}\|u - y_l\|^{4b\Re( \alpha_3)},
\]
where $K$ is a constant that does not depend on $y_1,\ldots,y_l$. Since $4b\Re(\alpha_j )> -2$ for $j=1,2,3$, it is easy to see from  the above bound that the condition \eqref{uicond} holds for sufficiently small~$\delta$. This completes the proof of the lemma.
\end{proof}
We are now ready to complete the proof of Theorem \ref{threepoint}.
\begin{proof}[Proof of Theorem \ref{threepoint}]
Integrals of the type appearing in Lemma \ref{threelmm} are sometimes called complex Selberg integrals. They can be exactly evaluated using formulas given by \citet{dotsenkofateev85} and \citet{aomoto87}. Let $\gamma(x) := \Gamma(x)/\Gamma(1-x)$, where $\Gamma$ is the classical Gamma function. Then by \cite[Equation (B.9)]{dotsenkofateev85}, the integral in Lemma \ref{threelmm} is equal to
\begin{align}\label{formula}
w! \pi^w \gamma(b^2)^{-w} \prod_{j=1}^w \gamma(b^2j)\prod_{r=1}^3\prod_{j=1}^{w}\gamma(1+2b\alpha_r + (j-1)b^2).
\end{align}
\citet{aomoto87} gives precise conditions on the parameters under which the above formula holds. In our setting, the conditions are that 
\begin{align*}
&2b\Re(\alpha_1) > -1, \ \ 2b\Re(\alpha_2 ) > -1, \ \ 2b\Re(\alpha_1+\alpha_2)< -1, \\
&2(w-1)b^2 + 2b\Re(\alpha_1+\alpha_2) < -1.
\end{align*}
Let us quickly verify that these conditions hold under the assumptions of the theorem. The first two inequalities are part of the assumptions. For the third inequality, note that since $w\ge 1$ and $\Re(\alpha_3) > -1/2b$, we have
\[
\Re(\alpha_1 + \alpha_2) \le Q-\Re(\alpha_3)-b = -\frac{1}{b} - \Re(\alpha_3) < -\frac{1}{2b}.
\]
Finally, for the fourth inequality, note that 
\[
2(w-1)b^2 + 2b\Re(\alpha_1+\alpha_2) = 2b(Q -\Re(\alpha_3) - b) = -2 - 2b\Re(\alpha_3)< -1,
\]
where the last inequality holds because $\Re(\alpha_3) > -1/2b$. This proves that all of Aomoto's conditions hold, and therefore the formula \eqref{formula} is valid.

Now, it is known~\cite{giribet12} that the $\Upsilon_b$ functions satisfies
\[
\Upsilon_b(x+b) = \gamma(bx) b^{1-2bx}\Upsilon_b(x)
\]
for all $x\in \C$. It is also known that the zeros of $\Upsilon_b$ are the points $mb + nb^{-1}$, where $m,n$ are either both positive integers, or both nonpositive integers~\cite{harlowetal11}. Thus, if $x\in \C$ is not one of the zeros, then for any positive integer $n$, iterating the above identity gives
\begin{align}\label{recur}
\Upsilon_b(x+nb) &= \Upsilon_b(x) b^{n - 2b\sum_{j=0}^{n-1} (x+jb)}\prod_{j=0}^{n-1} \gamma(b(x+jb))\notag\\
&= \Upsilon_b(x) b^{n - 2bnx - b^2 n(n-1)}\prod_{j=0}^{n-1} \gamma(b(x+jb)).
\end{align}
In particular, since $b$ is not a zero of $\Upsilon_b$ (because $b\ne mb + nb^{-1}$ for any positive integers $m,n$), this shows that 
\begin{align*}
\prod_{j=1}^n \gamma(j b^2) = b^{b^2n^2 + b^2n - n} \frac{\Upsilon_b(b+nb)}{\Upsilon_b(b)}.
\end{align*}
Next, we claim that $2\alpha_1+1/b$ is a not a zero of $\Upsilon_b$. To see this, first note that since $\Re(\alpha_1) > -1/2b$, we have that $2\Re(\alpha_1) + 1/b> 0$. So, if $2\alpha_1+1/b$ is a zero of $\Upsilon_b$, it must be of the form $mb + nb^{-1}$ for some positive integers $m,n$. Consequently, $2\alpha_1 = mb + (n-1)b^{-1}\ge 0$. But since $\Re(\alpha_2), \Re(\alpha_3) > -1/2b$ and $w \ge 1$, we have
\[
\Re(\alpha_1) = Q- \Re(\alpha_2+\alpha_3) -bw < Q+ \frac{1}{2b}+\frac{1}{2b} - b = 0.
\]
This shows that $\Upsilon_b(2\alpha_1 + 1/b)\ne 0$. Similarly, $\Upsilon_b(2\alpha_r + 1/b)\ne 0$ for $r=2,3$. Thus, by equation \eqref{recur}, 
\begin{align*}
&\prod_{j=1}^w \gamma(b^2j)\prod_{r=1}^3\prod_{j=1}^{w}\gamma(1+2b\alpha_r + (j-1)b^2) \\
&= \prod_{j=1}^w \gamma(b^2j)\prod_{r=1}^3\prod_{j=1}^{w}\gamma(b(2\alpha_r + 1/b + (j-1)b)) \\
&=b^{b^2w^2 + b^2w - w} \frac{\Upsilon_b(b+wb)}{\Upsilon_b(b)}\prod_{r=1}^3\biggl(\frac{\Upsilon_b(2\alpha_r + 1/b + wb)}{\Upsilon_b(2\alpha_r + 1/b)} b^{-w+2bw(2\alpha_r +1/b) + b^2w(w-1)}\biggr)\\
&= b^{2 b^2w -2w } \frac{\Upsilon_b(b+wb)}{\Upsilon_b(b)}\prod_{r=1}^3\frac{\Upsilon_b(2\alpha_r + 1/b + wb)}{\Upsilon_b(2\alpha_r + 1/b)}.
\end{align*}
The proof is completed by using the relation $\Upsilon_b(x) =\Upsilon_b(b+1/b - x)$ in  the denominator (see \cite{giribet12}), and the fact that $wb = Q - \sum_{j=1}^3 \alpha_j$, and finally, by noting that
\begin{align*}
\gamma(b^2)^{-w} &= \frac{\Gamma(1-b^2)^w}{\Gamma(b^2)^w} =\frac{\Gamma(-b^2)^w(-b^2)^w}{\Gamma(b^2+1)^w (b^2)^{-w}} = \gamma(-b^2)^w (-1)^w b^{4w},
\end{align*}
where we used the relation $\Gamma(x+1)=x\Gamma(x)$. Lastly, to see that $\Upsilon_b(b- 2\alpha_j)\ne 0$ for $j=1,2,3$, recall that $\Re(\alpha_j)>-1/2b$ for each $j$. Thus, $\Re(b - 2\alpha_j) < 2b$. On the other hand, since $\Re(w)\ge 1$, we have that for each $j$, 
\begin{align*}
\Re(\alpha_j)  &= \Re(Q-bw) - \sum_{1\le j'\le 3, \, j'\ne j}\Re(\alpha_{j'})\\
&< Q - b + \frac{2}{2b} = 0,
\end{align*}
which gives $\Re(b - 2\alpha_j) > b$. Since $\Re(b-2\alpha_j)$ is positive, it cannot be that $b-2\alpha_j = m b + n b^{-1}$ for two nonpositive integers $m$ and $n$. On other other hand, since $b\in (0,1)$ and $\Re(b - 2\alpha_j)<2b$, it cannot be that $b-2\alpha_j = m b + n b^{-1}$ for some positive integers $m$ and $n$. Thus, $b-2\alpha_j$ cannot be a zero of $\Upsilon_b$. 
\end{proof}

\section{Going beyond the special region}
In this section, we will show that the timelike DOZZ formula for the $3$-point function of timelike Liouville field theory admits a unique analytic continuation to a larger subset of the parameter space, which allows us to find a collection of nontrivial poles. 
\subsection{$\mathrm{SL}(2,\C)$-invariance of correlation functions}
Let $\hat{\C} = \C \cup \{\infty\}$ denote the Riemann sphere. A function $f:\hat{\C} \to \hat{\C}$ of the form
\begin{align}\label{fabcd}
f(z) = \frac{az+b}{cz+d}
\end{align}
where $a,b,c,d\in \C$ and  $ad -bc = 1$, and the right side is interpreted as $\infty$ if $cz+d = 0$, is called an $\slc$-transform. Note that the condition $ad -bc=1$ can be replaced by $ad-bc \ne 0$, because dividing $a,b,c,d$ by any square-root of $ad-bc$ gives an equivalent form of $f$ which satisfies $ad-bc = 1$.
\begin{thm}\label{sl2cthm}
Suppose that $k\ge 3$, $\Re(\alpha_j) > -1/2b$ for each $j$, and the parameter $w= (Q-\sum_{j=1}^k \alpha_j)/b$ is a positive integer. Let $f$ be an $\slc$-transform. Then for any distinct $z_1,\ldots,z_k\in \C$, 
\begin{align*}
C(\alpha_1,\ldots,\alpha_k;f(z_1),\ldots,f(z_k); b; \mu) &= \biggl(\prod_{j=1}^k |f'(z_j)|^{2\Delta_{\alpha_j}}\biggr) C(\alpha_1,\ldots,\alpha_k;z_1,\ldots,z_k; b; \mu).
\end{align*}
\end{thm}
\begin{proof}
Let $f$ be the function displayed in equation \eqref{fabcd}. It is clear that $f$ is holomorphic on $\C\setminus\{-d/c\}$, with derivative
\begin{align}\label{fprime}
f'(z) = \frac{ad-bc}{(cz+d)^2} = \frac{1}{(cz+d)^2}. 
\end{align}
Now, we can write $f$ as a function from $\R^2\setminus\{(-d/c,0)\}$ into $\R^2$, given by
\[
f(x,y) = (u(x,y), v(x,y)),
\]
where $u$ and $v$ are the real and imaginary parts of $f$. By the Cauchy--Riemann equations, 
\[
\frac{\partial u}{\partial x} = \frac{\partial v}{\partial y}, \ \ \ \frac{\partial u}{\partial y} = -\frac{\partial v}{\partial x}.
\]
Thus, the determinant of the Jacobian of $f$ (as a map from  $\R^2\setminus\{(-d/c,0)\}$ into $\R^2$) is given by 
\begin{align*}
\det J_f(x,y) &= \frac{\partial u}{\partial x}\frac{\partial v}{\partial y} - \frac{\partial u}{\partial y}\frac{\partial v}{\partial x}= \biggl(\frac{\partial u}{\partial x}\biggr)^2 +\biggl( \frac{\partial v}{\partial x}\biggr)^2.
\end{align*}
But, since $f$ is holomorphic, we have
\[
f'(x+iy) = \frac{\partial u}{\partial x} +i \frac{\partial v}{\partial x}.
\]
Thus, by equation \eqref{fprime}, 
\[
\det J_f(x,y) = |f'(x+iy)|^2 = \frac{1}{|c(x+iy) + d|^4}.
\]
Lastly, note that $f$ is a bijection of $\hat{\C}$ on to itself. By these facts, and the formula from Theorem \ref{intthmv2}, we get
\begin{align*}
&C(\alpha_1,\ldots,\alpha_k; f(z_1),\ldots,f(z_k);b;\mu) \\
&= \frac{e^{-i\pi w}\mu^{w}}{w!} (4/e)^{1-1/b^2}\prod_{1\le j<j'\le k} |f(z_j) - f(z_{j'})|^{4\alpha_j \alpha_{j'}}\\
&\qquad \qquad \cdot  \int_{\C^w} \biggl(\prod_{j=1}^k \prod_{l=1}^{w} |f(z_j)-t_l|^{4b\alpha_j}\biggr)\biggl(\prod_{1\le l<l'\le w}|t_l-t_{l'}|^{4b^2} \biggr) d^2t_1\cdots d^2t_w\\
&= \frac{e^{-i\pi w}\mu^{w}}{w!} (4/e)^{1-1/b^2}\prod_{1\le j<j'\le k} |f(z_j) - f(z_{j'})|^{4\alpha_j \alpha_{j'}}\int_{\C^w} \biggl(\prod_{j=1}^k \prod_{l=1}^{w} |f(z_j)-f(s_l)|^{4b\alpha_j}\biggr)\\
&\qquad \qquad \cdot  \biggl(\prod_{1\le l<l'\le w}|f(s_l)-f(s_{l'})|^{4b^2} \biggr) \prod_{l=1}^w \frac{1}{|cs_l + d|^4} d^2s_1\cdots d^2s_w.
\end{align*}
The function $f$ also has the interesting property that for any $z,z'\in \C \setminus\{-d/c\}$,
\begin{align*}
f(z) - f(z') &= \frac{az+b}{cz+d} - \frac{az' + b}{cz' + d}\\
&= \frac{(az+b)(cz'+d) - (az' + b)(cz+d)}{(cz+d)(cz'+d)}\\
&= \frac{(ad-bc)(z-z')}{(cz+d)(cz'+d)} = \frac{z-z'}{(cz+d)(cz'+d)}. 
\end{align*}
If we use this identity in the previous display, each $|f(z_j)-f(z_{j'})|$ is replaced by
\[
\frac{|z_j - z_{j'}|}{|cz_j+d||cz_{j'} + d|},
\]
each $|f(z_j)-f(s_l)|$ is replaced by 
\[
\frac{|z_j - s_l|}{|cz_j+d||cs_l + d|},
\]
and each $|f(s_l)-f(s_{l'})|$ is replaced by
\[
\frac{|s_l - s_{l'}|}{|cs_l+d||cs_{l'} + d|}.
\]
In the resulting expression, the powers of $|c s_l + d|$ add up to 
\begin{align*}
-4b\sum_{j=1}^k\alpha_j -4b^2(w-1) -4 &= -4b\sum_{j=1}^k\alpha_j - 4b\biggl(Q- \sum_{j=1}^k \alpha_j\biggr) + 4b^2 - 4\\
&= -4bQ + 4b^2 - 4 = 0.
\end{align*}
Similarly, the powers of $|cz_j + d|$ add up to
\begin{align*}
-4\alpha_j\sum_{j'\ne j} \alpha_{j'} - 4bw\alpha_j &= -4\alpha_j\sum_{j'\ne j} \alpha_{j'} - 4\biggl(Q- \sum_{j=1}^k \alpha_j\biggr) \alpha_j \\
&= -4Q\alpha_j +4\alpha_j^2 = -4\Delta_{\alpha_j}.
\end{align*}
This completes the proof.
\end{proof}


\subsection{The structure constants of timelike Liouville theory}\label{structuresec}
A corollary of Theorem \ref{sl2cthm} and Theorem \ref{threepoint} is that the $3$-point function can be written as the product of an explicit function of $z_1,z_2,z_3$ (with explicit constants depending on the $\alpha_1,\alpha_2,\alpha_3$) times the function of $\alpha_1,\alpha_2,\alpha_3$ displayed in Theorem \ref{threepoint}.
\begin{cor}\label{sl2ccor}
Let $C(\alpha_1,\alpha_2,\alpha_3; b;\mu)$ be the function defined in Theorem \ref{threepoint}. Then for any distinct $z_1,z_2,z_3\in \C$, we have
\[
C(\alpha_1,\alpha_2,\alpha_3; z_1,z_2,z_3;b;\mu) = C(\alpha_1,\alpha_2,\alpha_3; b;\mu)|z_{12}|^{2\Delta_{12}}|z_{13}|^{2\Delta_{13}}|z_{23}|^{2\Delta_{23}},
\]
where $z_{jk} := z_j -z_k$, and $\Delta_{jk}:= 2\Delta_{\alpha_j} + 2\Delta_{\alpha_k} -\sum_{l=1}^3\Delta_{\alpha_l}$. 
\end{cor}
\begin{proof}
Take any distinct $z_1, z_2,z_3,z_4\in \C$. Define $f:\C\setminus\{z_4\}\to \C$ as
\[
f(z) := \frac{(z-z_1)(z_2-z_4)}{(z-z_4)(z_2-z_1)}.
\]
Note that $f$ is an $\slc$-transform, since $f$ can be expressed as $(az+b)/(cz+d)$, where 
\[
a = \frac{z_2 - z_4}{u}, \ \ b = -\frac{z_1(z_2-z_4)}{u}, \ \ c = \frac{z_2-z_1}{u}, \ \ d = -\frac{z_4(z_2-z_1)}{u},
\]
where $u$ is any square-root of $(z_1-z_4)(z_2-z_4)(z_2-z_1)$. Note that 
\[
f(z_1) = 0, \ \  f(z_2) = 1, \ \ f(z_3) = \frac{(z_3-z_1)(z_2-z_4)}{(z_3-z_4)(z_2-z_1)}.
\]
Note also that 
\begin{align*}
|f'(z_1)| &= \frac{|(z_1-z_4)(z_2-z_4)(z_2-z_1)|}{|(z_2-z_1)(z_1-z_4)|^2}= \frac{|z_2-z_4|}{|(z_1-z_2)(z_1-z_4)|},\\
|f'(z_2)| &= \frac{|(z_1-z_4)(z_2-z_4)(z_2-z_1)|}{|(z_2-z_1)(z_2-z_4)|^2}= \frac{|z_1-z_4|}{|(z_1-z_2)(z_2-z_4)|},\\
|f'(z_3)| &= \frac{|(z_1-z_4)(z_2-z_4)(z_2-z_1)|}{|(z_2-z_1)(z_3-z_4)|^2}= \frac{|(z_2-z_4)(z_1-z_4)|}{|(z_1-z_2)(z_3-z_4)^2|}.
\end{align*}
Now, by Theorem \ref{sl2cthm},
\begin{align*}
C(\alpha_1,\alpha_2,\alpha_3; z_1,z_2,z_3;b;\mu) &= \frac{C(\alpha_1,\alpha_2,\alpha_3; 0,1,f(z_3);b;\mu)}{|f'(z_1)|^{2\Delta_{\alpha_1}} |f'(z_2)|^{2\Delta_{\alpha_2}} |f'(z_3)|^{2\Delta_{\alpha_3}} }.
\end{align*}
Also, $|f(z_3)|\to \infty$ as $z_4 \to z_3$. Thus, taking $z_4\to z_3$ and applying Theorem \ref{threepoint}, we get
\begin{align*}
&C(\alpha_1,\alpha_2,\alpha_3; z_1,z_2,z_3;b;\mu) \\
&= C(\alpha_1,\alpha_2,\alpha_3; b;\mu)\lim_{z_4\to z_3} \frac{|f(z_3)|^{4\Delta_{\alpha_3}}}{|f'(z_1)|^{2\Delta_{\alpha_1}} |f'(z_2)|^{2\Delta_{\alpha_2}} |f'(z_3)|^{2\Delta_{\alpha_3}} }\\
&= C(\alpha_1,\alpha_2,\alpha_3; b;\mu) |z_1-z_2|^{2\Delta_{\alpha_1}+2\Delta_{\alpha_2}-2\Delta_{\alpha_3}}|z_2 - z_3|^{-2\Delta_{\alpha_1}+2\Delta_{\alpha_2}+2\Delta_{\alpha_3}}\\
&\qquad \qquad \cdot |z_1-z_3|^{ 2\Delta_{\alpha_1}-2\Delta_{\alpha_2}+2\Delta_{\alpha_3}}.
\end{align*}
This completes the proof.
\end{proof}
The numbers $C(\alpha_1,\alpha_2,\alpha_3;b;\mu)$, as $\alpha_1,\alpha_2,\alpha_3$ vary, are called the structure constants of timelike Liouville field theory. Structure constants are key objects in a conformal field theory. For example, the above corollary shows that they contain all relevant information about the $3$-point function. 



\subsection{Trivial poles}
The poles of the $3$-point function contain important physical information. For example, they are crucial for the notion of operator resonances introduced by \citet{zamolodchikov91} in conformal perturbation theory. The resonances manifest themselves in poles of the correlation functions.

As noted in the statement Theorem \ref{threepoint}, the region covered by Theorem \ref{threepoint} does not contain any poles. However, the formula does show that there are poles as we approach the boundary of that region, as shown by the following result. We will refer to these as the `trivial poles', corresponding to the regions $\alpha_j = -1/2b$, for $j=1,2,3$.
\begin{prop}\label{poleprop}
Choose $b$ such that $b^2$ is irrational. If we take $\alpha_j \to -1/2b$ for exactly one $j$, while keeping $w$ fixed at a positive integer value and $\Re(\alpha_1),\Re(\alpha_2),\Re(\alpha_3)>-1/2b$, then the value of $C(\alpha_1,\alpha_2,\alpha_3;b;\mu)$ blows up to infinity.
\end{prop}
\begin{proof}
Suppose, without loss of generality, that we take $\alpha_1 \to -1/2b$ while keeping $\Re(\alpha_j) >-1/2b$ for each $j$ and $w$ fixed at a positive integer value. Then $\Upsilon_b(b - 2\alpha_1) \to 0$. So we only have to show that the numerator in the timelike DOZZ formula does not tend to zero. 

Let $(\alpha_1,\alpha_2,\alpha_3)$ be the limiting value, so that $\alpha_1=-1/2b$ and $\Re(\alpha_2),\Re(\alpha_3)\ge -1/2b$. The terms appearing in the numerator that can potentially tend to zero are $\Upsilon_b(bw+b)$ and $\Upsilon_b(2\alpha_j + bw + 1/b)$ for $j=1,2,3$. We will now show that these are all nonzero. 

As noted in \cite{harlowetal11}, the zeros of $\Upsilon_b$ consists of  the points $mb + nb^{-1}$ where either both $m,n$ are positive integers, or both $m,n$ are nonpositive integers. Since $b^2$ is irrational and $w$ is a positive integer, it follows that $bw+b$ can attain neither of these forms, and is therefore not a zero of $\Upsilon_b$.

Next, note that $2\alpha_1+bw+1/b = bw$ is not a zero of $\Upsilon_b$ since $b^2$ is irrational. Suppose that $2\alpha_2+bw+1/b$ is a zero of $\Upsilon_b$. Since $\Re(\alpha_2) \ge -1/2b$, we have $2\Re(\alpha_2) + bw+1/b \ge bw>0$, which means (again, by the irrationality of $b^2$) that it must be of the form $mb + nb^{-1}$ for some positive integers $m,n$. Thus,
\[
\alpha_2 = \frac{1}{2}(m-w)b + \frac{n-1}{2b}.
\]
On the other hand, since $\Re(\alpha_l)\ge -1/2b$ for each $l$, we have
\begin{align*}
\alpha_2 = \Re(\alpha_2) = Q -\Re( \alpha_{1} + \alpha_3) -bw\le Q+ \frac{1}{b} - bw = b(1-w).
\end{align*}
Combining, we get
\[
\frac{n-1}{2b} \le \biggl(1-\frac{m+w}{2}\biggr) b.
\]
Since $m,n,w$ are all positive integers, the above inequality can hold only if $m=n=w=1$. But then, $\alpha_2=0$, and hence
\[
\alpha_3 = Q-\alpha_1-\alpha_2-bw = -\frac{1}{2b}.
\]
This contradicts our assumption that $\alpha_2,\alpha_3 \ne -1/2b$. Thus, $2\alpha_2+bw+1/b$ is not a zero of $\Upsilon_b$. Similarly, $2\alpha_3+bw+1/b$ is not a zero of $\Upsilon_b$. This completes the proof.
\end{proof}

\subsection{Analytic continuation of the structure constants}\label{analyticsec}
In Subsection \ref{dozzsec}, we derived the timelike DOZZ formula for $C(\alpha_1,\alpha_2,\alpha_3;b;\mu)$ under the conditions that 
\begin{itemize}
\item $\Re(\alpha_j)> -1/2b$ for $j=1,2,3$, and
\item $w = (Q-\sum_{j=1}^3\alpha_j)/b$ is a positive integer, where $Q = b - 1/b$. 
\end{itemize}
In this section, our goal is to drop the first condition, while keeping the second. The gain in doing so is that the expanded region contains additional, `nontrivial' poles of $C$, while the original region does not. 

Suppose we fix the value of $w$ at a positive integer, and consider $\alpha_3$ as a function of $\alpha_1$ and $\alpha_2$, given by
\[
\alpha_3 = Q - \alpha_1-\alpha_2 -bw.
\]
Thus, with $w$ fixed to be a positive integer, the structure constant $C(\alpha_1,\alpha_2,\alpha_3;b;\mu)$ can be expressed as a function of $\alpha_1$ and $\alpha_2$:
\[
C_w(\alpha_1,\alpha_2) := C(\alpha_1,\alpha_2, Q-\alpha_1-\alpha_2-bw; b;\mu).
\]
In the region 
\[
\Omega := \{(\alpha_1,\alpha_2)\in \R^2: \alpha_1>-1/2b, \, \alpha_2>-1/2b, \, Q-\alpha_1-\alpha_2-bw > -1/2b\},
\]
Theorem \ref{threepoint} implies that 
\begin{align*}
C_w(\alpha_1,\alpha_2) &= e^{-i\pi w} (-\pi \mu\gamma(-b^2))^w (4/e)^{1-1/b^2}b^{2b^2w + 2w} \frac{\Upsilon_b(bw+b)}{\Upsilon_b(b)}\\
&\qquad \cdot \frac{\Upsilon_b(2\alpha_1+bw+1/b)\Upsilon_b(2\alpha_2+bw + 1/b)\Upsilon_b(Q-2\alpha_1-2\alpha_2-bw+b)}{\Upsilon_b(b - 2\alpha_1)\Upsilon_b(b-2\alpha_2) \Upsilon_b(2\alpha_1+2\alpha_2+2bw-b + 2/b)}.
\end{align*}
The following result allows us to extend $C_w$ to the whole of $\C^2$. Recall that a meromorphic function in several complex variables is a function that can be locally expressed as a ratio of two holomorphic functions.
\begin{thm}\label{analyticext}
Suppose that $w$ is a positive integer and $0<b< (2(w-1))^{-1/2}$. 
Then the value displayed on the right side of the above equation gives the unique meromorphic continuation of $C_w$ to $\C^2$. 
\end{thm}
\begin{proof}
Let $A$ be the meromorphic function on $\C^2$ defined by the right side of the above equation. Let $B$ be another meromorphic continuation of $C_w$ to $\C^2$. Then $A=B$ on $\Omega$. We have to show that $A=B$ everywhere on $\C^2$. First, note that since $b < (2(w-1))^{-1/2}$, we have
\begin{align*}
Q - (-1/2b) - (-1/2b) -bw &= b(1-w) = -b|w-1| > -\frac{1}{2b}.
\end{align*}
Thus, we can choose $\alpha_1, \alpha_2$ slightly greater than $-1/2b$ such that $Q-\alpha_1-\alpha_2-bw$ is also greater than $-1/2b$. This shows that $\Omega$ is nonempty. By the definition of $\Omega$, it is clear that $\Omega$ is open. Combining, we see that there exists a rectangle $(a,b)\times(c,d)\subseteq \Omega$, where $a< b$ and $c<d$. 

Take any $\alpha_1\in (a,b)$. Then $A(\alpha_1,\cdot)$ and $B(\alpha_1,\cdot)$ are meromorphic functions on $\C$ (since fixing all but one coordinates of a meromorphic function of several complex variables gives a meromorphic function in the remaining variable), which coincide on the interval $(c,d)$. Thus, they coincide everywhere on $\C$. This shows that $A(\alpha_1,\alpha_2)=B(\alpha_1,\alpha_2)$ for all $\alpha_1\in (a,b)$ and $\alpha_2\in \C$ (meaning, in particular, that they have the same poles). Now take any $\alpha_2\in \C$. Then $A(\cdot, \alpha_2)$ and $B(\cdot, \alpha_2)$ are meromorphic functions on $\C$, and by the first step, they coincide on the interval $(a,b)$. Thus, they agree everywhere. This completes the proof.
\end{proof}

\subsection{A set of nontrivial poles}\label{polesec}
As a corollary of Theorem \ref{analyticext}, we obtain the following conditional result, which extends the domain where the timelike DOZZ formula coincides with the timelike Liouville structure constants. Recall that the function $\Upsilon_b$ has no poles, and has zeros exactly at the points $mb + nb^{-1}$ where either both $m,n$ are positive integers, or both $m,n$ are nonpositive integers.
\begin{cor}\label{extcor}
Let $C(\alpha_1,\alpha_2,\alpha_3;b;\mu)$ denote the structure constants of timelike Liouville field theory. 
Let us make the physical assumption that there is a `true' set of structure constants $C(\alpha_1,\alpha_2,\alpha_3;b;\mu)$ that is a meromorphic function of $(\alpha_1,\alpha_2,\alpha_3)$, and that it is given by the formula derived via the path integral under the conditions of Theorem \ref{threepoint}. Then $C(\alpha_1,\alpha_2,\alpha_3;b;\mu)$ is given by the timelike DOZZ formula displayed in Theorem \ref{threepoint} for all $\alpha_1,\alpha_2,\alpha_3\in \C$ such that the number $w = (Q-\sum_{j=1}^3 \alpha_j)/b$ is a positive integer less that $1+\frac{1}{2b^2}$. Consequently, for a given $b$ and $\mu$, the function $C$ has poles at all $(\alpha_1,\alpha_2,\alpha_3)$ where all of the following conditions are satisfied:
\begin{itemize}
\item $w = (Q-\sum_{j=1}^3 \alpha_j)/b$ is a positive integer less than $1+\frac{1}{2b^2}$. 
\item $2\alpha_j+\frac{1}{b}$ is a zero of $\Upsilon_b$ for some $j$.
\item $2\alpha_j + bw + \frac{1}{b}$ is not a zero of $\Upsilon_b$ for any $j$.
\item $bw +b$ is not a zero of $\Upsilon_b$.
\end{itemize}
\end{cor}
\begin{proof}
The first assertion is a direct consequence of Theorem \ref{analyticext}, observing that the condition $b < (2(w-1))^{-1/2}$ is the same as $w < 1+\frac{1}{2b^2}$. For the second assertion, a simple inspection (recalling the relation $\Upsilon_b(x)=\Upsilon_b(b+1/b-x)$) shows that when the four conditions hold, then the denominator in the timelike DOZZ formula is zero and the numerator is nonzero.
\end{proof}
The following is an example of a `nontrivial' pole, that is, a pole that is not given by $\alpha_j =-1/2b$ for some $j$. Take any $b\in (0, 3^{-1/2})$ such that $b^2$ is irrational, and take 
\[
\alpha_1 = -\frac{1}{2}b - \frac{1}{2b}, \ \ \alpha_2 = -\frac{1}{2}b - \frac{1}{4b}, \ \ \alpha_3 = -\frac{1}{4b}.
\]
Then $w = (Q-\sum_{j=1}^3 \alpha_j)/b = 2$. Thus, 
\[
w < 3 < 1+\frac{1}{2b^2}. 
\]
Next, note that $2\alpha_j + \frac{1}{b} = -b$ is a zero of $\Upsilon_b$. Thirdly, note that 
\begin{align*}
&2\alpha_1 + bw + \frac{1}{b} =  b. 
\end{align*}
We claim that $b$ is not a zero of $\Upsilon_b$. Indeed, $b$ is strictly less than $mb + nb^{-1}$ for all positive integers $m,n$, and if $b = mb + nb^{-1}$ for some nonpositive integers $m,n$, then again $b^2 = n/(1-m)$ is rational, contradicting our assumption that $b^2$ is irrational. Thus, $2\alpha_1 + bw + \frac{1}{b}$ is not a zero of $\Upsilon_b$. Next, we have
\[
2\alpha_2 + bw + \frac{1}{b} = b + \frac{1}{2b}. 
\]
By a similar argument as above, this cannot be a zero of $\Upsilon_b$. Lastly, 
\[
2\alpha_3 + bw + \frac{1}{b} = 2b + \frac{1}{2b}, 
\]
which, again, cannot be a zero of $\Upsilon_b$ due to the irrationality of $b^2$. Thus, $C$ has a pole at the above values of $\alpha_1,\alpha_2,\alpha_3$ and $b$.

\section{Semiclassical limit of timelike Liouville field theory}
In this section, we will show that the correlation functions of timelike Liouville field theory that we computed using the path integral representation, have the correct behavior in the semiclassical limit. This means that as $b\to 0$, the path integrals concentrate more and more near certain critical points of the timelike Liouville action.

\subsection{Semiclassical limit with heavy operators}\label{semilimitsec}
In this subsection we return to the study of $k$-point correlations for general $k$. The semiclassical limit of timelike Liouville theory with heavy operators is obtained by taking $b\to 0$, and varying $\alpha_j$ and $\mu$ as $\alpha_j = \talpha_j/b$ and $\mu = \tmu/b^2$ where $\talpha_j$ and $\tmu$ remain fixed as $b\to 0$ through a sequence such that $w\to \infty$ through positive integers. Also, throughout this subsection and the next, we will take the $\talpha_j$'s to be real, to avoid technical complexities.  First, we identify the limit of the $k$-point correlation. In the next subsection, we will show that this limit is  `correct', in that it shows convergence to a critical point of the timelike Liouville action after insertion of heavy operators.  Throughout this section, we will use the notational convention that for  a function $f:S^2\to\R$, $Gf$ denotes the function
\[
Gf(x) := \int_{S^2}G(x,y) f(y) da(y),
\]
where $G$ is the Green's function for the spherical Laplacian appearing in equation \eqref{greendef}, whose explicit form was obtained in Lemma \ref{glimit}. 
Let $\cp$ be the set of probability density functions (with respect to the area measure) on $S^2$. Recall the following three functionals on $\cp$, defined in Subsection \ref{mainresultssec}:
\begin{align*}
H(\rho) &:= \int_{S^2} \rho(x)\ln \rho(x) da(x),\\
R(\rho) &:= \int_{(S^2)^2} \rho(x)\rho(y) G(x,y) da(x) da(y),\\
L(\rho) &:= \sum_{j=1}^k 4\talpha_j \int_{S^2} G(x_j, x) \rho(x)da(x).
\end{align*}
Let $\cp'$ be the subset of $\cp$ consisting of all $\rho$ such that $H(\rho)$ is finite. We claim that for $\rho\in \cp'$, the functionals $R(\rho)$ and $L(\rho)$ are also finite. To see this, observe that by Jensen's inequality, we have for each $x\in S^2$, 
\begin{align*}
|G\rho(x)| &\le  \int_{S^2} |G(x,y)| \rho(y) da(y) = \int_{S^2} \rho(y) \ln \frac{e^{|G(x,y)|} }{\rho(y)} da(y) + H(\rho)\\
&\le \ln \biggl(\int_{S^2} e^{|G(x,y)|} da(y)\biggr) + H(\rho) \le \ln \biggl(\int_{S^2} \frac{2e^{-1/2}}{\|x-y\| }da(y)\biggr) +H(\rho).
\end{align*}
But the last quantity actually does not depend on $x$, and is finite. 
That is, $|G\rho|$ is uniformly bounded. Clearly, this implies that $R(\rho)$ and $L(\rho)$ are finite. 
In the following theorem, quoted from Subsection \ref{mainresultssec}, we take the logarithm of the $k$-point correlation. While taking the logarithm, we interpret the logarithm of the $e^{-i\pi w}$ term appearing the formula from Theorem \ref{intthmv2} as $-i\pi w$. Since the remaining terms are real and positive, there is no ambiguity about their logarithms.

\semilimitthm*

The proof of Theorem \ref{semilimit} is somewhat lengthy. We start with the following lemma, which proves the `easy direction'.
\begin{lmm}\label{lowerbound}
In the setting of Theorem \ref{semilimit}, we have
\begin{align*}
&\liminf_{n\to \infty} \frac{1}{n} \log C(\talpha_1/b_n,\ldots,\talpha_k/b_n; x_1,\ldots,x_k;b_n; \tmu/b_n^2)+i\pi \\
&\ge 1 + \ln \tmu - \ln \beta +(1-\ln 4  )\sum_{j=1}^k \frac{\talpha_j^2}{\beta} + \sum_{j=1}^k \frac{\talpha_j (1+\talpha_j)}{\beta} \ln g(\sigma(x_j))\\
&\qquad \qquad  -\frac{4}{\beta}\sum_{1\le j<j'\le k}\talpha_j\talpha_{j'}G(x_j, x_{j'}) - \inf_{\rho\in \cp'} S(\rho).
\end{align*}
\end{lmm}
\begin{proof}
Let $w_n$ be the $w$ for $b=b_n$. Note that 
\begin{align*}
w_n = 1-\frac{1}{b_n^2} - \frac{1}{b_n}\sum_{j=1}^k \frac{\talpha_j}{b_n} = 1-\frac{1+\sum_{j=1}^k\talpha_j}{b_n^2} = n. 
\end{align*}
Moreover, $\talpha_j/b_n > -1/2b_n$ for each $j$. Thus, by Theorem \ref{intthm}, 
\begin{align*}
&C(\talpha_1/b_n,\ldots,\talpha_k/b_n; x_1,\ldots,x_k;b_n; \tmu/b_n^{2n}) \\
&= \frac{e^{-i\pi n}\tmu^n}{n!b^{2n}_n}\biggl(\prod_{j=1}^ke^{\chi(\talpha_j/b_n)(b_n-\talpha_j/b_n)} g(\sigma(x_j))^{-\Delta_{\talpha_j/b_n}}\biggr) \exp\biggl(-\frac{4}{b_n^2}\sum_{1\le j<j'\le k} \talpha_j\talpha_{j'}G(x_j, x_{j'})\biggr) \\
&\qquad \cdot \int_{(S^2)^n}\exp\biggl(-4\sum_{j=1}^k \sum_{l=1}^n\talpha_jG(x_j, y_l)-4b_n^2\sum_{1\le l<l'\le n}G(y_l, y_{l'}) \biggr) da(y_1)\cdots da(y_{n}).
\end{align*}
Let $I_n$ denote the above integral. Take any $\rho\in \cp'$. By Jensen's inequality (with the interpretation that $0\ln 0=0$),
\begin{align*}
&I_n = \int_{(S^2)^n}\exp\biggl(-4\sum_{j=1}^k \sum_{l=1}^n\talpha_jG(x_j, y_l)-4b_n^2\sum_{1\le l<l'\le n}G(y_l, y_{l'}) \\
&\qquad \qquad - \sum_{l=1}^n \ln \rho(y_l) \biggr) \prod_{l=1}^n \rho(y_l) da(y_1)\cdots da(y_{n})\\
&\ge \exp\biggl(-4\sum_{j=1}^k \sum_{l=1}^n\talpha_j\int_{S^2}G(x_j, y_l)\rho(y_l) da(y_l) \\
& \qquad -4b_n^2\sum_{1\le l<l'\le n}\int_{(S^2)^2}G(y_l, y_{l'})\rho(y_l)\rho(y_{l'}) da(y_l)da(y_{l'})  - \sum_{l=1}^n \int_{S^2}\rho(y_l)\ln \rho(y_l) da(y_l) \biggr).
\end{align*}
This shows that 
\begin{align*}
\liminf_{n\to \infty} \frac{1}{n}\ln I_n &\ge -S(\rho). 
\end{align*}
To complete the proof, note that as $n\to \infty$,
\begin{align*}
&\frac{\talpha_j}{nb_n}\biggl(b_n-\frac{\talpha_j}{b_n}\biggr) \to -\frac{\talpha_j^2}{\beta},\\
&-\frac{\Delta_{\talpha_j/b_n}}{n} = -\frac{\talpha_j}{nb_n}\biggl(b_n - \frac{1}{b_n} - \frac{\talpha_j}{b_n}\biggr) \to \frac{\talpha_j(1+\talpha_j)}{\beta},
\end{align*}
and $\frac{1}{n}\ln(n! b_n^{2n}) \to \ln \beta - 1$.
\end{proof}

The proof of the opposite direction requires some preparation.  Throughout the following discussion, $C, C_1,C_2,\ldots$ will denote positive constants that may depend only on $\talpha_1,\ldots,\talpha_k$ and $x_1,\ldots,x_k$, and nothing else. The values of these constants  may change from line to line or even within a line. The first step towards the proof of the upper bound is the following lemma.
\begin{lmm}\label{glemma}
For any distinct $x,y,x',y'\in S^2$, 
\[
-G(x,y) \le - G(x',y') +\frac{\|x-x'\|+\|y-y'\|}{\|x'-y'\|}.
\]
\end{lmm}
\begin{proof}
Recall that $\ln x\le x-1$ for all $x> 0$. Thus, 
\begin{align*}
-G(x,y)+G(x',y') &= \ln \frac{\|x-y\|}{\|x'-y'\|}\le \frac{\|x-y\|}{\|x'-y'\|} -1 = \frac{\|x-y\|-\|x'-y'\|}{\|x'-y'\|}.
\end{align*}
But by the triangle inequality,
\[
\|x-y\|-\|x'-y'\| \le \|(x-y)-(x'-y')\|\le \|x-x'\|+\|y-y'\|.
\]
Combining the two steps completes the proof.
\end{proof}
Next, let $\epsilon$ be a small positive real number of the form $\pi/2q$ for some integer $q$, to be chosen later. Throughout the following discussion, we will assume without mention, wherever needed, that $\epsilon$ is small enough. Here `small enough' means `smaller than a suitable number depending only on $\talpha_1,\ldots,\talpha_k$ and $x_1,\ldots,x_k$, and nothing else'. 

We construct a partition of $S^2$ as follows. First, draw latitudes on the sphere whose angles to the equator are integer multiples of $\epsilon$. That is, each latitude is a circle parametrized as $(\cos n\epsilon \cos \theta, \cos n \epsilon \sin \theta, \sin n\epsilon)$, $0\le \theta\le 2\pi$ for some $-(q-1)\le n\le q-1$. Thus, the distance between successive latitudes is of order $\epsilon$, and the caps enclosed by the highest and lowest latitudes have radius of order $\epsilon$. These latitudes divide $S^2$ into a collection of $2q$ annuli (including the caps at the top and the bottom). 

Let us number these annuli as $A_n$, $n=-q,-(q-1),\ldots,q-1$, going from bottom to top. We now further subdivide each annulus $A_n$ as follows.  Note that the latitude passing through the middle of $A_n$ has radius $r_n := \cos (n +\frac{1}{2})\epsilon$. Let $\delta_n := (\frac{1}{2\pi}+[ r_n/\epsilon])^{-1}$. Consider the equally spaced longitudes parametrized by $(\cos \phi \cos m\delta_n, \cos \phi \sin m\delta_n, \sin \phi)$, $-\pi/2\le \phi\le\pi/2$, for $0\le m \le [r_n/\epsilon]$. Subdivide $A_n$ using these longitudes. Then note that each subdivision is bounded above and below by two latitudes separated by distance of order $\epsilon$, and to its left and right by two longitudes that are separated by distance of order $\delta_n r_n$, which is of order $\epsilon$. Thus, we have produced a partition of $S^2$ such that each element of the partition is a `trapezoid' with length and breadth both of order $\epsilon$. 

Let us enumerate these trapezoids as $B_1,\ldots,B_L$, where $L$ is the number of trapezoids. Additionally, we assume that $\epsilon$ is so small that $x_1,\ldots,x_k$ are in distinct trapezoids. Without loss of generality, let us assume that $x_j\in B_j$ for $j=1,\ldots,k$. For $j=1,\ldots,L$, let $c_j$ be the area of $B_j$ and let $a_j \in B_j$ be the `center' of $B_j$, defined in the obvious way. For $1\le j,j'\le L$, let $g_{j,j'}$ be the average value of $G(x,y)$ for $x\in B_j$ and $y\in B_{j'}$. That is,
\[
g_{j,j'} = \frac{1}{c_j c_{j'}} \int_{B_j}\int_{B_{j'}} G(x,y) da(y) da(x).
\]
Also, for $1\le j\le k$ and $1\le j'\le L$, let $h_{j,j'}$ be the average value of $G(x_j,y)$ for $y\in B_{j'}$. That is,
\[
h_{j,j'} = \frac{1}{c_{j'}} \int_{B_{j'}} G(x_j,y) da(y). 
\]
The following two lemmas give crucial estimates relating the function $G$ to the averages defined above.
\begin{lmm}\label{gjjbound}
Take any $x,y\in S^2$, with $x\in B_j$ and $y\in B_{j'}$ for some $1\le j,j'\le n$. Then 
\[
-G(x,y) \le - g_{j,j'} + \frac{C\epsilon}{\epsilon+\|a_j - a_{j'}\|}.
\]
\end{lmm}
\begin{proof}
For the first inequality, note that by Lemma \ref{glemma},
\begin{align*}
-G(x,y) + g_{j,j'} &= \frac{1}{c_jc_{j'}} \int_{B_j}\int_{B_{j'}} (-G(x,y) + G(x',y')) da(y') da(x')\\
&\le \frac{1}{c_jc_{j'}} \int_{B_j}\int_{B_{j'}} \frac{\|x-x'\| + \|y-y'\|}{\|x'-y'\|} da(y') da(x')\\
&\le \frac{C_1\epsilon}{C_2 \epsilon^4} \int_{B_j}\int_{B_{j'}} \frac{1}{\|x'-y'\|} da(y') da(x').
\end{align*}
If $\|a_j -a_{j'}\|\ge C \epsilon$ for some large enough $C$, then $\|x'-y'\|\ge \frac{1}{2}(\epsilon + \|a_j - a_{j'}\|)$ for any $x'\in B_j$ and $y'\in B_{j'}$. On the other hand if $\|a_j - a_{j'}\| < C\epsilon$, then it is easy to see (using polar coordinates, for instance) that
\[
\int_{B_j}\int_{B_{j'}} \frac{1}{\|x'-y'\|} da(y') da(x')\le C_1\epsilon^3\le \frac{C_2\epsilon^4}{\epsilon + \|a_j - a_{j'}\|}. 
\]
Combining these observations, we get 
\[
-G(x,y)+g_{j,j'}\le \frac{C\epsilon}{\epsilon + \|a_j -a_{j'}\|}.
\]
This completes the proof.
\end{proof}
\begin{lmm}\label{hjjbound}
There is a universal constant $C_0$ such that the following holds. Take any $1\le j\le k$ and $1\le j'\le n$ such that $\|a_j - a_{j'}\|> C_0\epsilon$. Then for $y\in B_{j'}$, 
\[
|G(x_j, y) - h_{j,j'}| \le  \frac{C\epsilon}{\epsilon+ \|a_j -a_{j'}\|}.
\]
\end{lmm}
\begin{proof}
Choose $C_0$ such that if $\|a_j -a_{j'}\|\ge C_0 \epsilon$, then $\|x'-y'\|\ge \frac{1}{2}(\epsilon + \|a_j - a_{j'}\|)$ for any $x'\in B_j$ and $y'\in B_{j'}$. It is easy to see that $C_0$ can be chosen to be a universal constant. Then by Lemma \ref{glemma},
\begin{align*}
-G(x_j,y) + h_{j,j'} &= \frac{1}{c_{j'}} \int_{B_{j'}} (-G(x_j,y) + G(x_j, y')) da(y')\\
&\le \frac{1}{c_{j'}} \int_{B_{j'}} \frac{\|y-y'\|}{\|x_j - y'\|} da(y')\\
&\le \frac{C_1\epsilon}{C_2 \epsilon^2}\int_{B_{j'}} \frac{1}{\epsilon + \|a_j - a_{j'}\|} da(y')\le \frac{C_3\epsilon}{\epsilon + \|a_j - a_{j'}\|}.
\end{align*}
But by Lemma \ref{glemma} and the fact that $y\in B_{j'}$, we also have that
\begin{align*}
-h_{j,j'} + G(x_j,y)  &= \frac{1}{c_{j'}} \int_{B_{j'}} (-G(x_j,y') + G(x_j, y)) da(y')\\
&\le \frac{1}{c_{j'}} \int_{B_{j'}} \frac{\|y-y'\|}{\|x_j - y\|} da(y')\\
&\le \frac{C_1\epsilon}{C_2 \epsilon^2}\int_{B_{j'}} \frac{1}{\epsilon + \|a_j - a_{j'}\|} da(y')\le \frac{C_3\epsilon}{\epsilon + \|a_j - a_{j'}\|}.
\end{align*}
This completes the proof.
\end{proof}
Let $C_0$ be the universal constant from Lemma \ref{hjjbound}. For each $1\le j\le k$, let $N_j$ be the set of $1\le j'\le n$ such that $\|a_j - a_{j'}\|\le C_0\epsilon$. It is easy to see that the size of $N_j$ is bounded above by a universal constant (depending on the value of $C_0$).

Define a function $f:(S^2)^n \to \R$ as
\[
f(y_1,\ldots,y_n) := -4\sum_{j=1}^k \sum_{l=1}^n\talpha_jG(x_j, y_l)-4b_n^2\sum_{1\le l<l'\le n}G(y_l, y_{l'}),
\]
and a function $g$ on $\cp'$ (not to be confused with the metric $g$ on $\C$)  as
\[
g(\rho) := -n L(\rho) - 2b_n^2 n^2 R(\rho). 
\]
The following lemma gives an upper bound on $f(y_1,\ldots,y_n)$ in terms of $g(\rho_m)$, where $\rho_m$ is a probability density function constructed using the points $y_1,\ldots,y_n$.
\begin{lmm}\label{glemma3}
Take any distinct $y_1,\ldots,y_n \in S^2\setminus\{x_1,\ldots,x_k\}$. For $1\le j\le L$, let $m_j$ be the number of $l$ such that $y_l\in B_j$, and let $m := (m_1,\ldots,m_L)$. Let $\rho_m$ be the probability density function
\[
\rho_m(x) := \frac{m_j}{nc_j } \ \ \text{ if } x\in B_j, \ \ j=1,\ldots,L.
\]
For each $l$, let $j(l)$ be the index $j$ such that $y_l\in B_j$. For $1\le j\le k$ and $1\le l\le n$, let $\delta_{j,l} = 1$ if $j(l)\in N_j$ and $0$ otherwise. Then
\begin{align*}
f(y_1,\ldots,y_n) &\le g(\rho_m) + C\ln (1/\epsilon)+ \frac{C\epsilon}{n}\sum_{1\le j,j'\le L} \frac{m_j m_{j'}}{\epsilon + \|a_{j} - a_{j'}\|} \\ 
&\qquad + C\epsilon \sum_{j=1}^k\sum_{j'=1}^L \frac{m_{j'}}{\epsilon + \|a_j - a_{j'}\|} + 4\sum_{j=1}^k \sum_{l=1}^n \talpha_j (h_{j,j(l)} - G(x_j, y_l))\delta_{j,l}. 
\end{align*}
\end{lmm}
\begin{proof}
Note that for each $1\le j\le k$,
\begin{align*}
\sum_{l=1}^n h_{j,j(l)} &= \sum_{l=1}^n \frac{1}{c_{j(l)}}\int_{B_{j(l)}} G(x_j,y) da(y)\\
&= \sum_{j'=1}^L \frac{m_{j'}}{c_{j'}}\int_{B_{j'}} G(x_j,y) da(y)= n \int_{S^2} G(x_j, y) \rho_m(y) da(y).
\end{align*}
Similarly,
\begin{align*}
\sum_{1\le l<l'\le n} g_{j(l), j(l')} &= \sum_{1\le l<l'\le n} \frac{1}{c_{j(l)} c_{j(l')}}\int_{B_{j(l)}} \int_{B_{j(l')}} G(x,y) da(x) da(y)\\
&= \sum_{1\le j<j'\le L} \frac{m_j m_{j'}}{c_j c_{j'}}\int_{B_j} \int_{B_{j'}}G(x,y) da(x) da(y)\\
&\qquad + \sum_{j=1}^L \frac{{m_j \choose 2}}{c_j^2} \int_{B_j} \int_{B_j}G(x,y) da(x) da(y)\\
&= \frac{n^2}{2}\int_{(S^2)^2} G(x,y) \rho_m(x)\rho_m(y) da(x) da(y) -\sum_{j=1}^L \frac{m_j}{2}g_{j,j}.
\end{align*}
Combining the last two displays and observing that $|g_{j,j}|\le C\ln(1/\epsilon)$, we get
\begin{align}
-4\sum_{j=1}^k \sum_{l=1}^n \talpha_j h_{j,j(l)}- 4b_n^2 \sum_{1\le l<l'\le n} g_{j(l), j(l')}&\le g(\rho_m) + 4b_n^2\sum_{j=1}^L \frac{m_j}{2}g_{j,j}\notag \\
&\le g(\rho_m) + Cb_n^2 n \ln(1/\epsilon).\label{gl1}
\end{align}
Now,  for any $1\le l<l'\le n$, Lemma \ref{gjjbound} gives
\begin{align}
-4b_n^2 G(y_l, y_{l'}) &\le -4b_n^2 g_{j(l), j(l')} +\frac{C\epsilon}{\epsilon+ \|a_{j(l)}- a_{j(l')}\|}. \label{gl2}
\end{align}
Next, take any $1\le j\le k$ and $1\le l\le n$. If $j(l)\notin N_j$, then Lemma \ref{hjjbound} gives
\[
-4\talpha_j G(x_j, y_l) \le -4\talpha_j h_{j, j(l)} + \frac{C\epsilon}{\epsilon + \|a_j - a_{j(l)}\|}. 
\]
Thus, we get 
\begin{align}\label{gl3}
-4\talpha_j G(x_j, y_l) &\le \biggl(-4\talpha_j h_{j, j(l)} + \frac{C\epsilon}{\epsilon + \|a_j - a_{j(l)}\|}\biggr)(1-\delta_{j,l}) - 4\talpha_jG(x_j,y) \delta_{j,l}\notag \\
&\le   -4\talpha_j h_{j, j(l)} + \frac{C\epsilon}{\epsilon + \|a_j - a_{j(l)}\|}(1-\delta_{j,l}) \notag\\
&\qquad \qquad + 4\talpha_j( - G(x_j,y) + h_{j, j(l)}) \delta_{j,l}.
\end{align}
Combining the inequalities \eqref{gl1}, \eqref{gl2} and \eqref{gl3}, and observing that $b_n^2\le Cn^{-1}$, we get 
\begin{align*}
f(y_1,\ldots,y_n) &\le g(\rho_m) + C\ln (1/\epsilon) + \frac{C\epsilon}{n}\sum_{1\le l< l'\le n} \frac{1}{\epsilon + \|a_{j(l)} - a_{j(l')}\|} \\
&\qquad + C\epsilon \sum_{j=1}^k\sum_{l=1}^n \frac{1}{\epsilon + \|a_j - a_{j(l)}\|}  + 4\sum_{j=1}^k \sum_{l=1}^n \talpha_j (h_{j,j(l)} - G(x_j, y_l))\delta_{j,l}.
\end{align*}
Now note that 
\begin{align*}
\sum_{1\le l< l'\le n} \frac{1}{\epsilon + \|a_{j(l)} - a_{j(l')}\|} &\le \sum_{1\le j,j'\le L} \frac{m_j m_{j'}}{\epsilon+\|a_j - a_{j'}\|}.
\end{align*}
Similarly,
\begin{align*}
\sum_{j=1}^k\sum_{l=1}^n \frac{1}{\epsilon + \|a_j - a_{j(l)}\|} &\le \sum_{j=1}^k\sum_{j'=1}^L \frac{m_{j'}}{\epsilon + \|a_j - a_{j'}\|}.
\end{align*}
Combining these observations with the upper bound on $f(y_1,\ldots,y_n)$ obtained above, we get the desired inequality.
\end{proof}

Let $I_n$ denote the same integral as in the proof of Lemma \ref{lowerbound}. Then we can rewrite $(4\pi)^{-n}I_n$ as $\E(e^{f(U_1,\ldots, U_n)})$, 
where $U_1,\ldots,U_n$ are i.i.d.~uniform random points from $S^2$. 
For each $1\le j\le L$, let $M_j$ be the number of $i$ such that $U_i\in B_j$, and let $M := (M_1,\ldots,M_L)$. Let $\rho_M$ be defined as in Lemma \ref{glemma3}. For $1\le l\le L$, let $J_l$ be the index $j$ such that $U_l\in B_j$. For $1\le j\le k$ and $1\le l\le n$, let $\Delta_{j,l}=1$ if $J_l\in N_j$ and $0$ otherwise. Then by Lemma \ref{glemma3},
\begin{align*}
f(U_1,\ldots,U_n) &\le g(\rho_M) + C\ln (1/\epsilon) +  \frac{C\epsilon}{n}\sum_{1\le j,j'\le L} \frac{M_j M_{j'}}{\epsilon + \|a_{j} - a_{j'}\|} \\ 
&\qquad + C\epsilon \sum_{j=1}^k\sum_{j'=1}^L \frac{M_{j'}}{\epsilon + \|a_j - a_{j'}\|} + 4\sum_{j=1}^k \sum_{l=1}^n \talpha_j (h_{j,J_l} - G(x_j, U_l))\Delta_{j,l}.
\end{align*}
Let $\E'$ denote conditional expectation given $J_1,\ldots,J_n$. Note that the  $M_j$'s and the $\Delta_{j,l}$'s are functions of $J_1,\ldots,J_n$. Note also that given $J_1,\ldots, J_n$, the random points $U_1,\ldots,U_n$ are still independent, but $U_l$ is uniformly distributed on $B_{J_l}$ for each $l$. Lastly, note that for each $l$, $\Delta_{j,l}$ can be $1$ for at most one index $j$ if $\epsilon$ is small enough. Let us call this index $K_l$ if it exists. Let $R$ be the set of all $l$ for which $K_l$ exists. Then by the preceding observartions,
\begin{align*}
&\E'\biggl[\exp\biggl( 4\sum_{j=1}^k \sum_{l=1}^n \talpha_j (h_{j, J_l} - G(x_j, U_l))\Delta_{j,l}\biggr)\biggr] \\
&= \E'\biggl[\prod_{l\in R} e^{4\talpha_{K_l} (h_{K_l,J_l}- G(x_{K_l}, U_l))}\biggr]\\
&=  \prod_{l\in R}\E'[ e^{4\talpha_{K_l}(h_{K_l,J_l} - G(x_{K_l}, U_l))}]\\
&= \prod_{l\in R} \frac{e^{4\talpha_{K_l}h_{K_l,J_l} }}{c_{J_l}} \int_{B_{J_l}} e^{-4\talpha_{K_l} G(x_{K_l}, y)} da(y).
\end{align*}
Now note that since $\talpha_{K_l}>-1/2$ and $B_{J_l}$ is contained in a ball centered at $x_{K_l}$ with radius $\le C\epsilon$, 
\begin{align*}
\int_{B_{J_l}} e^{-4\talpha_{K_l} G(x_{K_l}, y)} da(y) &= \int_{B_{J_l}} 2e^{-1/2}\|x_{K_l}- y\|^{4\talpha_{K_l}} da(y)\le C\epsilon^{2+4\talpha_{K_l}}.
\end{align*}
On the other hand, 
\begin{align*}
h_{K_l,J_l} &= \frac{1}{c_{J_l}} \int_{B_{J_l}} G(x_{K_l}, y) da(y)\\
&=  \frac{1}{c_{J_l}} \int_{B_{J_l}} \biggl(-\ln\|x_{K_l}-y\| -\frac{1}{2}+\ln 2\biggr) da(y).
\end{align*}
From the above expression, it is not hard to see that $|h_{K_l,J_l} + \ln \epsilon|\le C$. Lastly, recall that $c_{J_l}\ge C\epsilon^2$. From these observations, it follows that 
\[
\frac{e^{4\talpha_{K_l}h_{K_l,J_l} }}{c_{J_l}} \int_{B_{J_l}} e^{-4\talpha_{K_l} G(x_{K_l}, y)} da(y) \le C. 
\]
Therefore, since
\[
|R| = \sum_{j=1}^k \sum_{j'\in N_j} M_{j'}, 
\]
we get
\[
\E'\biggl[\exp\biggl( 4\sum_{j=1}^k \sum_{l=1}^n \talpha_j (h_{j, J_l} - G(x_j, U_l))\Delta_{j,l}\biggr)\biggr] \le \exp\biggl(C \sum_{j=1}^k \sum_{j'\in N_j} M_{j'}\biggr).
\]
Combining the above bounds yields
\begin{align}\label{eprime}
\E'(e^{f(U_1,\ldots,U_n)}) &\le \exp(g(\rho_M) + C\ln (1/\epsilon) + T_1 + T_2 + T_3), 
\end{align}
where 
\begin{align*}
T_1 &:= \frac{C\epsilon}{n}\sum_{1\le j,j'\le L} \frac{M_j M_{j'}}{\epsilon + \|a_{j} - a_{j'}\|}, \\
T_2 &:= C\epsilon \sum_{j=1}^k\sum_{j'=1}^L \frac{M_{j'}}{\epsilon + \|a_j - a_{j'}\|},\\
T_3 &:= C\sum_{j=1}^k \sum_{j'\in N_j} M_{j'}.
\end{align*}
Take any $\theta > 1$. Let $\theta' := \theta/(\theta-1)$, so that 
\[
\frac{1}{\theta} + \frac{1}{3\theta'} + \frac{1}{3\theta'} + \frac{1}{3\theta'} = 1.
\]
Then by H\"older's inequality and the inequality \eqref{eprime}, we get
\begin{align}\label{fu1un}
\E(e^{f(U_1,\ldots,U_n)}) &\le e^{C\ln (1/\epsilon)}(\E(e^{\theta g(\rho_M)}))^{1/\theta} (\E(e^{3\theta' T_1})\E(e^{3\theta' T_2})\E(e^{3\theta' T_3}))^{1/3\theta'}.
\end{align}
We will now get upper bounds for the four expectations on the right. Let $P_n$ denote the set of all $L$-tuples $m = (m_1,\ldots,m_L)$, where the $m_j$'s are nonnegative integers summing to $n$. Let $p_j=c_j/4\pi$ be the probability that a uniformly chosen point from $S^2$ belongs to $B_j$.
\begin{lmm}\label{t1lemma}
We have
\[
\E(e^{3\theta' T_1}) \le C_1 \exp(C_2 \epsilon^{-2}\ln n + C_3 e^{C_2\theta'} \epsilon n). 
\]
\end{lmm}
\begin{proof}
Note that 
\begin{align*}
\E(e^{3\theta' T_1}) &= \sum_{m\in P_n}  \frac{n!\prod_{j=1}^L p_j^{m_j}}{\prod_{j=1}^Lm_j!} \exp\biggl(\frac{3C\theta' \epsilon}{n}\sum_{1\le j,j'\le L} \frac{m_j m_{j'}}{\epsilon + \|a_{j} - a_{j'}\|}\biggr).
\end{align*}
By Stirling's formula,
\begin{align}\label{stirling}
\frac{n!}{\prod_{j=1}^Lm_j!} &\le C_1 e^{C_2 \ln n + C_3 L} \frac{n^n}{\prod_{j=1}^L m_j^{m_j}} =  C_1 e^{C_2 \ln n + C_3 L} \prod_{j=1}^L (m_j/n)^{-m_j}.
\end{align}
Combining the above bounds (and recalling that $L\le C\epsilon^{-2}$), we get
\begin{align*}
\E(e^{3\theta' T_1})  &\le C_1 e^{C_2 \ln n + C_3 \epsilon^{-2}}  \sum_{m\in P_n} \exp\biggl(\frac{C\theta' \epsilon}{n}\sum_{1\le j,j'\le L} \frac{m_j m_{j'}}{\epsilon + \|a_{j} - a_{j'}\|}\\
&\hskip2in - \sum_{j=1}^L m_j \ln \frac{m_j/n}{p_j}\biggr).
\end{align*}
Since the size of $P_n$ is $\le (n+1)^L$, this shows that 
\begin{align*}
\E(e^{3\theta' T_1})  &\le C_1 e^{C_2  \epsilon^{-2}\ln n}  \max_{m\in P_n} e^{n f(m_1/n, \ldots, m_L/n)}, 
\end{align*}
where $f$ is defined as 
\begin{align*}
f(x_1,\ldots,x_n) := C_0\theta' \epsilon\sum_{1\le j,j'\le L} \frac{x_j x_{j'}}{\epsilon + \|a_{j} - a_{j'}\|} - \sum_{j=1}^L x_j \ln \frac{x_j}{p_j}
\end{align*}
for some suitable constant $C_0$. Consequently, if 
\[
\Delta_L := \{x = (x_1,\ldots,x_L)\in [0,1]^L: x_1+\cdots+x_L = 1\},
\]
then 
\begin{align}\label{t1main}
\E(e^{3\theta' T_1})  &\le C_1 e^{C_2  \epsilon^{-2}\ln n}  \max_{x\in \Delta_L} e^{n f(x)}.
\end{align}
Now, note that 
\begin{align*}
\frac{\partial f}{\partial x_j} &= 2C_0 \theta' \epsilon \sum_{j'=1}^L \frac{x_{j'}}{\epsilon + \|a_{j} - a_{j'}\|} - \ln \frac{x_j}{p_j}-L.
\end{align*}
If any $x_j$ tends to $0$, the corresponding partial derivative tends to $\infty$. Thus, $f$ attains its maximum value at one or more interior points of the simplex $\Delta_L$. Let $\hat{x}$ be such a point. Then by the method of Lagrange multipliers, we see that the following set of equations must be satisfied for some $\lambda \in \R$:
\[
\lambda + 2C_0 \theta' \epsilon \sum_{j'=1}^L \frac{\hat{x}_{j'}}{\epsilon + \|a_{j} - a_{j'}\|} = \ln \frac{\hat{x}_j}{p_j}, \ \ \text{ for } j=1,\ldots,L.
\]
This implies, in particular, that $\lambda \le \ln(\hat{x}_j/p_j)$ for each $j$. Rewriting this as $p_j e^\lambda \le \hat{x}_j$ and summing over $j$ on both sides, we get $e^\lambda \le 1$. Thus, exponentiating both sides of the above identity, we get
\begin{align}\label{hatxupp}
\hat{x}_j &= p_j e^\lambda \exp\biggl(2C_0 \theta' \epsilon \sum_{j'=1}^L \frac{\hat{x}_{j'}}{\epsilon + \|a_{j} - a_{j'}\|}\biggr) \le p_j e^{2C_0\theta'}. 
\end{align} 
By Jensen's inequality,
\[
\sum_{j=1}^L \hat{x}_j \ln \frac{\hat{x}_j}{p_j} = -\sum_{j=1}^L \hat{x}_j \ln \frac{p_j}{\hat{x}_j} \ge -\ln \biggl(\sum_{j=1}^L \hat{x}_j \frac{p_j}{\hat{x}_j}\biggr) = 0.
\]
Thus, by the inequality \eqref{hatxupp}, we get
\begin{align}\label{fhatx}
f(\hat{x}) &\le C_0\theta' \epsilon\sum_{1\le j,j'\le L} \frac{\hat{x}_j \hat{x}_{j'}}{\epsilon + \|a_{j} - a_{j'}\|}\notag \\
&\le C_0\theta' e^{4C_0\theta'} \epsilon\sum_{1\le j,j'\le L} \frac{p_j p_{j'}}{\epsilon + \|a_{j} - a_{j'}\|}.
\end{align}
Note that $p_j p_{j'}\le C\epsilon^4$. Next, note that  the number of $(j,j')$ such that $ \|a_{j} - a_{j'}\|$ lies between $r\epsilon$ and $(r+1)\epsilon$ for some nonnegative integer $r$ is bounded by $Cr\epsilon^{-2}$, and this number is zero if $r> C_1/\epsilon$. Thus,
\begin{align*}
\sum_{1\le j,j'\le L} \frac{p_j p_{j'}}{\epsilon + \|a_{j} - a_{j'}\|} &\le \sum_{0\le r\le C_1/\epsilon} \frac{C_2 \epsilon^4r\epsilon^{-2}}{(r+1)\epsilon}\le C_3.
\end{align*}
Using this in equation \eqref{fhatx}, we get that 
\[
f(\hat{x}) \le C_1e^{C_2\theta'}\epsilon.
\]
Finally, using this in equation \eqref{t1main} completes the proof.
\end{proof}

\begin{lmm}\label{t2lemma}
We have 
\[
\E(e^{3\theta' T_2})\le e^{C_1e^{C_2\theta'} \epsilon n}.
\]
\end{lmm}
\begin{proof}
Let $\xi_{j,l} := 1$ if $U_l\in B_j$ and $0$ otherwise. Then note that 
\begin{align*}
T_2 = C\epsilon \sum_{j=1}^k\sum_{j'=1}^L \sum_{l=1}^n \frac{\xi_{j',l}}{\epsilon + \|a_j - a_{j'}\|}.
\end{align*}
By the independence of $U_1,\ldots, U_n$, this gives
\begin{align}
\E(e^{3\theta' T_2}) &= \prod_{l=1}^n \E\biggl[\exp\biggl(C\theta'\epsilon \sum_{j=1}^k\sum_{j'=1}^L \frac{\xi_{j',l}}{\epsilon + \|a_j - a_{j'}\|}\biggr)\biggr]\notag \\
&= \prod_{l=1}^n \biggl[\sum_{j'=1}^L p_{j'} \exp\biggl(C\theta'\epsilon \sum_{j=1}^k \frac{1}{\epsilon + \|a_j - a_{j'}\|}\biggr)\biggr].\label{t2mid}
\end{align}
Now, note that the term within the exponent above is bounded by $C_1\theta'$. Now, for any $x\ge 0$, $e^y\le 1+e^x y$ for all $y\in [0,x]$. Thus, 
\begin{align*}
\sum_{j'=1}^L p_{j'} \exp\biggl(C\theta' \epsilon \sum_{j=1}^k \frac{1}{\epsilon + \|a_j - a_{j'}\|}\biggr) &\le \sum_{j'=1}^L p_{j'} \biggl(1+ C_1e^{C_2\theta'}\epsilon \sum_{j=1}^k \frac{1}{\epsilon + \|a_j - a_{j'}\|}\biggr)\\
&= 1 + C_1e^{C_2\theta'}\epsilon \sum_{j'=1}^L p_{j'} \biggl( \sum_{j=1}^k \frac{1}{\epsilon + \|a_j - a_{j'}\|}\biggr)\\
&\le 1 + C_3e^{C_2\theta'}\epsilon^3 \sum_{j'=1}^L \sum_{j=1}^k \frac{1}{\epsilon + \|a_j - a_{j'}\|}. 
\end{align*}
As in the proof of Lemma \ref{t1lemma}, it is easy to see that the double sum in the last line is bounded by $C\epsilon^{-2}$. Thus, the quantity on the left is bounded by $1+C_1 e^{C_2\theta'}\ep$. Using this in equation \eqref{t2mid} completes the proof.
\end{proof}

\begin{lmm}\label{t3lemma}
We have
\[
\E(e^{3\theta' T_3}) \le e^{C_1 e^{C_2\theta'} \epsilon^2 n}.
\]
\end{lmm}
\begin{proof}
Let $\xi_{j,l}$ be as in the proof of Lemma \ref{t2lemma}, so that
\begin{align*}
T_3 := C\sum_{j=1}^k \sum_{j'\in N_j} \sum_{l=1}^n \xi_{j',l}.
\end{align*}
By the independence of $U_1,\ldots,U_n$ (and assuming that $\epsilon$ is so small that $N_1,\ldots,N_k$ are disjoint), this gives
\begin{align*}
\E(e^{3\theta' T_3}) &\le \prod_{l=1}^n\E\biggl[\exp\biggl(C\theta' \sum_{j=1}^k \sum_{j'\in N_j} \xi_{j',l}\biggr)\biggr]\\
&= \biggl(\sum_{j=1}^k \sum_{j'\in N_j}p_{j'}e^{C\theta'} +  1- \sum_{j=1}^k \sum_{j'\in N_j}p_{j'}\biggr)^n \\
&\le (1 + C_1 e^{C_2\theta'}\epsilon^2)^n\le e^{C_1 e^{C_2\theta'} \epsilon^2 n}.
\end{align*}
This completes the proof.
\end{proof}
\begin{lmm}\label{t0lemma}
Define
\begin{align}\label{sthetadef}
S_\theta(\rho) := L(\rho) + 2\beta  R(\rho) + \frac{H(\rho)}{\theta}.
\end{align}
Then 
\[
\E(e^{\theta g(\rho_M)}) \le C_1 e^{C_2  \epsilon^{-2}\ln n}\sup_{\rho\in \cp'} e^{-n \theta S_\theta(\rho) - n\ln 4\pi}.
\]
\end{lmm}
\begin{proof}
Note that 
\begin{align*}
\E(e^{\theta g(\rho_M)}) &= \sum_{m\in P_n}\frac{n!\prod_{j=1}^L p_j^{m_j}}{\prod_{j=1}^Lm_j!} e^{\theta g(\rho_m)}.
\end{align*}
Applying the bound from equation \eqref{stirling}, this gives
\begin{align*}
\E(e^{\theta g(\rho_M)}) &\le  C_1 e^{C_2 \ln n + C_3 \epsilon^{-2}}\sum_{m\in P_n} \exp\biggl(\theta g(\rho_m)  - \sum_{j=1}^L m_j \ln \frac{m_j/n}{p_j}\biggr).
\end{align*}
Now, note that 
\begin{align*}
\sum_{j=1}^L m_j \ln \frac{m_j/n}{p_j} &= n \sum_{j=1}^L c_j\frac{m_j}{nc_j}\biggl( \ln \frac{m_j}{nc_j} + \ln 4\pi\biggr)\\
&= n \int_{S^2} \rho_m(x)(\ln \rho_m(x) + \ln 4\pi) da(x)\\
&= n H(\rho_m) + n \ln 4\pi.
\end{align*}
Also, note that
\begin{align*}
g(\rho_m) &= -n L(\rho_m) - 2b_n^2n^2 R(\rho_m)\\
&= -n L(\rho_m) -\frac{2\beta n^2}{n-1} R(\rho_m)\le -nL(\rho_m) - 2\beta n R(\rho_m),
\end{align*}
where the inequality holds because $R(\rho_m)\ge 0$. 
Thus, we get
\begin{align*}
\E(e^{\theta g(\rho_M)})  &\le C_1 e^{C_2 \ln n + C_3 \epsilon^{-2}}|P_n|\max_{m\in P_n} e^{\theta g(\rho_m) - nH(\rho_m)-n \ln 4\pi}\\
&\le C_1 e^{C_2  \epsilon^{-2}\ln n}\max_{m\in P_n} e^{-n \theta S_\theta(\rho_m) - n\ln 4\pi}\\
&\le C_1 e^{C_2  \epsilon^{-2}\ln n}\sup_{\rho\in \cp'} e^{-n \theta S_\theta (\rho) - n\ln 4\pi}.
\end{align*}
This completes the proof.
\end{proof}

Combining Lemmas \ref{t1lemma}, \ref{t2lemma}, \ref{t3lemma} and \ref{t0lemma}, we arrive at the following result, which `almost' proves the opposite direction in Theorem \ref{semilimit}.
\begin{lmm}\label{semiupperbound}
Take any $\theta > 1$. Let $S_\theta$ be defined as in equation \eqref{sthetadef}. 
In the setting of Theorem \ref{semilimit}, we have
\begin{align*}
&\limsup_{n\to \infty} \frac{1}{n} \log C(\talpha_1/b_n,\ldots,\talpha_k/b_n; x_1,\ldots,x_k;b_n; \tmu/b_n^2)+i\pi \\
&\le 1+ \ln \tmu - \ln \beta + (\ln 4 -1 )\sum_{j=1}^k \frac{\talpha_j}{\beta^2} + \sum_{j=1}^k \frac{\talpha_j (1+\talpha_j)}{\beta^2} \ln g(\sigma(x_j))\\
&\qquad \qquad -\frac{4}{\beta}\sum_{1\le j<j'\le k}\talpha_j\talpha_{j'}G(x_j, x_{j'}) - \inf_{\rho\in \cp'} S_\theta(\rho).
\end{align*}
\end{lmm}
\begin{proof}
Let $I_n$ be the integral from the proof of Lemma \ref{lowerbound}. Lemmas \ref{t1lemma}, \ref{t2lemma}, \ref{t3lemma} and \ref{t0lemma}, together with the inequality \eqref{fu1un}, give
\begin{align*}
\limsup_{n\to\infty} \frac{1}{n}\ln I_n &\le \frac{C_1}{\theta'}e^{C_2\theta'}\epsilon - \inf_{\rho\in \cp'} S_\theta(\rho).
\end{align*}
Taking $\epsilon \to 0$ completes the proof.
\end{proof}

Thus, to complete the proof of Theorem \ref{semilimit}, we need to show that $\theta$ can be taken to $1$ in Lemma \ref{semiupperbound} without spoiling the result. It turns out that this is related to the existence of a minimizer of $S$, so we proceed to prove both things at once. First let us generalize the definition of $H$ by defining, for any finite measure $\mu$ on $S^2$ and any measurable function $f:S^2 \to [0,\infty)$, 
\[
H_\mu(f) := \int_{S^2} f(x)\ln f(x) d\mu(x).
\]
It is easy to see that the above integral is well-defined and takes value in $(-\infty,\infty]$, because the function $x\ln x$ is bounded below on $[0,\infty]$ and $\mu$ is a finite measure. The following lemma gives a variational formula for $H_\mu(\rho)$.
\begin{lmm}\label{convlmm1}
For any finite measure $\mu$ on $S^2$ and any $\rho \in\cp$, 
\[
H_\mu(\rho) = \sup_{g\in C(S^2)} \biggl(\int_{S^2} \rho(x) g(x) d\mu(x) - \int_{S^2} e^{g(x)-1} d\mu(x)\biggr).
\]
\end{lmm}
\begin{proof}
Take any $g\in C(S^2)$. Let 
\[
\phi(g) := \int_{S^2} \rho(x) g(x) d\mu(x) - \int_{S^2} e^{g(x)-1} d\mu(x).
\]
An easy verification shows that for any $a\ge 0$,
\[
a\ln a = \sup_{b\in \R}(ab - e^{b-1}).
\]
Thus, for any $x\in S^2$,
\[
\rho(x) g(x) - e^{g(x)-1} \le \rho(x) \ln \rho(x).
\]
This shows that $H_\mu(\rho)\ge \sup_{g\in C(S^2)} \phi(g)$. Next, let 
\[
\rho_n(x) := 
\begin{cases}
n &\text{ if } \rho(x) > n,\\
\rho(x) &\text{ if } 1/n \le \rho(x)\le n,\\
1/n &\text{ if } \rho(x) < 1/n
\end{cases}
\]
for each positive integer $n$. Next, take any $n$, and any positive integer $m$. By Lusin's theorem, the finiteness of $\mu$, and the fact that $\rho_n$ is measurable and takes values in $[1/n, n]$, we can find a continuous function $\rho_{n,m}:S^2 \to [1/n, n]$ such that $\rho_{n,m} = \rho_n$ on a set $A_{n,m}$ whose complement satisfies $\mu(A_{n,m}^c) \le 2^{-m}$.  Let $g_{n,m} := 1+\ln \rho_{n,m}$. Then $g_{n,m}\in C(S^2)$, and 
\begin{align*}
\phi(g_{n,m}) &= \int_{S^2}\rho(x)(1+\ln \rho_{n,m}(x)) d\mu(x) - \int_{S^2}\rho_{n,m}(x) d\mu(x).
\end{align*}
Thus, if we let $g_n(x) := 1+\ln\rho_n(x)$, then 
\begin{align*}
|\phi(g_{n,m})-\phi(g_n)| &\le 2\ln n \int_{A_{n,m}^c} \rho(x) d\mu (x) + n \int_{A_{n,m}^c} d\mu(x). 
\end{align*}
Fix $n$. Since the sum of $\mu(A_{n,m}^c)$ over all $m$ is finite, almost every $x\in S^2$ belongs to at most finitely many of the $A_{n,m}^c$'s (by the Borel--Cantelli lemma). This shows that $1_{A_{n,m}^c}\to 0$ a.e.~with respect to $\mu$. Thus, by the dominated convergence theorem,
\[
\lim_{m\to \infty}  \int_{A_{n,m}^c} \rho(x) d\mu(x)  = 0.
\]
We conclude that 
\[
\lim_{m\to \infty} \phi(g_{n,m}) = \phi(g_n) = \int_{S^2}\rho(x)(1+\ln \rho_n(x)) d\mu(x) - \int_{S^2}\rho_n(x) d\mu(x).
\]
Let $\ln^+x := \max\{\ln x, 0\}$ and $\ln^-x := -\min\{\ln x, 0\}$. Then $\rho \ln^+\rho_n$ increases to $\rho \ln^+ \rho$ pointwise. Thus,
\[
\lim_{n\to\infty}\int_{S^2}\rho(x)\ln^+\rho_n(x) d\mu(x) = \int_{S^2}\rho(x)\ln^+\rho(x) d\mu(x) 
\]
by the monotone convergence theorem. Similarly, 
\begin{align*}
\lim_{n\to\infty}\int_{S^2}\rho(x)\ln^-\rho_n(x) d\mu(x) &= \int_{S^2}\rho(x)\ln^-\rho(x) d\mu(x),
\end{align*}
and $\int_{S^2} \rho_n(x)d\mu(x)\to \int_{S^2}\rho(x) d\mu(x)$. Combining, we get that
\[
\lim_{n\to \infty} \phi(g_n) = H_\mu(\rho).
\] 
Thus, 
\[
H_\mu(\rho) = \lim_{n\to \infty}\lim_{m\to\infty} H(\rho_{n,m}) = \lim_{n\to \infty}\lim_{m\to\infty}\phi(g_{n,m}).
\]
This shows that $H_\mu(\rho)\le \sup_{g\in C(S^2)} \phi(g)$. 
\end{proof}
The next lemma gives a variational formula for $R(\rho)$. Recall that $c(g)$ denotes the average value of a function $g$ on $S^2$.
\begin{lmm}\label{convlmm2}
For any $\rho\in \cp$,
\begin{align*}
R(\rho) &= \sup_{g\in C^\infty(S^2), \, c(g) = 0} \biggl(2\int_{S^2} \rho(x) g(x) da(x) +\frac{1}{2\pi} \int_{S^2} g(x)\Delta g(x) da(x) \biggr).
\end{align*}
\end{lmm}
\begin{proof}
Take any $g\in C^\infty(S^2)$ with $c(g) = 0$. Define 
\[
\phi(g) := 2 \int_{S^2} \rho(x) g(x) da(x) +\frac{1}{2\pi} \int_{S^2} g(x)\Delta g(x) da(x).
\]
Let $h= -\frac{1}{2\pi}\Delta g$. Since $G = (-\frac{1}{2\pi}\Delta)^{-1}$ on functions that integrate to zero, we have
\[
g(x) = \int_{S^2} G(x,y) h(y) da(y).
\]
Thus, if $R(\rho)$ is finite, then 
\begin{align*}
\phi(g) &= 2\int_{(S^2)^2} \rho(x) h(y) G(x,y) da(x)da(y) - \int_{(S^2)^2} h(x)h(y)G(x,y) da(x)da(y)\\
&= R(\rho) - R(\rho-h)\le R(\rho),
\end{align*}
where the last inequality holds because $G$ is a positive definite kernel. If $R(\rho)=\infty$, then $\phi(g)\le R(\rho)$ anyway. This proves one direction of the lemma. 

For the converse, take any $\rho \in \cp$. Define $\rho_n:= \min\{\rho(x), n\}$. Applying Lusin's theorem as in the proof of Lemma \ref{convlmm1}, we can find a sequence of continuous, $[0,n]$-valued  functions $\rho_{n,m} \to \rho_n$ pointwise as $m\to \infty$. Finally, for each $n$ and $m$, convolutions with the heat kernel on $S^2$ yield smooth $[0,n]$-valued functions $\rho_{n,m,l}\to \rho_{n,m}$ pointwise. 

Let $g_{n,m,l} := G \rho_{n,m,l}$, so that $\rho_{n,m,l}-c(\rho_{n,m,l}) = -\frac{1}{2\pi}\Delta g_{n,m,l}$. Then note that since $G$ integrates to zero on each coordinate,
\begin{align*}
\phi(g_{n,m,l}) &= 2\int_{(S^2)^2} \rho(x) \rho_{n,m,l}(y) G(x,y) da(x)da(y) \\
&\qquad - \int_{S^2} \rho_{n,m,l}(x)\rho_{n,m,l}(y) G(x,y) da(x)da(y). 
\end{align*}
Since $\rho_{n,m}$ and $\rho_{n,m,l}$ are bounded by $n$, we can send $l\to\infty$ and the $m\to\infty$ on both sides, and conclude that
\begin{align*}
\lim_{m\to\infty}\lim_{l\to\infty} \phi(g_{n,m,l}) &= 2\int_{(S^2)^2} \rho(x) \rho_{n}(y) G(x,y) da(x)da(y) \\
&\qquad - \int_{S^2} \rho_{n}(x)\rho_{n}(y) G(x,y) da(x)da(y).
\end{align*}
Finally, since $\rho_n$ increases to $\rho$ pointwise and $\ln(2/\|x-y\|) \ge 0$ for all $x,y\in S^2$, the monotone convergence theorem implies that as $n\to\infty$,
\begin{align*}
&\int_{(S^2)^2} \rho(x) \rho_{n}(y) G(x,y) da(x)da(y) \\
&= \int_{(S^2)^2} \rho(x)\rho_n(y) \ln \frac{2}{\|x-y\|} da(x) da(y) -\frac{1}{2} \int_{(S^2)^2} \rho(x)\rho_n(y)da(x) da(y)\\
&\to \int_{(S^2)^2} \rho(x)\rho(y) \ln \frac{2}{\|x-y\|} da(x) da(y) - \frac{1}{2}\int_{(S^2)^2} \rho(x)\rho(y)da(x) da(y)\\
&= \int_{(S^2)^2} \rho(x) \rho(y) G(x,y) da(x)da(y).
\end{align*}
as $n\to \infty$. Similarly,
\[
\int_{S^2} \rho_{n}(x)\rho_{n}(y) G(x,y) da(x)da(y) \to \int_{S^2} \rho(x)\rho(y) G(x,y) da(x)da(y).
\]
Thus, we get
\[
R(\rho) = \lim_{n\to\infty} \lim_{m\to\infty}\lim_{l\to\infty} \phi(g_{n,m,l}). 
\]
This proves that $R(\rho)\le \sup_{g\in C^\infty(S^2), \, c(g)=0} \phi(g)$.
\end{proof}

In the following discussion, let
\[
h(x) := 4\sum_{j=1}^k \talpha_j G(x_j, x),
\]
for notational convenience. We will frequently use the fact that $\int_{S^2} e^{-h(x)} da(x)<\infty$, which holds because $\talpha_j > -1/2$ for each $j$. For the same reason, we also have $\int_{S^2} e^{-\theta_0 h(x)} da(x)<\infty$ for some $\theta_0>1$. 
\begin{lmm}\label{lem:entropy-transfer}
Let $\nu$ be the finite measure on $S^2$ defined by $d\nu(x)=e^{-h(x)}da(x)$.
Let $\rho\in \cp$. Let $\mu$ be the probability measure with density $\rho$
with respect to the area measure $da$, and suppose that $\mu\ll \nu$.
Define
\[
\tau(x):=\frac{d\mu}{d\nu}(x)=\rho(x)e^{h(x)}.
\]
Assume that $H_\nu(\tau)<\infty$. 
Then
\[
\int_{S^2}|h(x)|\,\rho(x)\,da(x)<\infty,
\]
which implies that $L(\rho)$ is finite. Since 
\[
\int_{S^2}\tau(x)\ln \tau(x)\,d\nu(x)=H(\rho)+L(\rho),
\]
this implies that $H(\rho)<\infty$, and hence $\rho\in \cp'$.
\end{lmm}
\begin{proof}
Choose $c\in (0,\min\{1,\theta_0-1\})$. Since $d\nu=e^{-h}da$, we have
\[
\int_{S^2}e^{c|h(x)|}\,d\nu(x)
=
\int_{\{h\ge 0\}}e^{-(1-c)h(x)}\,da(x)
+
\int_{\{h<0\}}e^{-(1+c)h(x)}\,da(x).
\]
Because $1-c>0$, the first term is bounded by $a(S^2)=4\pi$.
Because $1+c<\theta_0$, the second term is bounded by $\int_{S^2}e^{-\theta_0 h(x)} da(x)$, which is finite. 
Hence
\[
\int_{S^2}e^{c|h(x)|}\,d\nu(x)<\infty.
\]
Next, use the elementary inequality
\[
uv\le u\ln u-u+e^v,
\qquad u\ge 0,\ v\in \mathbb{R},
\]
with $u=\tau(x)$ and $v=c|h(x)|$. Integrating with respect to $\nu$ gives
\[
c\int_{S^2}|h(x)|\tau(x)d\nu(x)
\le
\int_{S^2}\tau(x)\ln\tau(x)d\nu(x)
-\int_{S^2}\tau(x)d\nu(x)
+\int_{S^2}e^{c|h(x)|}d\nu(x).
\]
Since $H_\nu(\tau)<\infty$ by assumption, and we have shown that the last integral on the right is finite, and $\int_{S^2}\tau d\nu=\mu(S^2)=1$,
the right side is finite. Therefore,
\[
\int_{S^2}|h(x)|\,\rho(x)\,da(x)
=
\int_{S^2}|h(x)|\,\tau(x)\,d\nu(x)
<\infty.
\]
Consequently, $L(\rho)=\int h\rho da$ is finite. Finally, since $\tau=\rho e^h$, we have $\tau\ln\tau=\rho(\ln\rho+h)$. Integrating, we get
\[
\int_{S^2}\tau(x)\ln\tau(x)\,d\nu(x)=H(\rho)+L(\rho),
\]
with both sides interpreted in the extended-real sense. Since the left side is finite by assumption 
and $L(\rho)\in\mathbb{R}$, it follows that $H(\rho)<\infty$. Hence $\rho\in \cp'$.
\end{proof}

\begin{lmm}\label{optthm}
There exists unique $\hat{\rho}\in \cp'$ (up to equivalence almost everywhere) such that 
\[
S(\hat{\rho}) = \inf_{\rho\in \cp'} S(\rho).
\]
\end{lmm}
\begin{proof}
Let $\{\rho_n\}_{n\ge 1}$ be a sequence in $\cp'$ such that $S(\rho_n) \to  \inf_{\rho\in \cp'} S(\rho)$. For each $n$, let $\mu_n$ be the probability measure on $S^2$ that has density $\rho_n$ with respect to the area measure. Passing to a subsequence if necessary, let us assume that $\mu_n$ converges weakly to a probability measure $\mu$. Let $\rho_0 \equiv \frac{1}{4\pi}$, and let us assume without loss that for each $n$, 
\[
S(\rho_n) \le S(\rho_0)+1 = c(h) + \ln \frac{1}{4\pi} + 1 =: C_1,
\]
where $c(h) := \frac{1}{4\pi}\int_{S^2} h(x) da(x)$ is the zero mode of $h$. 
This implies, in particular, that for each $n$,
\begin{align}\label{c1bound}
C_1 &\ge 2\beta R(\rho_n) + \int_{S^2}\rho_n(x)\ln \frac{\rho_n(x)}{e^{-h(x)}} da(x) \ge \int_{S^2}\rho_n(x)\ln \frac{\rho_n(x)}{e^{-h(x)}} da(x).
\end{align}
Let $\nu$ be the finite measure on $S^2$ which has density $e^{-h(x)}$ with respect to the area measure. We claim that $\mu$ is absolutely continuous with respect to the measure $\nu$. To show this, take any Borel set $A\subseteq S^2$ with $\nu(A)=0$. Take any $\epsilon >0$. Note that for any $n$ and any $M>1$, the inequality \eqref{c1bound} implies that 
\begin{align*}
\int_{\{x: \rho_n(x) > Me^{-h(x)}\}} \rho_n(x) da(x) &\le \frac{1}{\ln M} \int_{S^2} \rho_n(x)\ln \frac{\rho_n(x)}{e^{-h(x)}} da(x)\le \frac{C_1}{\ln M}.
\end{align*}
Let us choose $M$ so large that the right side is less than $\epsilon$. Then for any Borel set $B\subseteq S^2$ and any $n$,
\begin{align*}
\mu_n(B) &= \int_B \rho_n(x) da(x)\\
&= \int_{B\cap \{x: \rho_n(x) > Me^{-h(x)}\}} \rho_n(x) da(x) + \int_{B\cap \{\rho_n(x) \le  Me^{-h(x)}\}} \rho_n(x) da(x)\\
&\le \int_{\{x: \rho_n(x) > Me^{-h(x)}\}} \rho_n(x) da(x) + \int_{B} M e^{-h(x)} da(x)\\
&\le \epsilon + M \nu(B).
\end{align*}
Since $\nu$ is a finite measure on $S^2$, it is regular. Thus, there exists an open set $V\supseteq A$ such that $\nu(V) \le \epsilon/M$. The above inequality shows, therefore, that $\mu_n(V) \le 2\epsilon$ for all $n$. Since $\mu_n \to \mu$ weakly, we get
\[
\mu(A)\le \mu(V) \le \liminf_{n\to \infty} \mu_n(V) \le 2\epsilon.
\]
Since $\epsilon$ is arbitrary, this proves that $\mu(A)=0$. Thus, $\mu$ is absolutely continuous with respect to $\nu$. Let $\tau$ denote its density with respect to $\nu$. Define
\[
\hat{\rho}(x) := \tau(x) e^{-h(x)}.
\]
We claim that this $\hat{\rho}$ works. To prove this, let 
\[
\tau_n(x) := \frac{\rho_n(x)}{e^{-h(x)}}. 
\]
Note that $\tau_n$ is a probability density function with respect to the measure $\nu$, and the resulting measure is simply $\mu_n$. 
Take any $g\in C(S^2)$. Then by the weak convergence of $\mu_n$ to $\mu$ and Lemma \ref{convlmm1},
\begin{align*}
\int_{S^2} \tau(x) g(x) d\nu(x) - \int_{S^2} e^{g(x)-1} d\nu(x) &= \lim_{n\to \infty} \biggl(\int_{S^2} \tau_n(x) g(x) d\nu(x) - \int_{S^2} e^{g(x)-1} d\nu(x)\biggr)\\
&\le \liminf_{n\to \infty} H_\nu(\tau_n).
\end{align*}
Taking supremum over $g\in C(S^2)$ gives
\begin{align}\label{c12}
H_\nu(\tau) \le \liminf_{n\to \infty} H_\nu(\tau_n).
\end{align}
By a similar argument using Lemma \ref{convlmm2}, we get
\[
R(\hat{\rho}) \le \liminf_{n\to\infty} R(\rho_n).
\]
Now note that
\begin{align*}
S(\hat{\rho}) &= 2\beta R(\hat{\rho}) + \int_{S^2}\hat{\rho}(x)\ln \tau(x) da(x) \\
&= 2\beta R(\hat{\rho}) + \int_{S^2}\tau(x)\ln \tau(x) d\nu(x) \\
&= 2\beta R(\hat{\rho}) + H_\nu(\tau).
\end{align*}
Combining this with the bounds obtained above, we get
\begin{align*}
S(\hat{\rho}) &\le 2\beta   \liminf_{n\to\infty} R(\rho_n)    +  \liminf_{n\to \infty} H_\nu(\tau_n)\\
&\le  \liminf_{n\to\infty} (2\beta R(\rho_n) + H_\nu(\tau_n))\\
&= \liminf_{n\to \infty} S(\rho_n) = \inf_{\rho\in \cp'} S(\rho).
\end{align*}
Finally, to see that $\hat{\rho}\in \cp'$, simply note that $H_\nu(\tau)<\infty$ (since $S(\hat{\rho})<\infty$  and $R(\hat{\rho})\ge 0$), and apply Lemma \ref{lem:entropy-transfer}. 

This proves the existence of a minimizer $\hat{\rho}$. To prove uniqueness, let $\rho_1$ and $ \rho_2$ be two minimizers of $S$ in $\cp'$. Let $\rho := \frac{1}{2}(\rho_1 + \rho_2)$. By the convexity of the function $x\ln x$, it follows that $\rho\in \cp'$. Now, note that $L$, $R$ and $H$ are convex functions. Thus, $S$ is also a convex function, and therefore,
\[
S(\rho) \le \frac{1}{2}(S(\rho_1)+S(\rho_2)). 
\]
But $\rho_1$ and $\rho_2$ are minimizers of $S$, and therefore, $S(\rho_1)\le S(\rho)$ and $S(\rho_2)\le S(\rho)$. The only way these three inequalities can be compatible with each other is if 
$S(\rho)=S(\rho_1)=S(\rho_2)$. This implies, in particular, that $H(\rho)=H(\rho_1)=H(\rho_2)$ (because $L,R$ are also convex). But $x\ln x$ is a strictly convex function. Thus, we must have $\rho=\rho_1=\rho_2$ almost everywhere on $S^2$.
\end{proof}

The following result, which is the final result of this section, shows that we can take $\theta \to 1$ in Lemma \ref{semiupperbound}. This completes all the steps required to finish the proof of Theorem~\ref{semilimit}.

\begin{lmm}
Let $S_\theta$ be defined as in equation \eqref{sthetadef}. Then
\[
\lim_{\theta \downarrow 1} \inf_{\rho\in \cp'} S_\theta(\rho) = \inf_{\rho\in \cp'} S(\rho).
\]
\end{lmm}
\begin{proof}
Clearly, $S_\theta(\rho)\le S(\rho)$ for any $\theta > 1$ and $\rho\in \cp'$, and $S_\theta(\rho)$ is a decreasing function of $\theta$ for each $\rho$. Thus,
\[
\lim_{\theta \downarrow 1} \inf_{\rho\in \cp'} S_\theta(\rho) \le \inf_{\rho\in \cp'} S(\rho),
\]
and the limit on the left exists. Thus, it suffices to prove the reverse inequality, with the limit replaced by a limit over any sequence $\{\theta_n\}_{n\ge 1}$ decreasing to $1$. Without loss, let us choose $\theta_1$ so close to $1$ that $e^{-\theta_1 h(x)}$ is integrable. Note that
\begin{align*}
S_\theta(\rho) &= \frac{1}{\theta} \biggl(\theta \int_{S^2} \rho(x) h(x)da(x) + 2\beta\theta R(\rho) + H(\rho)\biggr).
\end{align*}
The term inside the brackets on the right is the same as $S(\rho)$ but with $h$ replaced by $\theta h$ and $\beta$ replaced by $\beta\theta$. Thus, by a simple modification of Theorem \ref{optthm}, we see that for each $n$ there exists $\rho_n\in \cp'$ that minimizes $S_{\theta_n}$. Let $\mu_n$ be the probability measure with density $\rho_n$. Passing to a subsequence if necessary, let us assume that $\mu_n$ converges weakly to a probability measure $\mu$. We claim that $\mu$ is absolutely continuous with respect to $\nu$, where $\nu$ is the finite measure on $S^2$ that has density equal to $e^{-h(x)}$ with respect to the area measure.

To see this, let $\nu_n$ be the finite measure with density $e^{-\theta_n h(x)}$, so that 
\begin{align*}
S_{\theta_n}(\rho) &= 2\beta R(\rho) + \frac{ H_{\nu_n}(\rho e^{\theta_n h})}{\theta_n}.
\end{align*}
Thus, we have that for each $n$,
\begin{align}\label{hnun}
H_{\nu_n}(\rho_n e^{\theta_n h}) \le \theta_n S_{\theta_n}(\rho_n)= \theta_n \inf_{\rho\in \cp'} S_{\theta_n}(\rho).
\end{align}
Clearly, the right side is bounded above by a constant $C_1$ that does not depend on $n$ (seen, for example, by choosing $\rho\equiv \frac{1}{4\pi}$). Thus, for any $M> 1$, 
\begin{align*}
\int_{\{x: \rho_n(x) > Me^{-\theta_n h(x)}\}} \rho_n(x) da(x) &\le \frac{1}{\ln M} \int_{S^2} \rho_n(x)\ln \frac{\rho_n(x)}{e^{-\theta_n h(x)}} da(x)\\
&= \frac{H_{\nu_n}(\rho_n e^{\theta_n h})}{\ln M} \le \frac{C_1}{\ln M}.
\end{align*}
Take any $\epsilon >0$ and  choose $M$ so large that $C_1/\ln M<\epsilon$. Then for any Borel set $B\subseteq S^2$ and any $n$,
\begin{align*}
\mu_n(B) &= \int_B \rho_n(x) da(x)\\
&= \int_{B\cap \{x: \rho_n(x) > Me^{-\theta_n h(x)}\}} \rho_n(x) da(x) + \int_{B\cap \{\rho_n(x) \le  Me^{-\theta_n h(x)}\}} \rho_n(x) da(x)\\
&\le \int_{\{x: \rho_n(x) > Me^{-\theta_n h(x)}\}} \rho_n(x) da(x) + \int_{B} M e^{-\theta_n h(x)} da(x)\\
&\le \epsilon + M \nu_n(B).
\end{align*}
Take any Borel set $A\subseteq S^2$ such that $\nu(A)=0$. Since $\nu$ is a finite measure on $S^2$, it is regular. Thus, there exists an open set $V\supseteq A$ such that $\nu(V) \le \epsilon/M$.  Since $\mu_n \to \mu$ weakly, the above inequality implies that 
\[
\mu(A)\le \mu(V) \le \liminf_{n\to \infty} \mu_n(V) \le \epsilon + M \liminf_{n\to \infty} \nu_n(V).
\]
But note that for any Borel set $B\subseteq S^2$,
\[
|\nu_n(B) - \nu(B)| \le \int_{S^2} |e^{-\theta_n h(x)} - e^{-h(x)}| da(x), 
\]
and the integral on the right converges to zero as $n\to\infty$ by the dominated convergence theorem, due to the integrability of $e^{-\theta_1 h}$. Consequently, $\nu_n \to \nu$ in total variation distance. Combining this with the previous display, we get that $\mu(A)\le 2\epsilon$. 
Since $\epsilon$ is arbitrary, this proves that $\mu(A)=0$. This proves that  $\mu$ is absolutely continuous with respect to $\nu$. Let $\tau$ denote its density with respect to $\nu$, and define
\[
\rho(x) := \tau(x) e^{-h(x)}.
\]
We claim that $\rho\in \cp'$ and $S(\rho)\le \lim_{n\to\infty} S_{\theta_n}(\rho_n)$. Clearly, this will complete the proof of the lemma. To prove the claim, let 
\[
\tau_n(x) := \frac{\rho_n(x)}{e^{-\theta_nh(x)}}. 
\]
Note that $\tau_n$ is a probability density function with respect to the measure $\nu_n$, and the resulting measure is simply $\mu_n$. 
Take any $g\in C(S^2)$. Then by the weak convergence of $\mu_n$ to $\mu$ and of $\nu_n\to\nu$, and Lemma \ref{convlmm1},
\begin{align*}
&\int_{S^2} \tau(x) g(x) d\nu(x) - \int_{S^2} e^{g(x)-1} d\nu(x) \\
&= \lim_{n\to \infty} \biggl(\int_{S^2} \tau_n(x) g(x) d\nu_n(x) - \int_{S^2} e^{g(x)-1} d\nu_n(x)\biggr)\\
&\le \liminf_{n\to \infty} H_{\nu_n}(\tau_n).
\end{align*}
Taking supremum over $g\in C(S^2)$ gives
\begin{align}\label{hnun2}
H_\nu(\tau) \le \liminf_{n\to \infty} H_{\nu_n}(\tau_n) = \liminf_{n\to \infty} \theta_n^{-1}H_{\nu_n}(\tau_n),
\end{align}
since $\theta_n \downarrow 1$. 
By a similar argument using Lemma \ref{convlmm2}, we get
\[
R(\rho) \le \liminf_{n\to\infty} R(\rho_n).
\]
Now note that $S(\rho) = 2\beta R(\rho) + H_\nu(\tau)$. Combining this with the above bounds, we get
\begin{align*}
S(\rho) &\le 2\beta   \liminf_{n\to\infty} R(\rho_n)    +  \liminf_{n\to \infty} \theta_n^{-1}H_{\nu_n}(\tau_n)\\
&\le  \liminf_{n\to\infty} (2\beta R(\rho_n) + \theta_n^{-1} H_{\nu_n}(\tau_n))\\
&= \liminf_{n\to \infty} S_{\theta_n}(\rho_n).
\end{align*}
Lastly, note that $H_\nu(\tau) < \infty$ by the inequalities \eqref{hnun} and \eqref{hnun2}, and apply Lemma \ref{lem:entropy-transfer} to deduce that $H(\rho)$ is finite. This completes the proof.
\end{proof}

\subsection{Validity of the semiclassical limit}\label{validsec}
Recall that from equation \eqref{favg}, we have the formal expression
\begin{align*}
&\tilde{C}(\talpha_1/b,\ldots,\talpha_k/b;x_1,\ldots,x_k;b;\tmu/b^2 )\notag \\
&= \int\biggl(\prod_{j=1}^ke^{\chi(\talpha_j/b)(b-\talpha_j/b)} g(\sigma(x_j))^{-\Delta_{\talpha_j/b}}\biggr) \exp\biggl[\sum_{j=1}^k \biggl(\frac{2\talpha_j}{b} \phi(x_j) +\frac{2\talpha_j^2}{b^2} G(x_j,x_j)\biggr) \notag \\
&\qquad \qquad - 2Qc(\phi) -  \frac{1}{4\pi}\int_{S^2}\biggl( \phi(x)\Delta_{S^2} \phi(x) + \frac{4\pi \tmu}{b^2} e^{2b\phi(x)+2b^2G(x,x)}\biggr)da(x) \biggr]\md\phi,
\end{align*}
where $c(\phi)$ denotes the zero mode of $\phi$. 
Substituting $\psi = b \phi$ in the path integral, we see that the above path integral can be rewritten (up to a factor of $b^{-\infty}$ that we ignore, which arises due to replacing each $d\phi(x)$ by $b^{-1}d\psi(x)$, recalling that $\md \phi$ is shorthand for $\prod_{x\in S^2} d\phi(x)$ in a heuristic sense) as 
\[
\int e^{J(\psi)/b^2 - R_b(\psi)}\md\psi,
\]
where
\begin{align*}
J(\psi) &:= -\chi \sum_{j=1}^k \talpha_j^2 + \sum_{j=1}^k \talpha_j (1+\talpha_j)\ln g(\sigma(x_j)) + \sum_{j=1}^k (2\talpha_j\psi(x_j) +2\talpha_j^2 G(x_j,x_j))\notag \\
&\qquad \qquad + 2c(\psi) -  \frac{1}{4\pi}\int_{S^2}( \psi(x)\Delta_{S^2} \psi(x) + 4\pi \tmu e^{2\psi(x)})da(x),
\end{align*}
and $R_b(\psi)$ is a remainder term formally defined as
\begin{align*}
R_b(\psi) &:=\chi \sum_{j=1}^k\talpha_j -\sum_{j=1}^k \talpha_j \ln g(\sigma(x_j))+ 2c(\psi) + \tmu \int_{S^2} e^{2\psi(x)} \frac{e^{2b^2G(x,x)} -1 }{b^2} da(x)\\
&= \chi \sum_{j=1}^k\talpha_j -\sum_{j=1}^k \talpha_j \ln g(\sigma(x_j))+ 2c(\psi) + \tmu\sum_{j=0}^\infty \frac{2^{j+1} b^{2j}}{(j+1)!}\int_{S^2} e^{2\psi(x)} G(x,x)^{j+1}da(x).
\end{align*}
In particular, the formal series expansion of $R_b$ has only nonnegative powers of $b$. Thus, we may expect that as $b\to 0$, $C(\talpha_1/b,\ldots,\talpha_k/b;x_1,\ldots,x_k;b;\tmu/b^2 )$ should behave like $e^{J(\hat{\psi})/b^2}$ for some critical point $\hat{\psi}$ of $J$. Note that $J$ is not a well-defined map, since $G(x_j,x_j)=\infty$. However, this is not a serious problem, since this term has no dependence on $\psi$ and therefore does not play a role in the definition of critical points of $J$, which can be defined rigorously. But there is a more difficult hurdle. The following lemma, quoted from Subsection \ref{mainresultssec},  shows that $J$ has no critical points when the condition required for our Theorem~\ref{semilimit} is satisfied.

\nocriticallmm*

\begin{proof}
Take any $\psi$ for which all integrals in the definition of $J$ are absolutely convergent. For $t\in \R$, let $\psi_t(x) := \psi(x)+t$. Then note that 
\begin{align*}
&J(\psi_t) +\chi \sum_{j=1}^k \talpha_j^2 - \sum_{j=1}^k \talpha_j (1+\talpha_j)\ln g(\sigma(x_j)) \\
&= \sum_{j=1}^k (2\talpha_j(\psi(x_j) +t)+2\talpha_j^2 G(x_j,x_j))+ 2c(\psi+t)\notag \\
&\qquad \qquad -  \frac{1}{4\pi}\int_{S^2}( (\psi(x)+t)\Delta_{S^2} \psi(x) + 4\pi \tmu e^{2\psi(x) + 2t})da(x)\\
&= 2t\sum_{j=1}^k \talpha_j + 2t +  \sum_{j=1}^k (2\talpha_j\psi(x_j) +2\talpha_j^2 G(x_j,x_j))+ 2c(\psi)\notag \\
&\qquad \qquad -  \frac{1}{4\pi}\int_{S^2}( \psi(x)\Delta_{S^2} \psi(x) + 4\pi \tmu e^{2\psi(x)+2t})da(x).
\end{align*}
From the above formula, we get
\begin{align*}
\frac{\partial}{\partial t} J(\psi_t)\biggl|_{t=0} &= -2\beta - 2\tmu\int_{S^2}e^{2\psi(x)} da(x). 
\end{align*}
Since this is strictly negative, $\psi$ is not a critical point of $J$.
\end{proof}
So, what are the critical points of $J$? Let us now carry out a heuristic derivation of the critical point equation for $J$, which was earlier claimed to be equation \eqref{releq}. Take any two fields $\psi, \phi:S^2 \to \C$. If $\psi$ is a critical point of $J$, it must satisfy
\[
\frac{\partial}{\partial t} J(\psi+t\phi)\biggl|_{t=0} = 0.
\]
Note that 
\begin{align*}
\frac{\partial}{\partial t} J(\psi+t\phi) &= \sum_{j=1}^k 2\talpha_j \phi(x_j) + 2c(\phi) - \frac{1}{4\pi}\int_{S^2} \biggl\{\phi(x) \Delta_{S^2}(\psi(x)+t\phi(x))\\
&\qquad \qquad  + (\psi(x)+t\phi(x))\Delta_{S^2}\phi(x) + 8\pi \tmu\phi(x) e^{2(\psi(x)+t\phi(x))}\biggr\}da(x).
\end{align*}
Thus, applying integration by parts in the second step below, we get 
\begin{align*}
\frac{\partial}{\partial t} J(\psi+t\phi) \biggl|_{t=0} &= \sum_{j=1}^k 2\talpha_j \phi(x_j) + 2c(\phi) - \frac{1}{4\pi}\int_{S^2} \biggl\{\phi(x) \Delta_{S^2}\psi(x)\\
&\qquad \qquad  + \psi(x)\Delta_{S^2}\phi(x) + 8\pi \tmu\phi(x) e^{2\psi(x)}\biggr\}da(x)\\
&= \int_{S^2}\phi(x)\biggl\{\sum_{j=1}^k 2\talpha_j \delta_{x_j}(x) + \frac{2}{4\pi} - \frac{2\Delta_{S^2}\psi(x)}{4\pi} -2\tmu e^{2\psi(x)}\biggr\} da(x).
\end{align*}
If $\psi$ is a critical point of $J$, the above quantity must vanish for `every' $\phi$ (where we are not making the notion of `every' precise). At a heuristic level, this implies that 
\[
2\sum_{j=1}^k \talpha_j \delta_{x_j}(x) + \frac{1}{2\pi} - \frac{\Delta_{S^2}\psi(x)}{2\pi} -2\tmu e^{2\psi(x)} = 0
\]
for all $x\in S^2$, which is equation \eqref{releq}.

Even though $J$ has no critical points among real-valued functions (Lemma~\ref{nocritical}), it turns out that it does have critical points among complex-valued functions, and the semiclassical limit obtained in Theorem \ref{semilimit} can indeed be expressed using one such critical point. This is the content of the following theorem, quoted from Subsection \ref{mainresultssec}.

\semicritthm*

In the above theorem, $J(\hat{\psi})$ has to be divided by $\beta$ because the $J(\hat{\psi})/b^2$ gets converted to $J(\hat{\psi})/\beta$ after we  divide the logarithm by $n$, as we did in Theorem \ref{semilimit}. To prove Theorem \ref{semicrit}, we must first establish that $\hat{\rho}$ satisfies a functional equation, as mandated by the fact that $\hat{\rho}$ is a minimizer of $S$. This requires the following preliminary lemma.
\begin{lmm}\label{rhononzero}
The function $\hat{\rho}$ is nonzero and finite almost everywhere on $S^2$.
\end{lmm}
\begin{proof}
The finiteness follows from the facts that $\hat{\rho}$ is a probability density function. To prove that $\hat{\rho}$ is nonzero almost everywhere, let $A$ be the set where $\hat{\rho}$ is zero. Take any $\epsilon \in (0,1)$, and define
\begin{align*}
\rho_\epsilon(x) &:= 
\begin{cases}
\hat{\rho}(x)/(1+\epsilon a(A)) &\text{ if } x\notin A,\\
\epsilon/(1+\epsilon a(A)) &\text{ if } x\in A.
\end{cases}
\end{align*}
Then note that $\rho_\epsilon$ is a probability density function with respect to the area measure, and 
\begin{align*}
H(\rho_\epsilon) &=\int_{S^2}\rho_\epsilon(x) \ln \rho_\epsilon(x) da(x)\\
&= \frac{1}{1+\epsilon a (A)} \biggl(\int_A \hat{\rho}(x) \ln \frac{\hat{\rho}(x)}{1+\epsilon a(A)} da(x) + a(A) \epsilon \ln \frac{\epsilon}{1+\epsilon a(A)}\biggr)\\
&= \frac{H(\hat{\rho}) + a(A) \epsilon \ln\epsilon }{1+\epsilon a(A)} - \ln (1+\epsilon a(A))\\
&\le H(\hat{\rho}) + a(A) \epsilon \ln \epsilon + C\epsilon.
\end{align*}
Next, let $\chi(x) := 1$ if $x\in A$ and $0$ otherwise. Then $\rho_\epsilon = (\hat{\rho} +\epsilon \chi)/(1+\epsilon a(A))$. It is easy to see that $G\chi$ is a bounded function. Therefore
\begin{align*}
L(\rho_\epsilon) &= \frac{L(\hat{\rho}) + 4\epsilon \sum_{j=1}^k \talpha_j G\chi(x_j)}{1+\epsilon a(A)}\le L(\hat{\rho}) + C\epsilon,
\end{align*}
and 
\begin{align*}
R(\rho_\epsilon) &= \frac{1}{(1+\epsilon a(A))^2} \int_{S^2}(\hat{\rho}(x) G\hat{\rho}(x) + 2\epsilon\hat{\rho}(x) G\chi (x) + \epsilon^2 \chi(x)G\chi(x)) da(x)\\
&\le R(\hat{\rho}) + C\epsilon. 
\end{align*}
Combining the three inequalities obtained above, and remembering that $S(\hat{\rho}) \le S(\rho_\epsilon)$, we get 
$0\le a(A) \epsilon \ln \epsilon + C\epsilon$, which can be rewritten as $a(A) \le C/\ln(1/\epsilon)$. Since $\epsilon$ is arbitrary, this shows that $a(A)=0$.
\end{proof}

We can now show that $\hat{\rho}$ satisfies a certain functional equation, as a critical point of the function $S$.
\begin{lmm}\label{rhoeqlmm}
The function $\hat{\rho}(x)$ satisfies, for almost every $x\in S^2$, 
\begin{align}\label{rhohateq}
4\sum_{j=1}^k \talpha_j G(x_j,x) + 4\beta G\hat{\rho}(x) + \ln \hat{\rho}(x) + \lambda = 0,
\end{align}
where $\lambda$ is the number defined in equation \eqref{lambdaval}.
\end{lmm}
\begin{proof}
Let $\mu$ denote the probability measure on $S^2$ with probability density function $\hat{\rho}$ with respect to the area measure $v$. By the Lebesgue differentiation theorem, for almost every $x\in S^2$ we have
\[
\hat{\rho}(x) = \lim_{\delta \to 0} \frac{\mu(B(x,\delta))}{v(B(x, \delta))}.
\]
where $B(x,\delta)$ is the cap of radius $\delta$ centered at $x$ on $S^2$. Take any distinct $u,v$ where this holds. Note that by Lemma \ref{rhononzero}, the $\mu$-measure of any open cap is nonzero. Take any $\delta\in (0,\frac{1}{2}\|u-v\|)$. Then, choose $\epsilon>0$ that is smaller than  both $\mu(B(u,\delta))$ and $\mu(B(v,\delta))$. Define 
\begin{align*}
\rho_{\epsilon,\delta}(x) &= 
\begin{cases}
(1 + \frac{\epsilon }{\mu(B(u,\delta))})\hat{\rho}(x) &\text{ if } x\in B(u,\delta),\\
(1- \frac{\epsilon}{\mu(B(v,\delta))})\hat{\rho}(x) &\text{ if } x\in B(v,\delta),\\
\hat{\rho}(x) &\text{ if } x\notin B(u,\delta)\cup B(v,\delta).
\end{cases}
\end{align*}
It is easy to see that $\rho_{\epsilon,\delta}$ is a probability density with respect to the area measure. Let $\chi$ be the function that is $1/\mu(B(u,\delta))$ on $B(u,\delta)$, $-1/\mu(B(v,\delta))$ on $B(v,\delta)$, and $0$ everywhere else. Then $\rho_{\epsilon, \delta} = (1+\epsilon \chi)\hat{\rho}$. Using the fact that $H(\hat{\rho})<\infty$, we can now apply the dominated convergence theorem and a simple calculation to get
\begin{align*}
\frac{\partial}{\partial\epsilon}S(\rho_{\epsilon,\delta})\biggl|_{\epsilon = 0} &= 4\sum_{j=1}^k \talpha_j\int_{S^2} G(x_j,x)\chi(x)\hat{\rho}(x) da(x) + 4\beta \int_{S^2} \chi(x)\hat{\rho}(x) G\hat{\rho}(x) da(x) \\
&\qquad + \int_{S^2}\chi(x)\hat{\rho}(x) \ln \hat{\rho}(x) da(x) + \int_{S^2}\chi(x) \hat{\rho}(x) da(x).
\end{align*}
Since $S$ is minimized at $\hat{\rho}$, the above quantity must be zero. Taking $\delta\to 0$, Lebesgue's differentiation theorem  and our choice of $u,v$ give $h(u) = h(v)$, where
\begin{align*}
h(x) := 4\sum_{j=1}^k \talpha_jG(x_j,x)+ 4\beta G\hat{\rho}(x) + \ln \hat{\rho}(x) +  1.
\end{align*}
This proves that $h$ is constant almost everywhere. The value of the constant is now obtained from the condition that $\int_{S^2}\hat{\rho}(x) da(x) = 1$. 
\end{proof}

We are now ready to complete the proof of Theorem \ref{semicrit}.

\begin{proof}[Proof of Theorem \ref{semicrit}]
First, let us prove that $\hat{\psi}$ is indeed a critical point of $J$. As we already justified via heuristic calculation in the beginning of this subsection, a critical point $\psi$ must satisfy the (generalized) functional equation
\begin{align}\label{criteq}
2\sum_{j=1}^k \talpha_j \delta_{x_j}(x) + \frac{1}{2\pi} - \frac{1}{2\pi}\Delta_{S^2} \psi(x) - 2\tmu e^{2\psi(x)} = 0,
\end{align}
where $\delta_x(\cdot)$ denotes the Dirac delta at a point $x$. 
To prove that $\hat{\psi}$ satisfies this equation, recall the equation \eqref{psihatdef} defining $\hat{\psi}$. From this equation and the fact that 
\[
\Delta_{S^2} G(\cdot, y)  =-2\pi \delta_y(\cdot) + \frac{1}{2}
\]
in the sense of distributions, we get
\begin{align}
\Delta_{S^2} \hat{\psi}(x) &= -2\beta \Delta_{S^2} G\hat{\rho}(x) +4\pi \sum_{j=1}^k \talpha_j \delta_{x_j}(x) -4\pi \sum_{j=1}^k \talpha_j\notag \\
&= 4\pi \beta \hat{\rho}(x) +4\pi\sum_{j=1}^k \talpha_j \delta_{x_j}(x) + 1,\label{deltahatpsi}
\end{align}
where both sides are to be interpreted as distributions. 
But by the equation \eqref{rhohateq} satisfied by $\hat{\rho}$, as established in Lemma \ref{rhoeqlmm}, we get
\begin{align*}
\hat{\rho}(x) &= \exp\biggl(-4\sum_{j=1}^k \talpha_j G(x_j,x) - 4\beta G\hat{\rho}(x) -\lambda\biggr).
\end{align*}
Again appealing to the definition \eqref{psihatdef} of $\hat{\psi}$, we deduce from the above identity that 
\begin{align}\label{hatrhoeq}
\hat{\rho}(x) &= \exp(2\hat{\psi}(x) -\ln \beta -i\pi +\ln \tmu) = -\frac{ \tmu e^{2\hat{\psi}(x)}}{ \beta}. 
\end{align}
Plugging this into equation \eqref{deltahatpsi}, we get
\begin{align*}
\Delta_{S^2} \hat{\psi}(x) &= -4\pi \tmu e^{2\hat{\psi}(x)}+ 4\pi\sum_{j=1}^k \talpha_j \delta_{x_j}(x) + 1.
\end{align*}
But this is exactly the equation \eqref{criteq}. Thus, $\hat{\psi}$ is a critical point of $J$. Next, note that by equation \eqref{rhohateq},
\begin{align*}
S(\hat{\rho}) &= 4 \sum_{j=1}^k \talpha_j G\hat{\rho}(x_j) + 2\beta\int_{S^2} \hat{\rho}(x) G\hat{\rho}(x) da(x) + \int_{S^2}\hat{\rho}(x)\ln  \hat{\rho}(x) da(x)\\
&= -\lambda - 2\beta \int_{S^2} \hat{\rho}(x)G\hat{\rho}(x) da(x). 
\end{align*}
Now note that from the definition \eqref{psihatdef} of $\hat{\psi}$, it follows that 
\begin{align}\label{chatpsi}
c(\hat{\psi}) = -\frac{\lambda}{2} +\frac{1}{2}\ln \beta + \frac{i\pi}{2} - \frac{1}{2}\ln \tmu.
\end{align}
Thus, we may rewrite equation \eqref{psihatdef} as 
\[
\hat{\psi}(x) = -2\beta G\hat{\rho}(x) + c(\hat{\psi})  -2\sum_{j=1}^k \talpha_j G(x,x_j),
\]
Using equation \eqref{deltahatpsi} to express $\hat{\rho}$ in terms of $\Delta_{S^2} \hat{\psi}$, and the above equation to express $G\hat{\rho}$ in terms of $\hat{\psi}$, we get 
\begin{align*}
\int_{S^2} \hat{\rho}(x)G\hat{\rho}(x) da(x)&= \frac{1}{8\pi \beta^2} \int_{S^2}\biggl(\Delta_{S^2} \hat{\psi}(x) -4\pi \sum_{j=1}^k \talpha_j \delta_{x_j}(x) -1\biggr)\\
&\qquad \qquad \cdot \biggl(-\hat{\psi}(x) +c(\hat{\psi}) -2\sum_{j=1}^k \talpha_j G(x,x_j)\biggr) da(x).
\end{align*}
Now, using the facts that integrating $G$ with respect to one coordinate always yields zero, and that the integral of the Laplacian of any function vanishes, we get that the integral on the right is formally equal to
\begin{align*}
&-\int_{S^2}\hat{\psi}(x) \Delta_{S^2} \hat{\psi}(x) da(x) +4\pi\sum_{j=1}^k \talpha_j \hat{\psi}(x_j) + 4\pi c(\hat{\psi}) \\
&\qquad +  \biggl(-4\pi \sum_{j=1}^k\talpha_j -4\pi \biggr) c(\psi) +4\pi\sum_{j=1}^k \talpha_j (\hat{\psi}(x_j) - c(\hat{\psi})) +8\pi \sum_{1\le j,j'\le k} \talpha_j \talpha_{j'} G(x_j,x_{j'})\\
&= -\int_{S^2}\hat{\psi}(x) \Delta_{S^2} \hat{\psi}(x) da(x) +8\pi\sum_{j=1}^k \talpha_j \hat{\psi}(x_j) \\
&\qquad \qquad +8\pi \sum_{1\le j,j'\le k} \talpha_j \talpha_{j'} G(x_j,x_{j'}) -8\pi c(\hat{\psi})\sum_{j=1}^k \talpha_j.
\end{align*}
(We say that the above identity is `formal' because of the appearance of infinities in the form on $G(x_j,x_j)$. But the quantity that we started out with in the previous display is finite. The resolution of this apparent paradox is that the infinities from $G(x_j,x_j)$ must be formally canceling out with the infinities from the integral of $\hat{\psi} \Delta_{S^2}\hat{\psi}$.) 
Combining with the previous calculations, this gives
\begin{align*}
S(\hat{\rho}) &= -\lambda + \frac{1}{4\pi\beta}\int_{S^2}\hat{\psi}(x)\Delta_{S^2} \hat{\psi}(x) da(x) -\frac{2}{\beta}\sum_{j=1}^k \talpha_j \hat{\psi}(x_j) \\
&\qquad \qquad - \frac{2}{\beta} \sum_{1\le j,j'\le k} \talpha_j \talpha_{j'} G(x_j,x_{j'}) -\frac{2c(\hat{\psi})}{\beta}\sum_{j=1}^k \talpha_j.
\end{align*}
Recalling that the limit in Theorem \ref{semilimit} is equal to 
\begin{align*}
&1+\ln \tmu - \ln \beta - i\pi +(1-\ln 4  )\sum_{j=1}^k \frac{\talpha_j^2}{\beta} + \sum_{j=1}^k \frac{\talpha_j (1+\talpha_j)}{\beta} \ln g(\sigma(x_j))\notag \\
&\qquad \qquad -\frac{4}{\beta}\sum_{1\le j<j'\le k}\talpha_j\talpha_{j'}G(x_j, x_{j'}) - S(\hat{\rho}),
\end{align*}
we see from the above calculations that it is formally equal to 
\begin{align}
&1+ \ln \tmu - \ln \beta - i\pi +(1-\ln 4 )\sum_{j=1}^k \frac{\talpha_j^2}{\beta} + \sum_{j=1}^k \frac{\talpha_j (1+\talpha_j)}{\beta} \ln g(\sigma(x_j))\notag \\
&\qquad \qquad + \lambda +\frac{2}{\beta}\sum_{j=1}^k (\talpha_j \hat{\psi}(x_j) +\talpha_j^2 G(x_j,x_j)) \notag \\
&\qquad \qquad \qquad  - \frac{1}{4\pi\beta}\int_{S^2}\hat{\psi}(x)\Delta_{S^2} \hat{\psi}(x) da(x) + \frac{2c(\hat{\psi})}{\beta}(\beta+1).\label{semi}
\end{align}
Recall the value of $c(\hat{\psi})$ from equation \eqref{chatpsi}, and observe that by equation \eqref{hatrhoeq},
\begin{align}\label{1eq}
1 = -\frac{\tmu}{\beta}\int_{S^2} e^{2\hat{\psi}(x)} da(x). 
\end{align}
Using these identities in equation \eqref{semi} to substitute the values of $c(\hat{\psi})$ and $1$, we see that the expression is formally equal to $J(\hat{\psi})/\beta$, as follows. Substituting the value of $c(\hat{\psi})$ from equation \eqref{chatpsi}, we get
\begin{align*}
\frac{2c(\hat{\psi})}{\beta}(\beta+1) &= 2c(\hat{\psi}) + \frac{2c(\hat{\psi})}{\beta}\\
&= 2\biggl(-\frac{\lambda}{2} +\frac{1}{2}\ln \beta + \frac{i\pi}{2} - \frac{1}{2}\ln \tmu\biggr)+  \frac{2c(\hat{\psi})}{\beta}.
\end{align*}
Putting this back into equation \eqref{semi}, we see that the quantity displayed in equation~\eqref{semi} is equal to 
\begin{align*}
&1  -\chi \sum_{j=1}^k \frac{\talpha_j^2}{\beta} + \sum_{j=1}^k \frac{\talpha_j (1+\talpha_j)}{\beta} \ln g(\sigma(x_j))+\frac{2}{\beta}\sum_{j=1}^k (\talpha_j \hat{\psi}(x_j) +\talpha_j^2 G(x_j,x_j)) \notag \\
&\qquad \qquad    - \frac{1}{4\pi\beta}\int_{S^2}\hat{\psi}(x)\Delta_{S^2} \hat{\psi}(x) da(x) +  \frac{2c(\hat{\psi})}{\beta}.
\end{align*}
Next, replacing the leading $1$ above by the right side of equation \eqref{1eq}, we get that the above quantity is equal to 
\begin{align*}
&  -\chi \sum_{j=1}^k \frac{\talpha_j^2}{\beta} + \sum_{j=1}^k \frac{\talpha_j (1+\talpha_j)}{\beta} \ln g(\sigma(x_j)) +\frac{2}{\beta}\sum_{j=1}^k (\talpha_j \hat{\psi}(x_j) +\talpha_j^2 G(x_j,x_j)) \notag \\
&\qquad \qquad     - \frac{1}{4\pi\beta}\int_{S^2}(\hat{\psi}(x)\Delta_{S^2} \hat{\psi}(x) + 4\pi \tmu e^{2\hat{\psi}(x)})da(x) +  \frac{2c(\hat{\psi})}{\beta}.
\end{align*}
But this is just $J(\hat{\psi})/\beta$. 
This completes the proof.
\end{proof}

\section{Acknowledgements}
The author is grateful to Edward Witten for introducing him to this problem during a sabbatical at the Institute for Advanced Study at Princeton, and numerous helpful discussions subsequently. The author is also indebted to the two anonymous referees for numerous helpful remarks. This research was supported in part by NSF grants DMS-2113242 and DMS-2153654.

\bibliographystyle{abbrvnat}

\bibliography{myrefs}

\end{document}